\documentclass[10pt, reqno]{amsart} 
\usepackage[margin=1in]{geometry} 

\usepackage{graphicx}
\usepackage[nice]{nicefrac}
\usepackage{amssymb}
\usepackage{epstopdf}
\usepackage{amsmath,amsthm,amscd,amssymb} 
\usepackage{verbatim}
\allowdisplaybreaks
\usepackage{latexsym}
\usepackage[colorlinks,citecolor=red,pagebackref,hypertexnames=false]{hyperref}
\usepackage{mathtools}
\usepackage{esint}
\usepackage{enumerate}
\usepackage{bbm}
\usepackage{verbatim}
\usepackage{multirow}
 \usepackage[table,xcdraw]{xcolor}

\usepackage{tikz}
\usepackage{cancel}
\usepackage{color}
\usepackage{array}
\usepackage{gensymb}
\usepackage{tabularx}
\usepackage{booktabs}
\usetikzlibrary{fadings}
\usetikzlibrary{patterns}
\usetikzlibrary{shadows.blur}
\usetikzlibrary{shapes}
\usetikzlibrary{decorations.pathreplacing}

\numberwithin{equation}{section}

\theoremstyle{plain}
\newtheorem{theorem}{Theorem}[section]
\newtheorem{lemma}[theorem]{Lemma}
\newtheorem{corollary}[theorem]{Corollary}
\newtheorem{proposition}[theorem]{Proposition}

\newtheorem{question}[theorem]{Question}

\theoremstyle{definition}
\newtheorem{definition}[theorem]{Definition}

\newtheorem{obs}[theorem]{Observation}

\theoremstyle{remark}
\newtheorem{remark}[theorem]{Remark}

\newtheorem{notation}[theorem]{Notation}
\newtheorem{claim}[theorem]{Claim}

\def\BBR {{\mathbb R}}

\newcommand*{\rom}[1]{\expandafter\@slowromancap\romannumeral #1@}

\newcommand{\tops}{\textrm{top}}

\newcommand{\BMO}{{\textrm{BMO}}}
\newcommand{\osc}{{\textrm osc}}
\newcommand{\kernel}{{\textrm Kernel}}
\def\BBD {{\mathbb D}}

\title{Biparameter $\BMO$ under the action of a rotation}

\author{Fr\'ed\'eric Bernicot}
\author{Yujia Zhai}

\address{CNRS-Universit\'e de Nantes, Laboratoire Jean Leray, 2, rue de la Houssini\`ere,44322 Nantes cedex 3, France}
\email{frederic.bernicot@univ-nantes.fr}
\email{yujia.zhai@univ-nantes.fr}

\thanks{Both authors are supported by ERC project FAnFArE no. 637510. Y. Zhai's research is also supported by the region Pays de la Loire.}

\begin{document}

\begin{abstract} In this work, we aim to study the action of composing by a rotation on the biparameter $\BMO$ space in $\BBR^2$. This $\BMO$ space is not preserved by a rotation since it relies on the structure of axis-parallel rectangles. We will quantify this fact by interpolation inequalities.  One straightforward application of the interpolation inequalities is a boundedness property of directional Hilbert transforms.
\end{abstract}

\maketitle

\section{introduction}

On the Euclidean space, the $\BMO$ spaces (for ``bounded mean oscillation'') are very important functional spaces and are of considerable interest in harmonic analysis. Indeed this concept is a very convenient substitute of the $L^\infty$ space (which is not easy to handle with when working on the frequency side). The $\BMO$ spaces are larger than the $L^\infty$ but not too much to constitute a suitable candidate as an end-point of the Lebesgue spaces $L^p$ when $p \to \infty$. In particular, we can cite the well-known interpolation results, boundedness of singular operators, duality with Hardy space, link with Carleson measures, John-Nirenberg properties, boundedness of commutators and so on.

\medskip 
Now on the product space $\BBR^2 = \BBR \times \BBR$, two kinds of $\BMO$ spaces can be considered, the one-parameter one (built on the collection of balls or squares) and the  biparameter one (built on the collection of rectangles). Let us first review some definitions, in order to introduce the $\BMO$ spaces and their dyadic versions.

\medskip

We first recall that a dyadic interval of $\BBR$ is an interval of the form $2^k[j,j+1)$ for some $j,k\in \mathbb{Z}$. Let $\BBD=\BBD^0$ denote the collection of all the dyadic intervals and we refer to it as the dyadic grid of $\BBR$. For some $\delta\in\BBR$, we denote by $\BBD^\delta$ the shifted grid with the parameter $\delta$:
$$ \BBD:=\{2^k[j,j+1),\ j,k\in\mathbb{Z}\} \qquad \textrm{and} \qquad \BBD^\delta:=\{\delta_I+I,\ I\in \BBD\},$$
where $\delta_I=\delta$ if $|I|\leq 1$, $\delta_I=\delta+ \frac{1}{3}(|I|-1)$ if $|I|=2^k$ with $k> 0$ even and
$\delta_I=\delta+ \frac{1}{3}(2|I|-1)$ if $|I|=2^k$ with $k> 0$ odd. The way of shifting (differently according to the scale) allows us to keep the nested property of $\BBD^\delta$: each interval is the exact disjoint union of two children.

\medskip

The one-parameter $\BMO(\BBR)$ on $\BBR$ is defined as follows: a function $f\in L^2_{loc}(\BBR)$ belongs to $\BMO(\BBR)$ if
\begin{equation} \|f\|_{\BMO(\BBR)}:= \sup_{I \textrm{ interval}} \osc_I(f) <\infty \label{eq:def} \end{equation}
where we take the supremum over all the bounded intervals $I\subset \BBR$ and for such interval $\osc_I$ denotes the $L^2$-oscillation defined by
$$ \osc_I(f) := \bigg(\frac{1}{|I|}\int_{I}\Big|f(y) - \fint_I f \Big|^2 dy\bigg)^{\frac{1}{2}}.$$
The quantity $\|\cdot \|_{BMO(\BBR)}$ is a pseudo-norm. In order to consider $\BMO(\BBR)$ as a normed space, one has to quotient it by the ``{\it kernel}'' (the functions $f$ for which the pseudo norm $\|f\|_{\BMO}$ equals to $0$) which is 
$$ \kernel(\BMO(\BBR)) = \{ \quad \textrm{constant functions} \quad \}.$$

\medskip

According to a dyadic grid ($\BBD^0$ or $\BBD^\delta$), we may also define the dyadic $\BMO^{0}$, $\BMO^{\delta}$ respectively with the following pseudo-norm
$$\|f\|_{\BMO^\delta(\BBR)}:= \sup_{I \in \BBD^\delta} \osc_I(f) $$
which is also characterized by a Carleson criterion involving the Haar wavelets
$$\|f\|_{\BMO^\delta(\BBR)} \simeq \sup_{I \in \BBD^\delta} \bigg(\frac{1}{|I|} \sum_{\substack{J\in \BBD^\delta \\ J\subset I}} |\langle f, h_J\rangle |^2 \bigg)^{1/2}.$$
We recall that for an interval $I=[a,b) \subset \BBR$, we define its two children
$$ I^{left}:=[a,\frac{a+b}{2}) \qquad \textrm{and} \qquad I^{right}:=[\frac{a+b}{2},b)$$
and its corresponding ($L^2$-normalized) Haar wavelet is given by
$$ h_I(x):= |I|^{-1/2} \big( \mathbbm{1}_{I^{left}}(x) - \mathbbm{1}_{I^{right}}(x) \big),$$
where $\mathbbm{1}_J$ is the characteristic function of $J \subseteq \mathbb{R}$.

\medskip

Then we know from \cite{Mei,LiPipherWard} that for some $\delta$, we have
$$ \BMO(\BBR) = \BMO^0(\BBR) \cap \BMO^\delta(\BBR).$$
Sometimes for some combinatorial argument, it is convenient to work and take advantage of a dyadic grid, sometimes it is more convenient to work with the whole collection of intervals. This results (and others in \cite[Page 197]{BookHNVW} with $\delta=1/3,2/3$) show that we can combine these two aspects.

\bigskip

Now if one looks at $\BBR^2$ as the Euclidean space of dimension $2$, we can automatically extend all this approach by considering oscillations and Haar wavelets on the collection of dyadic cubes and by that way, one considers the so-called one parameter $\BMO(\BBR^2)$ space and the ``one-parameter" highlights the fact that the elementary geometrical objects - dyadic cubes - have only $1$ scale.

\medskip

However, if we see $\BBR^2=\BBR \times \BBR$ as a product space, then we move to the notion and the setting of the biparameter world: the elementary geometrical object is now a product $I\times J$ of two dyadic intervals, thus a rectangle whose scale is given by two parameters. This so-called biparameter $\BMO(\BBR^2)$ space has been introduced by Chang and Fefferman in \cite{ChangFefferman} and appears for the study of biparameter singular operators. A few things are known about this space, including interpolation \cite{Lin}, link with Carleson measures \cite{Chang}, John-Nirenberg inequality \cite{ChangFefferman}, duality with Hardy spaces on product domains \cite{ChangFefferman} and boundedness of commutators \cite{HPW}. We refer the reader to \cite{Fefferman} for more details about the biparameter theory. In this context, there is the rectangular $\BMO$ space, whose definition seems to be a natural extension of \eqref{eq:def}, taking now the supremum of all the oscillations over the whole collection of rectangles. Unfortunately, the rectangular $\BMO$ space turns out not to be the best candidate for problems in time-frequency analysis, such as the failure of John-Nirenberg properties (see \cite{FergusonSadowski} for details). Hence we aim to focus on a subspace of the rectangular $\BMO$ space, which does not have a simple definition in terms of oscillation as in \eqref{eq:def} because one wants to consider some deeper cancellation. We call it biparameter $\BMO$ space, denoted by $\BMO(\mathbb{R}^2)$ and defined in terms of ``Carleson criterion'': a locally $L^2$ function $F$ of $\BBR^2$ is said to be in $\BMO(\BBR^2)$ if 
\begin{equation} \sup_{\Omega} \left( \frac{1}{|\Omega|} \iint_{T(\Omega)} |F \ast (\psi_{s} \otimes \psi_{t})(x,y)|^2 \, \frac{dsdt dxdy}{st} \right)^{1/2} <\infty,
\label{eq:defBMO}
\end{equation}
where we take the supremum over all open subset $\Omega$ with finite measure and where $\psi$ is a smooth function with a vanishing integral and $\psi_s$ is the $L^1$ dilated function, namely $s^{-1} \psi(s^{-1} \cdot)$, and $T(\Omega)$ is the Carleson tent
$$ T(\Omega):= \{(s,t,x,y),\ [x-s,x+s] \times [y-t,y+t] \subset \Omega\}.$$
Taking the supremum over all the open subsets, and not only the collection of rectangles, brings a richer structure on the space and also some technical difficulties since the open subsets can be arbitrary and so don't have any geometrical properties.
In this biparameter setting, we also have a kernel whose structure is more delicate:
$$ \kernel(\BMO(\BBR^2)) = \{ \quad F:=(x,y) \mapsto  \phi_x + \phi_y \quad \}$$
where
$\phi_x$ is a function depending only on $x$ and $\phi_y$ a function depending only on $y$.

The bridge between dyadic and non-dyadic $\BMO$ spaces has also been studied in the biparameter setting. We first define the dyadic $\BMO(\BBR^2)$ spaces as follows: for $a,b\in\{0,\delta\}$, we define the dyadic biparameter space $\BMO^{a,b}(\BBR^2)$ by 
\begin{equation*}
\BMO^{a,b}(\BBR^2) := \{F \in L^2_{loc}(\BBR^2): \|F\|_{\BMO^{a,b}(\BBR^2)} < \infty  \},
\end{equation*}
where
$$\|F\|_{\BMO^{a,b}(\BBR^2)} = \sup_{\Omega} \bigg(\frac{1}{|\Omega|} \sum_{\substack{I,J\in \BBD^a\times \BBD^b \\ I\times J\subset \Omega}} |\langle F, h_I \otimes h_J\rangle |^2 \bigg)^{1/2}$$
based on the collection of dyadic rectangles of $\BBD^{a,b}:=\BBD^a \times \BBD^b$ with the supremum over all the open subset $\Omega \subseteq \BBR^2$ of finite measure. Then we know from \cite{LiPipherWard} that for some $\delta\in(0,1/2)$, we have
$$ \BMO(\BBR^2) = \bigcap_{a,b\in\{0,\delta\}} \BMO^{a,b}(\BBR^2).$$
In all the sequel, we fix such a parameter $\delta$ and we can see the previous formulation as a definition of the full $\BMO(\BBR^2)$ space.
 
\medskip
 
With the understanding of the appropriate definitions of one-parameter and biparameter BMO norms, we would like to study some elementary properties of both norms, namely the behavior of BMO norms under compositions. In the one-parameter setting, it is very easy to see that bi-Lipschitz maps are ``almost preserving'' the structure and geometry of the collection of dyadic squares and so such maps are preserving the $\BMO$ by composition.  
Such question has indeed been studied and was motivated by studying the Euler equation with a $\BMO$-type vorticity in \cite{BernicotKeraani,BernicotKeraani2}. 
Keeping in mind this motivation, to study the dynamic of vortices in fluids mechanics, the Euler equation may involve some operators of type $R_1 \circ R_2$, where $R_i$ is the Riesz transform in $\BBR^2$ along the $i^{th}$ coordinate for $i=1,2$. These operators $R_i$ are also called directional Riesz (or Hilbert) transform along the canonical vector $e_i$. They also naturaly appear in the context of characterizing $\BMO$ with commutators \cite{FL,OPS}, or also for example if one wants to decompose a rough homogeneous singular integral operator through the method of rotations.  Because they are acting on only one direction of $\BBR^2$, they are not bounded in the one-parameter $\BMO$ space, which cannot take into account such anisotropic property. However because $R_1=H \otimes \textrm{Id}$ and $R_2=\textrm{Id}\otimes H$ are both the tensor product of the identity and the one-dimensional Hilbert transform $H$, it is easy to see that both $R_1$ and $R_2$ are bounded on the biparameter $\BMO$ space. It is then natural to ask what happens for a more general directional Riesz transform, given by a direction which is not one of the two canonical ones, but instead the composition of $R_1$ or $R_2$ through a rotation. This question would be closely related to the biparameter BMO norm under the compositions with measure-preserving maps and particularly rotations.

\medskip 

From the definition \eqref{eq:defBMO} of the $\BMO(\BBR^2)$ space, it is naturally invariant by translation and by biparameter dilations: for $\tau\in \BBR^2$ and $\delta_1,\delta_2>0$, if $F\in \BMO(\BBR^2)$ then
$$ (x,y) \mapsto F( (x,y) -\tau) \in \BMO(\BBR^2) \qquad \textrm{and} \qquad (x,y) \mapsto F(\delta_1 x, \delta_2 y) \in \BMO(\BBR^2) $$
and their biparameter $\BMO$ norm is equal to the biparameter $\BMO$ norm of $F$. After these two operations, one can ask a very natural question about how the biparameter BMO norm behaves under the composition by a generic measure-preserving map.

\medskip 

\begin{question} \label{conj_1}
For any measure-preserving map $\phi$ of $\BBR^2$, does there exist a constant $C(\phi)$ such that any $F \in \BMO(\BBR^2)$ satisfies the following inequality: 
\begin{equation} \label{conj_1_ineq}
\|F \circ \phi\|_{\BMO(\BBR^2)} \leq C(\phi) \|F\|_{\BMO(\BBR^2)} \ ?
\end{equation}
\end{question}

Instead of looking at a generic measure-preserving map, we consider a special class of measure-preserving maps - rotations. Different from translations and dilations, rotations do not preserve the kernel of $\BMO(\mathbb{R}^2)$, which provides a negative answer to Question \ref{conj_1}. A specific counterexample will be provided in Section \ref{subsec:counterexample1} to illustrate the behavior of the biparameter BMO norm under the action of rotations. 

Inspired by the observation on the effect of rotations on the biparameter BMO norm, we would like to understand whether it is possible to bound the left hand side of (\ref{conj_1_ineq}) by the biparameter $\BMO$ norm interpolated with another stronger norm.  
Our first attempt is to assume some $L^p$ integrability on $F$ with $p<\infty$, which excludes the possibility that $F$ lies in the kernel of $\BMO(\mathbb{R}^2)$ (see Remark \ref{haar_composition}). If we further impose some regularity on $F$ (see Section \ref{sec:motivation} for more motivations), we know that for $sp>2$, the Sobolev space $W^{s,p}(\BBR^2)$ is strictly and continuously included in $L^\infty$ and thus in $\BMO(\mathbb{R}^2)$. We thus formulate the second conjecture as follows.

\begin{question}\label{conj_2}
For any rotation map $\phi^{\theta}$ with the angle $\theta \in [0, 2\pi)$ and for any $p>2$, does there exist some constants $C(\phi)$ and $\epsilon(\phi,p)<1$ such that any $F \in \BMO(\BBR^2) \cap W^{1,p}(\BBR^2)$ satisfies the following inequality:
\begin{equation}
\|F \circ \phi\|_{\BMO(\BBR^2)} \leq C(\phi) \|F\|_{\BMO(\BBR^2)} + \epsilon(\phi,p)\|F\|_{W^{1,p}(\BBR^2)}\ ?
\end{equation}
\end{question}
In this question, we allow some small perturbation taking into account the regularity of the initial function $F$. From a naive point of view, since the identity map $\phi^{\theta}=\textrm{Id}$ or equivalently $\theta = 0$ preserves the $\BMO(\BBR^2)$ space, we would like to positively answer this question with the fact that the small constant satisfies
$\displaystyle \lim_{\phi\rightarrow \textrm{Id}} \epsilon(\phi,p) = 0$.
However, it turns out that Question \ref{conj_2} is too ideal to be true. We are going to focus on a counterexample to disprove such inequality in Section \ref{subsec:counterexample2}. 

The computations for the counterexample suggest that in order to have a small coefficient in front of the Sobolev norm, we cannot expect a uniformly bounded coefficient in front of the biparameter $\BMO$ norm.
We are going to prove the following main positive result.
\begin{theorem} \label{main_thm_pos}
Suppose that $\phi^{\theta}$ is a rotation of $\BBR^2$ and $p \in (2,\infty)$ with $s>\frac{2}{p}$. Then for every $\gamma \in (0, \delta)$ with $\delta:= \min(s- \frac{2}{p}, \frac{1}{p})>0$, there exists a constant $C$ such that for every $\epsilon\in(0,1)$ and every function $F \in \BMO(\BBR^2) \cap W^{s,p}(\BBR^2)$, we have
\begin{equation}
\|F \circ \phi\|_{\BMO(\BBR^2)} \leq C \epsilon^{-1}(1+ \log \epsilon^{-1}) \|F\|_{\BMO(\BBR^2)} + C \epsilon^{\frac{\gamma}{2}}\|F\|_{W^{s,p}(\BBR^2)}. \label{main_pos_bound}
\end{equation}
\end{theorem}

Let us first observe that by the Sobolev embedding $W^{s,p}(\BBR^2) \hookrightarrow L^\infty \hookrightarrow \BMO(\BBR^2)$ for $sp >2$, the estimate \eqref{main_pos_bound} is trivially true for $\epsilon \simeq 1$ whereas the case $\epsilon \ll 1$ is nontrivial and remains to be proved.

To the best our knowledge, there is no abstract results or black boxes which could help us to understand such phenomena. As a consequence, we need to prove such inequality by hands, trying to keep as much as we can the information about cancellation and orthogonality in our decompositions and localizations estimates.

\medskip

By optimizing in $\epsilon$, we deduce the following corollary of Theorem \ref{main_thm_pos}.
 
\begin{theorem} \label{thm:ff}
Suppose that $\phi^\theta$ is a rotation of $\BBR^2$ and $p\in(2,\infty)$ with $s>\frac{2}{p}$. Then for every $\alpha \in(0,\frac{\delta}{2+\delta}]$ with $\delta:= \min(s- \frac{2}{p}, \frac{1}{p})>0$, there is a constant $C$ such that for every function $F \in \BMO(\BBR^2) \cap W^{s,p}(\BBR^2)$, we have
\begin{equation*}
\|F \circ \phi\|_{\BMO(\BBR^2)} \leq C \|F\|_{\BMO(\BBR^2)}^\alpha \|F\|_{W^{s,p}(\BBR^2)}^{1-\alpha}. 
\end{equation*}
\end{theorem}

\medskip

As an application, we obtain some inequalities for the directional Hilbert transform: let $v$ be a unit vector of $\BBR^2$, then the directional Hilbert transform along this direction is defined by
$$ H_{v}(F)(x,y) := p.v. \int_{-\infty}^\infty F((x,y)-t v) \, \frac{dt}{t}.$$
Moreover for a kernel $\Omega \in L^1({\mathbb S}^1)$ with vanishing integral, the rough homogeneous integral operator $T_\Omega$ is defined as
$$ T_\Omega(F)(x,y) := p.v. \int F((x,y)-u) \Omega(\frac{u}{|u|}) \, \frac{du}{|u|^2} = \frac{1}{2} \int_{{\mathbb S}^1} H_{v}(F)(x,y) dv$$
(see \cite[Section 5.2]{Grafakos} for example).

By observing that for the rotation $\phi$ such that $\phi(e_1)=v$ (where $e_1$ is the first canonical vector of $\BBR^2$), we have
$$ H_v(F) = H_{e_1}[F \circ \phi] \circ \phi^{-1},$$
and iterating twice Theorem \ref{thm:ff} (with the fact that $H_{e_1}$ is bounded on $\BMO(\mathbb{R}^2)$ and on Sobolev spaces), we deduce the following:
\begin{corollary} 
Fix $p\in(2,\infty)$ with $s>\frac{2}{p}$ and consider any $\alpha \in(0,\frac{\delta}{2+\delta}]$ with $\delta:= \min(s- \frac{2}{p}, \frac{1}{p})>0$, there is a constant $C$ such that for every function $F \in \BMO(\BBR^2) \cap W^{s,p}(\BBR^2)$ and every unit vector $v$, we have
\begin{equation*}
\|H_v(F) \|_{\BMO(\BBR^2)} \leq C \|F\|_{\BMO(\BBR^2)}^{\alpha^2} \|F\|_{W^{s,p}(\BBR^2)}^{1-\alpha^2}. 
\end{equation*}
For $\Omega \in L^1({\mathbb S}^1)$ with $\int \Omega(v) dv=0$, then
\begin{equation*}
\|T_{\Omega}(F) \|_{\BMO(\BBR^2)} \leq C \|\Omega\|_{L^1({\mathbb S}^1)}\|F\|_{\BMO(\BBR^2)}^{\alpha^2} \|F\|_{W^{s,p}(\BBR^2)}^{1-\alpha^2}. 
\end{equation*}
\end{corollary}

\medskip

The paper is organized as follows. After some preliminaries in Section \ref{sec:pre}, we first disprove Questions \ref{conj_1} and \ref{conj_2} in Section \ref{sec:questions}. Section \ref{sec:motivation} motivates the statement of Theorem \ref{main_thm_pos}, which is then proved in Sections \ref{sec:main0} and \ref{sec:main}.

\section{Preliminaries} \label{sec:pre}

\begin{notation} For $I$ an interval of center $c_I$ and radius $r_I$  and a factor $\lambda>0$, $\lambda I$ denotes the concentric interval of center $c_I$ and radius $\lambda r_I$. Similarly if $R=I \times J$ is a rectangle then $\lambda R$ represents the rectangle $\lambda I \times \lambda J$ and $|R|=|I||J|$ will denote its measure.

The strong maximal function will be denoted $MM$ (to emphasize the biparameter structure) and is defined for a function $f\in L^1_{loc}(\BBR^2)$ by
\begin{align*}
MM(f)(x,y) & := \sup_{R: (x,y) \in R} \ \frac{1}{|R|}\int_{R} |f(t,s)| \, dtds, 
\end{align*}
where we take the supremum of the averages of $f$, on every rectangle $R$ containing the point $(x,y)\in \BBR^2$.
\end{notation}

\begin{definition}[Eccentricity]\label{ecc}
For a rectangle $R$ with $|R| = R_1\times R_2$, where $R_1$ and $R_2$ denote the horizontal and vertical side lengths of $R$ respectively, we define the \textbf{eccentricity} of $R$ by $\varepsilon_{R}:=\frac{R_2}{R_1}$.
\end{definition}

\begin{lemma} \label{lemma:regularity} Uniformly in $\alpha\in\{0,\delta\}^2$, fix a function $F\in W^{s,p}(\BBR^2)$ for some $p\in(2,\infty)$ and $s\in (\frac{2}{p},1)$. Then for every integers $k_1,k_2$ and every exponents $\gamma_1,\gamma_2\in (-\delta,\delta)$ with $\delta:=\min\{\frac{1}{p}, s-\frac{2}{p}\}>0$ we have
$$ \left(\sum_{\substack{R\in \BBD^{\alpha} \\ |R|=2^{k_1} \times 2^{k_2}}} |\langle F, h_R\rangle |^p \right)^{1/p} \lesssim 2^{k_1 (\frac{1}{2}+\gamma_1)} 2^{k_2 (\frac{1}{2}+\gamma_2)} \|F\|_{W^{s,p}(\BBR^2)}.$$
\end{lemma}

\begin{proof}  This estimate can be obtained by interpolation. Aiming that, it suffices us to prove the three end-points $p=2$ or $\infty$ and $s=\frac{2}{p}$ or $s=1$.

\begin{itemize}
\item For $p=\infty,s=0$, this is direct with $\gamma_1=\gamma_2=0$ since for every rectangle $R$ of dimensions $2^{k_1}\times  2^{k_2}$ we have
$$ |\langle F, h_R\rangle | \lesssim \|F\|_\infty \|h_R\|_1 \lesssim \|F\|_{\infty} |R|^{1/2}.$$
\item For $p=\infty,s=1$, so $F$ is supposed to be Lipschitz. If $k_i\leq 0$, we perform an integration by parts along the coordinate $i$, and we obtain
$$ |\langle F, h_R\rangle | \lesssim \|F\|_{W^{1,\infty}} \|H_R\|_1 \lesssim \|F\|_{W^{1,\infty}} |R|^{1/2} 2^{k_1\gamma_1} 2^{k_2 \gamma_2},$$
where $H_R$ is the primitive function of $h_R$ along the coordinates for which $k_i\leq 0$ and where we chose $\gamma_i=1$ if $k_i\leq 0$ and else $\gamma_i=0$.
\item For $p=2,s=1$, this is direct by Plancherel inequality and integration by parts with $\gamma_1=\gamma_2=\frac{1}{2}$.
\end{itemize}
\end{proof}

\begin{lemma}[Shifted Dyadic Grids - {\cite[Proposition 2.1]{Mei}}] \label{lemma:grid} For some $\delta>0$, there exists a constant $c=c(\delta)$ such that for any interval $I\subset \BBR$, there exist $\alpha\in\{0,\delta\}$ and $I_0\in \BBD^\alpha$ such that
$$ I\subset I_0 \subset c I.$$
For any rectangle $R\subset \BBR^2$, there exist $\alpha\in\{0,\delta\}^2$ and $R_0 \in \BBD^\alpha$ such that
$$ R\subset R_0 \subset c R.$$
\end{lemma}

For notational convenience, we choose the constant $c=2$, which is always possible by considering a larger collection of shifted dyadic grids. This numerical constant will mainly appear in \eqref{choice_grid} and does not play any important role.

\section{Counter-examples for Questions \ref{conj_1} and \ref{conj_2}. }
\label{sec:questions}

\subsection{Counter-example for Question \ref{conj_1}} 
\label{subsec:counterexample1}

We want here to explain why the answer of Question \ref{conj_1} for rotations is negative.

Consider the function $F(x,y):=\mathbbm{1}_{[0,1]}(x)$. Since $F$ is constant along the $y$-variable, it belongs to the kernel of $\BMO(\BBR^2)$. It suffices us to check that with $\phi:=\phi^\theta$ the rotation of angle $\theta\in(0,2\pi)\setminus\{\frac{\pi}{2},\pi,\frac{3\pi}{2}\}$ then $F\circ \phi$ is not vanishing in $\BMO(\BBR^2)$.

For simplicity, consider the case where $\theta\in(0,\frac{\pi}{2})$ and consider a dyadic rectangle $R:=I\times J$ of this type: the intervals $I=2^k(i,i+1)$ and $J=2^\ell(j,j+1)$ satisfying
\begin{itemize}
\item the eccentricity satisfies $\frac{2^\ell}{2^k} \simeq \tan(\theta)$ and both $2^k,2^\ell \ll 1$;
\item there exists a point $(x,y)$ such that $y=-x\tan(\theta) + \sin(\theta)^{-1}$ and
$$ 2^k(i+\frac{1}{4}) \leq x \leq 2^k(i+\frac{1}{2}) \qquad \textrm{and} \qquad 2^\ell(j+\frac{1}{4}) \leq y \leq 2^\ell(j+\frac{1}{2}).$$
\end{itemize}
The function $F\circ \phi$ is supported and equal to $1$ on the strip 
$$ S:=\{(x,y),\ -x \tan(\theta) \leq y \leq -x \tan(\theta) + \sin(\theta)^{-1}\}.$$
So by the picture (see Figure \ref{figure1}), one easily deduces that
\begin{equation} 2^k(i,i+1/4) \times 2^\ell (j,j+1/4) \subset S \cap R \subset 2^k(i,i+1/2) \times 2^\ell (j,j+1/2)  \label{ScapR}
\end{equation}
which is one quarter of the initial rectangle $R$.

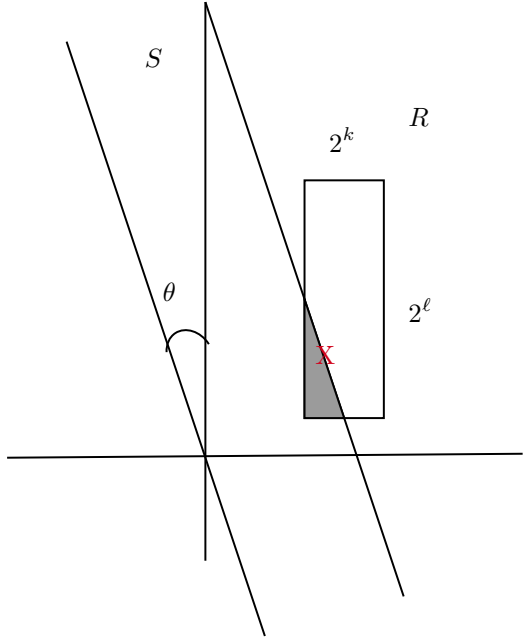
\begin{figure}
\centering

\tikzset{every picture/.style={line width=0.75pt}} 

\begin{tikzpicture}[x=0.75pt,y=0.75pt,yscale=-1,xscale=1]

\draw    (180,20) -- (180,302) ;
\draw    (80,250) -- (340,248) ;
\draw    (110,40) -- (210,340) ;
\draw    (180,20) -- (280,320) ;
\draw   (230,110) -- (270,110) -- (270,230) -- (230,230) -- cycle ;
\draw  [draw opacity=0] (160.42,196.72) .. controls (160.25,192.76) and (161.77,189.13) .. (164.85,187.11) .. controls (169.99,183.76) and (177.51,186.19) .. (181.65,192.53) .. controls (181.69,192.59) and (181.72,192.64) .. (181.75,192.69) -- (172.35,198.61) -- cycle ; \draw   (160.42,196.72) .. controls (160.25,192.76) and (161.77,189.13) .. (164.85,187.11) .. controls (169.99,183.76) and (177.51,186.19) .. (181.65,192.53) .. controls (181.69,192.59) and (181.72,192.64) .. (181.75,192.69) ;
\draw  [fill={rgb, 255:red, 155; green, 155; blue, 155 }  ,fill opacity=1 ] (230,170) -- (250,230) -- (230,230) -- cycle ;

\draw (148,42) node [anchor=north west][inner sep=0.75pt]   [align=left] {$\displaystyle S$};
\draw (157,160) node [anchor=north west][inner sep=0.75pt]   [align=left] {$\displaystyle \theta $};
\draw (281,72) node [anchor=north west][inner sep=0.75pt]   [align=left] {$\displaystyle R$};
\draw (281,168) node [anchor=north west][inner sep=0.75pt]   [align=left] {$\displaystyle 2^{\ell }$};
\draw (241,82) node [anchor=north west][inner sep=0.75pt]   [align=left] {$\displaystyle 2^{k}$};
\draw (234,192) node [anchor=north west][inner sep=0.75pt]  [color={rgb, 255:red, 208; green, 2; blue, 27 }  ,opacity=1 ] [align=left] {X};

\end{tikzpicture}

\caption{Position of $S$ and $R$, the point $(x,y)$ is located at {\color{red} X} and here, $R$ has an eccentricity exactly equals to $\tan(\theta)$.} \label{figure1}
\end{figure}

 As a consequence, $F\circ \phi$ and $h_{R}$ are both constant on the support of $(F\circ \phi)\cdot h_R$ and so with \eqref{ScapR},
$$ |\langle F \circ \phi, h_R\rangle| = \frac{|S\cap R|}{|R|} \geq \frac{1}{16}.$$
This yields in particular a lower bound
$$ \frac{1}{16}\leq \|F\circ \phi\|_{\BMO(\BBR^2)}$$
which, combined with $\|F\|_{\BMO(\BBR^2)}=0$, gives a negative answer to Question \ref{conj_1}.

\bigskip

\begin{remark} \label{rem:kernel}
Indeed we see by this counterexample that for a function $\phi$ to satisfy the inequality (\ref{conj_1_ineq}) in Question \ref{conj_1}, a necessary condition is that $\phi$ preserves {\it the kernel} of $\BMO(\BBR^2)$.

Let
\begin{equation*}
F:=(x,y) \mapsto f_x + g_y
\end{equation*}
a function in the kernel.
Hence $F=0$ in $\BMO(\BBR^2)$ so if we have a map $\phi$ which answers positively to Question \ref{conj_1} then $F \circ \phi$ has also to be in the kernel of $\BMO(\BBR^2)$.

If we look for $\phi$ a rotation, then there is only the four 'trivial' possibilities of rotations of angle $0,\frac{\pi}{2},\pi,\frac{3\pi}{2}$, which are easily preserving the $\BMO$ norm (since the image of a rectangle is again a rectangle with sides parallel to the axis). Observe that the previous counterexample requires $|\tan(\theta)|<\infty$ and $\sin(\theta)\neq 0$, which exactly excludes these trivial situations.
\end{remark}

\begin{remark}\label{haar_composition}
This counterexample not only proves the failure of preservation of biparameter $\BMO$ norm under rotations, but also suggests extra conditions one needs to impose on $F$. One reasonable condition will be to impose some integrability $F \in L^p$ for some $p<\infty$ to exclude the functions in the kernel.

Indeed one first notices that for any $F \in \BMO$, $F$ can be written as 
\begin{equation}
F = \sum_{R \in \mathbb{D}}\langle F, h_R \rangle h_R + f_x + g_y.
\end{equation}
Suppose that $F \in L^p$ for some $p\in(1,\infty)$. Then by Khincthine's inequality and the square function estimate, 
\begin{align}
\left\|\sum_{R \in \mathbb{D}}\langle F, h_R \rangle h_R\right\|_p \sim \left\|\left(\sum_{R \in \mathbb{D}} |\langle F, h_R \rangle|^2\right)^{\frac{1}{2}} \right\|_p \lesssim \|F\|_{p}.
\end{align}
Therefore, by triangle's inequality,
\begin{align}
\| f_x + g_y\|_p = \left\| F - \sum_{R \in \mathbb{D}}\langle F, h_R \rangle h_R \right\|_p \leq \|F\|_p + \left\| \sum_{R \in \mathbb{D}}\langle F, h_R \rangle h_R \right\|_p \lesssim \|F\|_p < \infty
\end{align}
which imposes the condition that for almost every $(x,y) \in \mathbb{R}^2$,
\begin{equation}
f_x \equiv g_y \equiv  0.
\end{equation}
In conclusion, assuming that $F\in L^p(\BBR^2)$ for some $p\in(1,\infty)$ will prevent us to deal with the kernel of $\BMO(\BBR^2)$.
\end{remark}

We move the reader to Section \ref{sec:motivation}, for more motivation about introducing a small term with a Sobolev norm.

\subsection{Counter-example for Question \ref{conj_2}}
\label{subsec:counterexample2}
To provide a negative answer to Question \ref{conj_2}, one would like to show its negation as stated in the following theorem.
\begin{theorem}
Suppose that $\phi:=\phi^\theta$ is the rotation in $\BBR^2$ of angle $\theta$. For any constant $C(\theta)$, there exists some $p>2$ and some $F \in \BMO(\BBR^2) \cap W^{1,p}(\BBR^2)$ such that for any $\epsilon(\theta,p)<1$, the following two conditions cannot hold simultaneously.
\begin{enumerate}
\item
 $\displaystyle \lim_{\theta\rightarrow 0}\epsilon(\theta,p) = 0$
 \item
 The inequality 
\begin{equation} \label{conj_2_eq}
\|F \circ \phi^{\theta}\|_{\BMO(\BBR^2)} \leq C(\theta) \|F\|_{\BMO(\BBR^2)} + \epsilon(\theta,p)\|F\|_{W^{1,p}(\BBR^2)}
\end{equation}
is valid.
\end{enumerate}
\end{theorem}
\begin{proof}
Define 
$$
F(x,y) := f(x) g(y),
$$
where 
\begin{enumerate}
\item
$
f(x) := \psi_{[0,1]} (x)
$
denotes a smooth variant of the Haar wavelet on $[1,2]$;
\item
Let $\alpha := \alpha(p)$ denote a positive number depending on $p$. Define
$$
g(y) :=  
\begin{cases}
1 \quad\quad\quad\quad\quad \text{for} \ \ y \in [-1,1] \\
|y|^{-\alpha} \quad\quad\ \ \text{for} \ \ y \in (\infty,-1) \cup (1,\infty). \\
\end{cases}
$$
\end{enumerate}
Observe that the function $F$ depends on $p$ through the parameter $\alpha$, we will adopt a simplified notation and keep this dependence implicit. We will prove the following three properties of $F$:
\begin{itemize}
\item \textit{Property (A):} For the choice $\alpha(p):=\frac{2}{p}$, there exists a numerical constant $C_1$ (independent of $p$) such that
\begin{equation*} 
\sup_{2\leq p <\infty} \ \|F\|_{W^{1,p}} \leq C_1.
\end{equation*}
\item \textit{Property (B):} For the choice $\alpha(p)=\frac{2}{p}$, there exists a numerical constant $C_2$, independent of $p$ such that for every $p\in[4,\infty)$
\begin{equation*}
\|F\|_{\BMO} \leq C_2 p^{-\frac{1}{4}}.
\end{equation*}
\item \textit{Property (C):} There exists a numerical constant $c>0$, such that for every $\theta\in(0,2\pi)\setminus\{\frac{\pi}{2},\pi,\frac{3\pi}{2}\}$
\begin{equation*} 
c< \|F \circ \phi^{\theta}\|_{\BMO}.
\end{equation*}
\end{itemize}

These three properties will be proved in the next lemmas, we are now going to explain how they imply the desired statement.

Assume that the inequality \eqref{conj_2_eq} is true for this specific function $F$ with the choice $\alpha(p)=\frac{2}{p}$, for the sake of contradiction. By combining \textit{Properties} A, B and C, one deduces that
\begin{equation} \label{conj_2_eq_final}
c < \|F\circ \phi^{\theta}\|_{\BMO} \leq C(\theta) C_2 p^{-\frac{1}{4}} + C_1 \epsilon(\theta,p).
\end{equation}
For any $C(\theta)$, as large as we want, one can choose $p>4$ large enough such that 
\begin{equation} \label{ch_1}
C(\theta) C_2 p^{-\frac{1}{4}} < \frac{c}{10}.
\end{equation}
For such exponent $p \in (4,\infty)$, one concludes that
\begin{equation*}
c < \frac{c}{10} + C_1\epsilon(\theta,p)
\end{equation*}
which implies that 
\begin{equation*}
\epsilon(\theta,p) > \frac{9}{10}c \cdot C_1^{-1}
\end{equation*}
and completes the proof of the theorem.
\end{proof}

We now have to detail the proofs of the three properties, which will be done through the next lemmas.

\begin{lemma} \label{lemmaA} For the choice $\alpha(p):=\frac{2}{p}$, there exists a numerical constant $C_1$ (independent of $p$) such that
\begin{equation} \label{prop_1} \tag{$A$}
\sup_{2\leq p <\infty} \ \|F\|_{W^{1,p}} \leq C_1.
\end{equation}
\end{lemma}

\begin{proof} One first notices that
\begin{equation} \label{p_norm}
\|F\|_{L^p(\mathbb{R}^2)} = \|f\|_{p}\|g\|_p
\end{equation}
where for some numerical constant $Q:=\max(\|F\|_2,\|F\|_\infty)$ we have 
$$ \sup_{2\leq p < \infty} \ \|F\|_p \leq Q.$$ 
Moreover
\begin{align} \label{g_norm}
\|g\|_p^p & = 2 + 2\int_{1}^{\infty}y^{-\alpha p}dy \nonumber \\
& = 2 + 2\frac{1}{\alpha p - 1} \nonumber \\
& = 2(1+ \frac{1}{\alpha p -1}),
\end{align}
where we used $\alpha p=2>1$.
By plugging the above estimates into \eqref{p_norm}, one obtains
\begin{equation*}
\|F\|_{L^p} \leq 2^{\frac{1}{p}} Q(1+ \frac{1}{\alpha p -1})^{\frac{1}{p}} = 4^{\frac{1}{p}}Q.
\end{equation*}
Moreover, by applying the computation \eqref{g_norm}, one derives that
\begin{equation*}
\|D_x F\|_p = \|f'\|_p\|g\|_p \leq 2^{\frac{1}{p}} Q' (1+ \frac{1}{\alpha p -1})^{\frac{1}{p}}=4^{\frac{1}{p}} Q',
\end{equation*}
with another numerical constant $Q'$.
Meanwhile,
\begin{equation*}
\|D_y F\|_p = \|f\|_p\|g'\|_p  
\end{equation*}
where
\begin{align}
\|g'\|_p^p  & = 2\alpha^p \int_{1}^{\infty}y^{-\alpha p - p} dy \nonumber \\
& = 2\alpha^p \frac{1}{\alpha p + p -1} \nonumber \\
& = 2 \frac{\alpha^p}{\alpha p + p -1}.
\end{align}
Then
\begin{equation*}
\|D_y F\|_p \leq Q\cdot 2^{\frac{1}{p}}\frac{\alpha}{(\alpha p + p -1)^{\frac{1}{p}}} \leq 2^{\frac{1}{p}}\frac{2}{p(1+p)^{\frac{1}{p}}} Q.
\end{equation*}
In conclusion,
\begin{equation} \label{sob_norm}
\|F\|_{W^{1,p}} \leq (Q+Q')4^{\frac{1}{p}} 
+Q 2^{\frac{1}{p}}\frac{2}{p(1+p)^{\frac{1}{p}}}. 
\end{equation}
Then \eqref{sob_norm} can be simplified as
\begin{equation} \label{sob_norm_bound}
\|F\|_{W^{1,p}} \leq 2(Q+Q')4^{\frac{1}{p}} \big(1+\frac{2}{p(1+p )^{\frac{1}{p}}}\big).
\end{equation}
For any $p \in [2,\infty)$, \eqref{sob_norm_bound} can be estimated by a universal constant independent of $p$. In particular, since $(\frac{1}{p+1})^{\frac{1}{p}} \leq 1$ for any $2 \leq p < \infty$,
\begin{equation}
\|F\|_{W^{1,p}} \leq 8(Q+Q') =: C_1.
\end{equation}
\end{proof}

\begin{lemma} \label{lemmaB} For the choice $\alpha(p)=\frac{2}{p}$, there exists a numerical constant $C_2$, independent of $p$ such that for every $p\in[4,\infty)$
\begin{equation} \label{prop_2} \tag{$B$}
\|F\|_{\BMO(\BBR^2)} \leq C_2 p^{-\frac{1}{4}}.
\end{equation}
\end{lemma}

\begin{proof}
Estimate \eqref{prop_2} follows immediately from the two following Claims \ref{claim_1} and \ref{claim_2} with the choice of $\alpha = \frac{2}{p}$.
\end{proof}

\begin{claim}\label{claim_1}
\begin{equation}
\|F\|_{{\BMO(\BBR^2)}}  \lesssim \|g\|_{\BMO(\BBR)}.
\end{equation}
\end{claim}
\begin{proof}[Proof of Claim \ref{claim_1}]
By definition of the biparameter $\BMO$ norm, 
$$ \|F\|_{{\BMO}} = \sup_{a,b\in\{0,\delta\}} \ \|F\|_{\BMO^{a,b}(\BBR^2)}.$$
So fix some parameters $a,b\in\{0,\delta\}$, we have
\begin{equation}
\|F\|_{{\BMO^{a,b}}}^2  =\sup_{\Omega} \frac{1}{|\Omega|} \sum_{\substack{I \times J \subseteq \Omega \\ I,J\in \BBD^a \times \BBD^b}} | \langle f \otimes g, h_{I}\otimes h_J \rangle|^2 = \sup_{\Omega} \frac{1}{|\Omega|} \sum_{\substack{I \times J \subseteq \Omega \\ I,J\in \BBD^a \times \BBD^\delta}}| \langle f, h_I \rangle|^2 |\langle g, h_J \rangle|^2,
\end{equation}
where the supremum is taken over all subset $\Omega \subset \BBR^2$.
Since $f$ is a smooth variant of the Haar wavelet on the interval $[1,2]$, one can intuitively guess that some open set $\Omega$ for which the sup is attained has to be of the form 
$$ \Omega = [1,2] \times \Omega_2.$$
More rigorously, by integration by parts, we have
$$ |\langle f, h_I\rangle |\lesssim |I|$$
and is non-vanishing only if $I \subset [0,4]$. For such $I$, let $\Omega_I$ denotes a maximal disjoint union of dyadic intervals in $\mathbb{R}_y$ such that
$$ I \times \Omega_I \subset \Omega.$$
In particular,
$$\displaystyle \Omega_I = \bigsqcup_{K\in K_I} K$$
for some maximal collection $K_I$ of disjoint dyadic intervals (of $\BBD^b$).
We deduce that
\begin{align}
\frac{1}{|\Omega|} \sum_{I \times J \subseteq \Omega}| \langle f, h_I \rangle|^2 |\langle g, h_J \rangle|^2 & \lesssim \frac{1}{|\Omega|} \sum_{I \subset [0,4]} |I|^2  \sum_{J \subseteq \Omega_I}|\langle g, h_J \rangle|^2  =    \frac{1}{|\Omega|} \sum_{\substack{I \subset [0,4] \\ I\in \BBD^a}} |I|^2 \sum_{K\in K_I} \sum_{J \subseteq K}|\langle g, h_J \rangle|^2 \nonumber \\
& \lesssim   \frac{1}{|\Omega|} \sum_{\substack{I\subset [0,4] \\ I\in \BBD^a}} |I|^2 \sum_{K\in K_I}|K| \|g\|_{\BMO}^2  \lesssim   \|g\|_{\BMO}^2 \frac{1}{|\Omega|} \sum_{I\subset [0,4]} |I|^2 |\Omega_I|  \nonumber \\
& \lesssim   \|g\|_{\BMO}^2 \sum_{2^k \leq 4} 2^k \frac{1}{|\Omega|} \sum_{I,\ |I|=2^k} |I| |\Omega_I| \lesssim \|g\|_{\BMO}^2 \sum_{2^k \leq 4} 2^k  \nonumber \\
& \lesssim   \|g\|_{\BMO}^2,
\end{align}
where the last equality follows from the fact that for a fixed scale $k$, then $(I\times \Omega_I)_{I, |I|=2^k}$ is a disjoint collection of subset of $\Omega$ so
$$\sum_{I,\ |I|=2^k} |I| |\Omega_I| \leq |\Omega|.$$
Indeed as initially guessed, the main contribution is obtained for the scale $1$ which corresponds to the situation where the considered subset $\Omega$ is of the type $\Omega = [1,2] \times \Omega_{[1,2]}$.
\end{proof}

\begin{claim} \label{claim_2}  For $\alpha\in(0,1/2)$, for every interval $J\subset \BBR$,
\begin{equation}
\\osc_J(g):=\bigg(\frac{1}{|J|}\int_{J}|g - \fint_J g|^2 dy\bigg)^{\frac{1}{2}} \lesssim \alpha^{\frac{1}{4}}
\end{equation}
and so $\|g\|_{\BMO(\BBR)} \lesssim \alpha^{\frac{1}{4}}$.
\end{claim}

\begin{proof}[Proof of Claim \ref{claim_2}]
The proof proceeds as a case-by-case analysis. A detailed argument would be included for the the first case while the second case follows from a similar argument. 
\newline
\noindent
\textbf{Case I. $|J| \leq \alpha^{-\frac{1}{2}}$}
\begin{enumerate}
\item[(a)]
For ``small'' intervals of the form $J = [1, 1 + j]$ where $j := |J|$, one can invoke the Poincar\'e inequality to estimate
\begin{align} 
\osc_{J}(g) & = \bigg(\frac{1}{|J|}\int_{J}|g - \fint_J g|^2 dy \bigg)^{\frac{1}{2}} \lesssim |J| \bigg( \frac{1}{|J|}\int_{1}^{1+j}|\nabla g|^2 dy \bigg)^{\frac{1}{2}} \label{poincare} \\
 & \lesssim \alpha |J|^{\frac{1}{2}} \bigg( \int_{1}^{1+j}y^{-2\alpha - 2} dy \bigg)^{\frac{1}{2}} \lesssim \alpha |J|^{\frac{1}{2}} \bigg( \frac{1-(1+j)^{-2 \alpha - 1} }{2\alpha  + 1} \bigg)^{\frac{1}{2}} \nonumber 
\end{align}
If $j \geq 1$, one can apply the trivial estimate that given $-2\alpha -1 < 0$,
$$ 0 < (1+j)^{-2\alpha -1}  \leq 1$$ 
and thus
\begin{equation} \label{pc_1}
0 \leq 1- (1+j)^{-2\alpha -1} < 1 \leq (2\alpha + 1)j.
\end{equation}
If $j < 1$, one employs the Taylor expansion
\begin{equation*}
(1+ j)^{-2\alpha -1} = 1 + (-2\alpha - 1 )j + O\left(j^{2}\right)
\end{equation*}
which leads to
\begin{equation} \label{pc_2}
1- (1+j)^{-2\alpha -1} \leq  (2\alpha  +1)j.
\end{equation}
Combining (\ref{pc_1}) and (\ref{pc_2}), one obtains the following estimate for \eqref{poincare},
$$
\osc_J(g) \leq \bigg(\frac{1}{|J|}\int_{J}|g - \fint_J g|^2 dy \bigg)^{\frac{1}{2}} \lesssim \alpha |J| \lesssim \alpha^{\frac{1}{2}},
$$
where the last inequality follows from the condition that $j \leq  \alpha^{-\frac{1}{2}}$.
\item[(b)]
For intervals of the form $J = [1-j_1, 1+ j_2]$ for some $j_1 < 2$ or $J = [-1-j_1, 1 + j_2]$ for some $j_1, j_2 \geq 0 $, the argument in (a) can be applied similarly. In particular, for $J = [-1-j_1, 1 + j_2]$,
\begin{align*}
\bigg(\frac{1}{|J|}\int_{J}|g - \fint_J g|^2 dy \bigg)^{\frac{1}{2}} \leq & C|J| \bigg( \frac{1}{|J|}\int_{-1-j_1}^{1+j_2}|\nabla g|^2 dy \bigg)^{\frac{1}{2}} \\
& = C|J| \bigg( \frac{1}{|J|}\int_{-1-j_1}^{-1}|\nabla g|^2 dy + \frac{1}{|J|} \int_{1}^{1+j_2}|\nabla g|^2 dy  \bigg)^{\frac{1}{2}}
\end{align*}
where both integrals can be estimated using the same reasoning as above to derive that
$$
|J| \frac{1}{|J|}\int_{-1-j_1}^{-1}|\nabla g|^2 dy \lesssim \alpha ^2 j_1 \qquad \textrm{and} \qquad 
|J| \frac{1}{|J|}\int_{1}^{1+j_2}|\nabla g|^2 dy \lesssim \alpha^2 j_2.
$$
When combined, one attains that 
\begin{equation*} 
\osc_J(g) = \bigg(\frac{1}{|J|}\int_{J}|g - \fint_J g|^2 dy \bigg)^{\frac{1}{2}} \lesssim \alpha|J| \leq C \alpha^{\frac{1}{2}}.
\end{equation*}
\item[(c)] Since we are estimating the oscillation through the Poincar\'e inequality and that the gradient $\nabla g$ is increasing in $(-\infty,-1)$ vanishing on $(-1,1)$ and decreasing on $(1,\infty)$, one easily deduces that intervals maximizing the oscillation have to be of the form $(a)$ or $(b)$.
\end{enumerate}
Therefore, for $0 < \alpha < \frac{1}{2}$, one can summarize the estimate for Case I as follows
\begin{equation} \label{case_1}
\osc_J(g) \lesssim \alpha^{\frac{1}{2}} \lesssim \alpha^{\frac{1}{4}}.
\end{equation}

\noindent
\textbf{Case II. $|J| > \alpha^{-\frac{1}{2}}$}
\begin{enumerate}
\item[(a)]
For the interval of the form $J = [1, 1+j]$, one first notices that for any $y \in \mathbb{R}$, $|g(y) - g(1+j)| \leq 1$, which gives 
\begin{align*}
\osc_J(g)^2 & \lesssim \frac{1}{|J|}\int_{1}^{1+j}|g(y) - g(1+j)|^2 dy  \\
& \lesssim \frac{1}{|J|}\int_{1}^{1+j}|g(y) - g(1+j)| dy  \\
& \lesssim \frac{1}{|J|}\int_{1}^{1+j} y^{-\alpha} - (1+j)^{-\alpha}dy \\
 & \lesssim \frac{1}{|J|} \left(\frac{(1+j)^{1-\alpha} -1}{1-\alpha}- (1+j)^{-\alpha} j \right)  \\
 & \lesssim \frac{1}{|J|} \left(\frac{(1+j)^{1-\alpha} - 1}{1-\alpha}- (1+j)^{1-\alpha} + (1+j)^{-\alpha}  \right)  \\
 & \lesssim \frac{1}{j}(1+j)^{1-\alpha}(\frac{1}{1-\alpha} - 1) + \frac{1}{j (1+j)^{\alpha}}
\end{align*}
where, according to the fact that $0 < \alpha < 1/2$ and $j > \alpha^{-\frac{1}{2}}$,
\begin{equation*}
\frac{1}{j}(1+j)^{1-\alpha}(\frac{1}{1-\alpha} - 1) \leq \frac{1+j}{j}(\frac{1}{1-\alpha} - 1) \leq 2 \frac{1+j}{j} \alpha \lesssim \alpha \lesssim \sqrt{\alpha}
\end{equation*}
and 
\begin{equation*}
\frac{1}{j (1+j)^{\alpha}} \leq \frac{1}{j} \leq \sqrt{\alpha}.
\end{equation*}
So finally $\osc_J(g)\lesssim \alpha^{\frac{1}{4}}$.
\item[(b)]
Suppose that $J = [1-j_1, 1+ j_2]$ for some $j_1 < 2$ or $J = [-1-j_1, 1+ j_2]$ for some $j_1, j_2 > 0$, the computation in (a) is still valid with slight modifications. For example, consider $J = [-1-j_1, 1+ j_2]$ with $j_1 < j_2$, then $1\leq \alpha^{-\frac{1}{2}} \lesssim |J|\simeq j_2$ and  
\begin{align*}
\osc_J(g) & \lesssim \frac{1}{|J|}\int_{-1-j_1}^{1+j_2}|g(y) - g(1+j_2)| dy \\
& \lesssim \frac{1}{|J|}\left( \int_{-1}^{1} + \int_{-1-j_1}^{-1}  + \int_{1}^{1+j_2} \right)  |g(y) - g(1+j_2)| dy \\
 & \lesssim \frac{1}{|J|}\left( \int_{-1}^{1} + \int_{-1-j_2}^{-1}  + \int_{1}^{1+j_2} \right)  |g(y) - g(1+j_2)| dy \\
 & \lesssim \frac{1}{|J|}\left( \int_{-1}^{1}  |g(y) - g(1+j_2)| dy+ 2 \int_{1}^{1+j_2} |g(y) - g(1+j_2)| dy \right)
\end{align*}
where 
$$
 \frac{1}{|J|} \int_{-1}^{1}  |g(y) - g(1+j_2)| dy = 2\frac{1}{|J|} (1-(1+j_2)^{-\alpha}) \leq \frac{2}{|J|} \leq 2\alpha^{\frac{1}{2}}
$$
and the previous argument generates the estimate
$$
\frac{1}{|J|}\int_{1}^{1+j_2} |g(y) - g(1+j_2)| dy \lesssim \alpha^{\frac{1}{2}}.
$$
\item[(c)] Since this argument relies on a exact computation of the oscillation, it can be automatically extended to other kind of interval $[a,a+j]$ with $a$ different of $1$; and a similar estimate can be proved.
\end{enumerate}
Therefore, for $0 < \alpha < \frac{1}{2}$, one can summarize the estimate for Case II as follows
\begin{equation} \label{case_2}
\osc_J(g) \lesssim \alpha^{\frac{1}{4}}.
\end{equation}

One concludes by combining \eqref{case_1} and \eqref{case_2}, for $\alpha\in(0,1/2)$
$$
\|g\|_{\BMO(\BBR)} \lesssim \alpha^{\frac{1}{4}}.
$$
\end{proof}

\begin{lemma} \label{lemmaC}
There exists a numerical constant $c>0$, such that for every $\theta\in(0,2\pi)\setminus\{\frac{\pi}{2},\pi,\frac{3\pi}{2}\}$
\begin{equation} \label{prop_3} \tag{$C$}
c< \|F \circ \phi^{\theta}\|_{\BMO}.
\end{equation}
\end{lemma}

\begin{proof}
To obtain this lower bound of the $\BMO$ norm, we are considering a specific rectangle $R$, which will be the same as the one already used in Subsection \ref{subsec:counterexample1}. Indeed, there the counterexample phenomena are ``localized'' (in other words they do not exist at infinity). Here we want to consider Sobolev functions $F$, so the function $g$ is decaying a little bit at infinity, but locally around the unit square $[-1,1]^2$, our present function $F$ is 'similar' as the one used in the previous counterexample and for which we already proved a lower bound of the $\BMO$ norm. There are only two minor differences: $f$ is now a smooth version of the Haar wavelet, instead of being the characteristic function of the interval. Since the phenomena are localized only around one extremity of the function, considering the Haar wavelet or the characteristic function does not change anything. The smooth variant can be considered as a small perturbation, hence also does not bring new difficulties.

\medskip

So consider first the 'exact case' where $f=h_{[1,2]}$ is the exact Haar wavelet.
Then by following the example illustrated in Subsection \ref{subsec:counterexample1}, we can consider the same rectangle $R$ (as there) and we see that, with the same notations as there, $F\circ \phi^\theta$ is supported on the strip $S$ and $S\cap R$ is contained in one quarter of $R$ (where $h_R$ is constant) and also in one half of the strip $S$ (where $F\circ \phi^\theta$ is constant). So by that way, $F\circ \phi^\theta$ and $h_R$ are both constant on the support of $(F\circ \phi^\theta) \cdot h_R$. A similar computation yields then
$$ \frac{1}{16} \leq |\langle F\circ \phi^\theta,h_R\rangle|.$$

Now, there is no problem to see that if $f$ is a smooth variant of the Haar wavelet, close enough, then we will still have
$$ \frac{1}{32} \leq |\langle F\circ \phi^\theta,h_R\rangle|,$$
which implies the lower bound for the $\BMO$ norm 
$$
\frac{1}{32} \leq \|F \circ \phi^{\theta} \|_{\BMO}.
$$
\end{proof}

\section{Motivation of Theorem \ref{main_thm_pos}} \label{sec:motivation}
Before delving into the proof of Theorem \ref{main_thm_pos}, we would first motivate why we expect the two norms on the right hand side of (\ref{main_pos_bound}) to appear and why the weight in front of each norm is natural. Suppose that $\psi \in \mathcal{S}(\mathbb{R}^2)$ - a smooth version of the Haar wavelet - is defined as
\begin{equation*}
\psi(x,y) := \psi_1(x) \psi_2(y)
\end{equation*}
where $\psi_1, \psi_2 \in \mathcal{S}(\mathbb{R})$ are compactly supported in the interval $[-\frac{1}{2}, \frac{1}{2}]$ and satisfy 
\begin{align*}
 \int_{\mathbb{R}} \psi_1(x)\, dx  = \int_{\mathbb{R}} \psi_2(y) \, dy=0.
\end{align*}
Let $K$ denote the dyadic rectangle illustrated in Figure \ref{figure:motivation} centered at the origin with $|K|= K_1 \times K_2$. Define
\begin{equation*}
F(x,y) := \frac{1}{|K|^{\frac{1}{2}}}\psi(K_1^{-1}x, K_2^{-1}y)
\end{equation*}
which is a smooth $L^2$-normalized bump function adapted to $K$ and $F \in W^{s,p}$ for $p \geq 2$ and $0 \leq s \leq 1$.
To study the oscillation of $F$ on the dyadic rectangle $K$, a strictly simpler problem would be to understand the interaction of $F$ with one wavelet $h_K$ and in particular, one has the estimate
\begin{equation}
|\langle F, h_K \rangle| \leq \|\psi\|_2.
\end{equation}
This could be viewed as the case when $\theta = 0$ and so $\phi^\theta=\textrm{Id}$. As $|\theta|$ grows positive but still ``small'' enough, one would expect $F\circ \phi^{\theta}$ to be contained in $2\epsilon^{-1}K$ - the $2\epsilon^{-1}$ enlargement of $K$ which could be majorized by the first term in (\ref{main_pos_bound}). This case is drawn as the blue rectangle in Figure \ref{figure:motivation}. Furthermore, the quantification of the ``small'' angle $\theta$ depends on the eccentricity of $K$. If $K_2 \leq \epsilon^{-1} K_1$, then for any $0 \leq \theta < \frac{\pi}{2}$, $F \circ \phi^{\theta}$ would always be compactly supported in $2 \epsilon^{-1}K$. 
If $ K_2 > \epsilon^{-1} K_1$, then $\theta$ needs to satisfy that $ |\sin \theta| \leq \frac{\epsilon^{-1}K_1}{K_2}$. As $|\theta|$ increases such that $|\sin \theta| > \frac{\epsilon^{-1}K_1}{K_2}$, then we would expect $|\langle F \circ \phi^{\theta}, h_K \rangle|$ to be bounded by the second term in (\ref{main_pos_bound}). 
\begin{enumerate}
\item
\textbf{Geometric approach - multilinear Kakeya perspective:}
Heuristically, when $\theta$ gets close to $\frac{\pi}{2}$, we could give the following trivial estimate
\begin{equation} \label{multilinear_Kakeya}
|\langle F \circ \phi^{\theta}, h_K \rangle| \leq |K|^{-\frac{1}{2}}|K|^{-\frac{1}{2}}|K^{\theta}\cap K|,
\end{equation}
where $K^{\theta}$ denote the rotated $K$ depicted as the red rectangle in Figure \ref{figure:motivation}. In this scenario, we encounter the intersection of two transversal rectangles which could be bounded by
\begin{equation*}
|K^{\theta}\cap K| \lesssim K_1^2 |\sin \theta|^{-1}
\end{equation*}
and by recalling the condition that $|\sin \theta| >\frac{\epsilon^{-1}K_1}{K_2} $, we can deduce that
\begin{equation} \label{one_wavelet}
|\langle F \circ \phi^{\theta}, h_K \rangle| \lesssim \epsilon.
\end{equation}
\item
\textbf{Analytic approach:}
In contrast with the first approach which reduces the estimate to a purely geometric problem, we could estimate the left hand side of (\ref{one_wavelet}) using the regularity of $F$ and the integration by parts. We first denote (the $L^2$-normalized dilated verion)
\begin{equation*}
\psi_{K^{\frac{\pi}{2}}}(x,y) := - |K|^{-\frac{1}{2}} \psi(K_2^{-1}x, K_1^{-1}y),
\end{equation*}
then for $\theta = \frac{\pi}{2}$,
\begin{align*}
|\langle F \circ \phi^{\theta}, h_K \rangle| = |\langle \psi_{K^{\frac{\pi}{2}}}, h_K \rangle| \leq & \|\partial_x\psi_{K^{\frac{\pi}{2}}}\|_{L^{p}}\cdot \|H^1_K\|_{L^{p'}},
\end{align*}
where 
\begin{equation*}
H^1_K(x,y) := \int_{\infty}^x h_K(s,y) ds.
\end{equation*}
Moreover, we have the estimates that
\begin{equation}
\label{IBP}
  \begin{split} 
& \|H^1_K\|_{L^{p'}} \lesssim K_1 \cdot |K|^{-\frac{1}{2}} \cdot |K|^{\frac{1}{p'}} =K_1 \cdot |K|^{\frac{1}{2}} \cdot |K|^{-\frac{1}{p}}  \\
& \|\partial_x\psi_{K^{\frac{\pi}{2}}}\|_{L^p} \lesssim K_2^{-1} \cdot |K|^{-\frac{1}{2}} \cdot |K|^{\frac{1}{p}}.
\end{split}
\end{equation}
Finally by combining both inequalities in (\ref{IBP}) and recalling the condition $K_2 > \epsilon^{-1} K_1$, we conclude that 
\begin{align}\label{ana}
 |\langle \psi_{K^{\frac{\pi}{2}}}, h_K \rangle| \lesssim \frac{K_1}{K_2} < \epsilon.
\end{align}

\item 
\textbf{Comparison of geometric and analytic approaches}
Though in the current setting, both approaches give the same estimate, it is noteworthy that the geometric approach is less optimal than the analytic approach because the former never takes advantage of the oscillations of $F$ which could be captured by the latter approach. More precisely, the geometric approach would work equally well for any $L^2$-normalized non-negative function supported on $K$. The simplest example would be 
\begin{equation}
F= |K|^{-\frac{1}{2}}\mathbbm{1}_{K}.
\end{equation}
However, the analytic approach exploits the information of $\partial_x (F \circ \phi^{\theta})$ not only for the domain intersecting with $K$, but also the behavior of  $\partial_x (F \circ \phi^{\theta})$ outside and far away from $K$, which gives the small factor $K_2^{-1}$ in (\ref{ana}).

\medskip

We would now look at a more complicated example of $F$ to elaborate this advantage.

Let $K^{\frac{\pi}{2}}$ denote the rectangle rotated by the angle $\frac{\pi}{2}$ from $K$ and let $\mathcal{K}^{\frac{\pi}{2}}_t$ denote the collection of rectangles which are vertically translated copy of $K^{\frac{\pi}{2}}$ as indicated by the red rectangles in Figure \ref{figure:compare_geo_ana}. Let
\begin{equation}\label{def_superposition_wavelet}
\psi_{\mathcal{K}_t^{\frac{\pi}{2}}} := |K|^{-\frac{1}{2}} \sum_{K' \in \mathcal{K}^{\frac{\pi}{2}}_t}-\psi(K_2^{-1}x, K_1^{-1}y - y_{K'}).
\end{equation}
where $y_{K'}$ denotes the $y$-coordinate of the center of $K' \in \mathcal{K}^{\frac{\pi}{2}}_t$.
Define
\begin{equation*}
F := \psi_{\mathcal{K}_t^{\frac{\pi}{2}}} \circ \phi^{-\frac{\pi}{2}}.
\end{equation*}
Equivalently, we have $F \circ \phi^{\frac{\pi}{2}} = \psi_{\mathcal{K}_t^{\frac{\pi}{2}}}$.
\begin{enumerate}
\item \textbf{Geometric approach:}
Without the consideration of oscillations, we perform the same geometric estimate for each $K'$ so that for $\theta=\pi/2$
\begin{align}
|\langle F \circ \phi^{\theta}, h_K \rangle| \leq |K|^{-\frac{1}{2}} |K|^{-\frac{1}{2}}\sum_{K' \in \mathcal{K}^{\frac{\pi}{2}}_t}|K' \cap K| \sim 1.
\end{align}
\item
\textbf{Analytic approach:}
We can take advantage of the mild oscillations of $\psi_{\mathcal{K}_t^{\frac{\pi}{2}}}$ in $x$-direction and the cancellation of some relatively highly oscillating piece $h_K$ in $x$-direction. We use the integration by parts to apply both information, which yields the same estimate as before:
\begin{align*}
|\langle \psi_{\mathcal{K}_t^{\frac{\pi}{2}}}, h_K \rangle| \lesssim \|\partial_x\psi_{\mathcal{K}_t^{\frac{\pi}{2}}}\|_{L^{p}}\cdot \|H^1_K\|_{L^{p'}}. 
\end{align*}
Then by applying the estimate for $H_K^1$ in (\ref{IBP}), we obtain
\begin{equation} \label{mot_eccentricity}
|\langle \psi_{\mathcal{K}_t^{\frac{\pi}{2}}}, h_K \rangle| \lesssim  \|\partial_x\psi_{\mathcal{K}_t^{\frac{\pi}{2}}}\|_{L^{p}} \cdot K_1^{1-\frac{1}{p}}K_2^{-\frac{1}{p}}|K|^{\frac{1}{2}}.
\end{equation}
By choosing $p = 2$, we conclude that
\begin{equation} \label{mot_long_tall_total}
|\langle \psi_{\mathcal{K}_t^{\frac{\pi}{2}}}, h_K \rangle| \lesssim  \epsilon^{\frac{1}{2}}|K|^{\frac{1}{2}}\|\partial_x\psi_{\mathcal{K}_t^{\frac{\pi}{2}}}\|_{L^{2}}.
\end{equation}
It is also possible to recover the same estimate in (\ref{ana}) as follows. We will choose $p = \infty$. Then
since $\psi_{\mathcal{K}_t^{\frac{\pi}{2}}}$ is defined (in \eqref{def_superposition_wavelet}) as the superposition of functions with disjoint supports and each function is a translated copy of $\psi(K_2^{-1}x, K_1^{-1},y)$, we obtain
\begin{equation} \label{ana_recover}
\|\partial_x\psi_{\mathcal{K}_t^{\frac{\pi}{2}}}\|_{L^{\infty}} \leq K_2^{-1}|K|^{-\frac{1}{2}}\|\psi\|_{W^{1,\infty}}
\end{equation}
which generates the exactly same estimate as (\ref{ana}) after combining (\ref{IBP}) for $H_K^1$ and (\ref{ana_recover}):
\begin{equation}
|\langle \psi_{\mathcal{K}_t^{\frac{\pi}{2}}}, h_K \rangle| \lesssim K_2^{-1} |K|^{-\frac{1}{2}} \cdot K_1 |K|^{\frac{1}{2}} < \epsilon.
\end{equation}
\end{enumerate}
\end{enumerate}
Motivated by the comparison and computations, we would like to take advantage of the regularity of $F \in W^{s,p}$ to extract small factors of the wavelet where $F \circ \phi^{\theta}$ is tested on so that we could obtain some similar estimate as (\ref{mot_eccentricity}) and thus (\ref{mot_long_tall_total}). One key observation is that the condition on the eccentricity $\varepsilon_K < \epsilon$ (or $\varepsilon_K > \epsilon^{-1}$) plays an important role to attain the small power $\epsilon$ from (\ref{mot_eccentricity}). The following proposition generalizes the estimate (\ref{mot_long_tall_total}) under the condition on $\varepsilon_K$ when the wavelets are localized to a dyadic rectangle. 
\begin{figure}
\centering

\tikzset{every picture/.style={line width=0.75pt}} 

\begin{tikzpicture}[x=0.75pt,y=0.75pt,yscale=-1,xscale=1]

\draw   (337.33,204.07) -- (337.33,302.95) -- (324.33,302.95) -- (324.33,204.07) -- cycle ;
\draw  [dash pattern={on 4.5pt off 4.5pt}] (305.84,116.97) -- (355.53,117.03) -- (355.3,389.03) -- (305.6,388.97) -- cycle ;
\draw  [color={rgb, 255:red, 74; green, 144; blue, 226 }  ,draw opacity=1 ] (326.62,203.32) -- (352.21,298.83) -- (339.65,302.19) -- (314.06,206.68) -- cycle;
\draw  [color={rgb, 255:red, 144; green, 19; blue, 254 }  ,draw opacity=1 ] (296.48,215.82) -- (366.4,285.73) -- (357.21,294.92) -- (287.29,225.01) -- cycle ;
\draw  [color={rgb, 255:red, 208; green, 2; blue, 27 }  ,draw opacity=1 ] (284.27,235.53) -- (381.65,252.7) -- (379.39,265.5) -- (282.02,248.33) -- cycle ;

\draw (310,131.4) node [anchor=north west][inner sep=0.75pt]    {$2\epsilon^{-1} K$};
\draw (325.33,207.47) node [anchor=north west][inner sep=0.75pt]    {$K$};

\end{tikzpicture}

\caption{Naive motivation for Theorem \ref{main_thm_pos}.} \label{figure:motivation}
\end{figure}
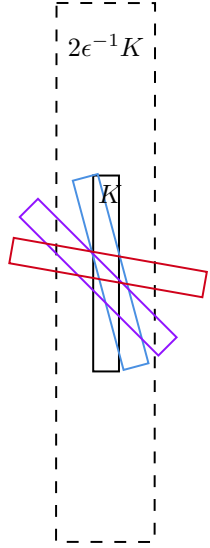

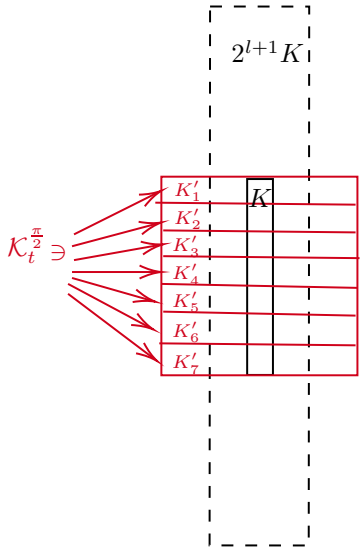
\begin{figure}
\centering

\tikzset{every picture/.style={line width=0.75pt}} 

\begin{tikzpicture}[x=0.75pt,y=0.75pt,yscale=-1,xscale=1]

\draw   (342.33,223.07) -- (342.33,321.95) -- (329.33,321.95) -- (329.33,223.07) -- cycle ;
\draw  [dash pattern={on 4.5pt off 4.5pt}] (310.62,135.98) -- (360.62,136.02) -- (360.38,408.02) -- (310.38,407.98) -- cycle ;
\draw  [color={rgb, 255:red, 208; green, 2; blue, 27 }  ,draw opacity=1 ] (286.06,221.88) -- (384.94,221.88) -- (384.94,322.12) -- (286.06,322.12) -- cycle ;
\draw [color={rgb, 255:red, 208; green, 2; blue, 27 }  ,draw opacity=1 ]   (283,235) -- (384,236) ;
\draw [color={rgb, 255:red, 208; green, 2; blue, 27 }  ,draw opacity=1 ]   (287,249) -- (384,250) ;
\draw [color={rgb, 255:red, 208; green, 2; blue, 27 }  ,draw opacity=1 ]   (287,262) -- (386,263) ;
\draw [color={rgb, 255:red, 208; green, 2; blue, 27 }  ,draw opacity=1 ]   (286,276) -- (385,278) ;
\draw [color={rgb, 255:red, 208; green, 2; blue, 27 }  ,draw opacity=1 ]   (287,290) -- (384,292) ;
\draw [color={rgb, 255:red, 208; green, 2; blue, 27 }  ,draw opacity=1 ]   (285,306) -- (384,307) ;
\draw [color={rgb, 255:red, 208; green, 2; blue, 27 }  ,draw opacity=1 ]   (239,276) -- (280.24,298.06) ;
\draw [shift={(282,299)}, rotate = 208.14] [color={rgb, 255:red, 208; green, 2; blue, 27 }  ,draw opacity=1 ][line width=0.75]    (10.93,-3.29) .. controls (6.95,-1.4) and (3.31,-0.3) .. (0,0) .. controls (3.31,0.3) and (6.95,1.4) .. (10.93,3.29)   ;
\draw [color={rgb, 255:red, 208; green, 2; blue, 27 }  ,draw opacity=1 ]   (239,281) -- (282.4,313.79) ;
\draw [shift={(284,315)}, rotate = 217.07] [color={rgb, 255:red, 208; green, 2; blue, 27 }  ,draw opacity=1 ][line width=0.75]    (10.93,-3.29) .. controls (6.95,-1.4) and (3.31,-0.3) .. (0,0) .. controls (3.31,0.3) and (6.95,1.4) .. (10.93,3.29)   ;
\draw [color={rgb, 255:red, 208; green, 2; blue, 27 }  ,draw opacity=1 ]   (242,251) -- (284.21,229.89) ;
\draw [shift={(286,229)}, rotate = 513.4300000000001] [color={rgb, 255:red, 208; green, 2; blue, 27 }  ,draw opacity=1 ][line width=0.75]    (10.93,-3.29) .. controls (6.95,-1.4) and (3.31,-0.3) .. (0,0) .. controls (3.31,0.3) and (6.95,1.4) .. (10.93,3.29)   ;
\draw [color={rgb, 255:red, 208; green, 2; blue, 27 }  ,draw opacity=1 ]   (241,273) -- (281.08,284.45) ;
\draw [shift={(283,285)}, rotate = 195.95] [color={rgb, 255:red, 208; green, 2; blue, 27 }  ,draw opacity=1 ][line width=0.75]    (10.93,-3.29) .. controls (6.95,-1.4) and (3.31,-0.3) .. (0,0) .. controls (3.31,0.3) and (6.95,1.4) .. (10.93,3.29)   ;
\draw [color={rgb, 255:red, 208; green, 2; blue, 27 }  ,draw opacity=1 ]   (241,270) -- (284,270) ;
\draw [shift={(286,270)}, rotate = 180] [color={rgb, 255:red, 208; green, 2; blue, 27 }  ,draw opacity=1 ][line width=0.75]    (10.93,-3.29) .. controls (6.95,-1.4) and (3.31,-0.3) .. (0,0) .. controls (3.31,0.3) and (6.95,1.4) .. (10.93,3.29)   ;
\draw [color={rgb, 255:red, 208; green, 2; blue, 27 }  ,draw opacity=1 ]   (242,264) -- (285.03,256.35) ;
\draw [shift={(287,256)}, rotate = 529.9200000000001] [color={rgb, 255:red, 208; green, 2; blue, 27 }  ,draw opacity=1 ][line width=0.75]    (10.93,-3.29) .. controls (6.95,-1.4) and (3.31,-0.3) .. (0,0) .. controls (3.31,0.3) and (6.95,1.4) .. (10.93,3.29)   ;
\draw [color={rgb, 255:red, 208; green, 2; blue, 27 }  ,draw opacity=1 ]   (241,257) -- (284.07,245.52) ;
\draw [shift={(286,245)}, rotate = 525.0699999999999] [color={rgb, 255:red, 208; green, 2; blue, 27 }  ,draw opacity=1 ][line width=0.75]    (10.93,-3.29) .. controls (6.95,-1.4) and (3.31,-0.3) .. (0,0) .. controls (3.31,0.3) and (6.95,1.4) .. (10.93,3.29)   ;

\draw (320,151.4) node [anchor=north west][inner sep=0.75pt]    {$2^{l+1} K$};
\draw (328.33,226.47) node [anchor=north west][inner sep=0.75pt]    {$K$};
\draw (207,246.4) node [anchor=north west][inner sep=0.75pt]  [color={rgb, 255:red, 208; green, 2; blue, 27 }  ,opacity=1 ]  {$\mathcal{K}^{\frac{\pi }{2}}_{t}$};
\draw (290.98,222.42) node [anchor=north west][inner sep=0.75pt]  [font=\scriptsize,color={rgb, 255:red, 208; green, 2; blue, 27 }  ,opacity=1 ,rotate=-359.89]  {$K'_{1}$};
\draw (290.98,236.42) node [anchor=north west][inner sep=0.75pt]  [font=\scriptsize,color={rgb, 255:red, 208; green, 2; blue, 27 }  ,opacity=1 ,rotate=-359.89]  {$K'_{2}$};
\draw (290.01,250.4) node [anchor=north west][inner sep=0.75pt]  [font=\scriptsize,color={rgb, 255:red, 208; green, 2; blue, 27 }  ,opacity=1 ,rotate=-359.89]  {$K'_{3}$};
\draw (290.01,264.4) node [anchor=north west][inner sep=0.75pt]  [font=\scriptsize,color={rgb, 255:red, 208; green, 2; blue, 27 }  ,opacity=1 ,rotate=-359.89]  {$K'_{4}$};
\draw (290.01,278.4) node [anchor=north west][inner sep=0.75pt]  [font=\scriptsize,color={rgb, 255:red, 208; green, 2; blue, 27 }  ,opacity=1 ,rotate=-359.89]  {$K'_{5}$};
\draw (290.01,293.4) node [anchor=north west][inner sep=0.75pt]  [font=\scriptsize,color={rgb, 255:red, 208; green, 2; blue, 27 }  ,opacity=1 ,rotate=-359.89]  {$K'_{6}$};
\draw (290.01,309.4) node [anchor=north west][inner sep=0.75pt]  [font=\scriptsize,color={rgb, 255:red, 208; green, 2; blue, 27 }  ,opacity=1 ,rotate=-359.89]  {$K'_{7}$};
\draw (240,265.6) node [anchor=north west][inner sep=0.75pt]  [color={rgb, 255:red, 208; green, 2; blue, 27 }  ,opacity=1 ,rotate=-180]  {$\in $};

\end{tikzpicture}
\caption{Comparison of geometric and analytic approaches} \label{figure:compare_geo_ana}
\end{figure}

\begin{proposition}\label{long_tall_pre} Suppose that $F \in W^{s,p}$ for some $s\in(0,1)$, $p\in[2,\infty)$ with $s>\frac{2}{p}$. Fix a dyadic grid $\mathbb{D}^{\alpha}$ and a dyadic rectangle $K_0 \in \mathbb{D}^{\alpha}$. 
Let $\zeta$ denote a positive real number. Then for any rotation map $\phi^{\theta}$ with $0 \leq \theta < \frac{\pi}{2}$,
\begin{align} 
&\sum_{\substack{K \in \mathbb{D}^{\alpha} \\ \substack{K \subseteq K_0 \\ \varepsilon_K \leq \zeta}}} \left|\langle F \circ \phi^{\theta}, h_K \rangle \right|^2 \lesssim \zeta^{2\gamma } |K_0|^{} \|F\|_{W^{s,p}}^2, \label{eq:propex-1} \\
& \sum_{\substack{K \in \mathbb{D}^{\alpha} \\ \substack{K \subseteq K_0 \\ \varepsilon_K \geq \zeta^{-1}}}} \left|\langle F \circ \phi^{\theta}, h_K \rangle \right|^2 \lesssim \zeta^{2\gamma } |K_0|^{} \|F\|_{W^{s,p}}^2  \label{eq:propex-2}
\end{align}
for an arbitrary exponent $\gamma\in(0,\delta)$ with $\delta:=s-\frac{2}{p}$.
\end{proposition}

\begin{proof}
Let us detail the argument for \eqref{eq:propex-1} and the same argument will also hold for \eqref{eq:propex-2}.

Since the Sobolev norm preserves the rotation map, we have $F \circ \phi \in W^{s,p}$. Then we directly have (see Lemma \ref{lemma:regularity})
$$  |\langle F \circ \phi^{\theta}, h_K \rangle| \lesssim \|F\|_{W^{s,p}} |K_1|^{\gamma_1} |K_2|^{\gamma_2} |K|^{1/2}$$
where we can choose arbitrarily $\gamma_1,\gamma_2 \in (-\delta,\delta)$. So we get
\begin{align*}
\sum_{\substack{K \in \mathbb{D}^{\alpha} \\ K \subseteq K_0 \\ \varepsilon_K \leq \zeta}} |\langle F \circ \phi^{\theta}, h_K \rangle|^2  & \lesssim \|F\|_{W^{s,p}}^2 \sum_{\substack{K \in \mathbb{D}^{\alpha} \\ K \subseteq K_0 \\ \varepsilon_K \leq \zeta}} |K| |K_1|^{2\gamma_1} |K_2|^{2\gamma_2} \\
& \lesssim \|F\|_{W^{s,p}}^2 \sum_{2^{k_1} \leq |K_0^1|} \sum_{\substack{2^{k_2} \leq |K_0^2| \\ 2^{k_2}\leq \zeta 2^{k_1}}} \frac{|K_0|}{2^{k_1+k_2}} 2^{k_1+k_2} 2^{2\gamma_1 k_1} 2^{2\gamma_2 k_2},
\end{align*}
since for $k_1,k_2$ fixed, then the number of rectangles $K\subset K_0$ of dimensions  $2^{k_1} \times 2^{k_2}$ is exactly $\frac{|K_0|}{2^{k_1+k_2}}$. Using that we can choose $\gamma_2>0$ (as close as we want of $\delta$) so that
\begin{align*}
\sum_{\substack{K \in \mathbb{D}^{\alpha} \\ K \subseteq K_0 \\ \varepsilon_K \leq \zeta}} |\langle F \circ \phi^{\theta}, h_K \rangle|^2  & \lesssim \|F\|_{W^{s,p}}^2 |K_0|^{} \sum_{2^{k_1} \leq |K_0^1|} \sum_{\substack{2^{k_2} \leq |K_0^2| \\ 2^{k_2}\leq \zeta 2^{k_1}}} 2^{2\gamma_1 k_1} 2^{2\gamma_2 k_2}  \\
& \lesssim \|F\|_{W^{s,p}}^2 |K_0|^{} \zeta^{2\gamma_2 } \sum_{2^{k_1} \leq |K_0^1|} 2^{2(\gamma_1+\gamma_2) k_1}  \\
& \lesssim \|F\|_{W^{s,p}}^2 |K_0|^{} \zeta^{2\gamma_2},
\end{align*}
where at the end, we compute the sum over $k_1$ by chosing $\gamma_1\in(-\delta,\delta)$ such that $\gamma_1+\gamma_2> 0$ if $k_1\leq 0$ and $\gamma_1+\gamma_2<0 $ if $k_1\geq 0$ (which is always possible).
\end{proof}

\section{Proof of Theorem \ref{main_thm_pos} - Localizations and Decompositions}

\label{sec:main0}

We will start  the proof of Theorem \ref{main_thm_pos}. Let us fix the rotation $\phi=\phi^\theta$ for some angle $\theta$, a function $F\in \BMO \cap W^{s,p}$ for some exponents $s\in(0,1)$, $p\in[2,\infty)$ with $s>\frac{2}{p}$ and a small parameter $\epsilon \in(0,1)$. 

We aim to control the biparameter $\BMO$ norm of $F \circ \phi$, so we have to bound any oscillation with respect to arbitrary open subsets. So let us fix an open subset of finite measure $\Omega \subset \BBR^2$ and a dyadic grid $\mathbb{D}^\alpha$ and we would like to control
\begin{equation} 
 \osc_{\Omega}(F \circ \phi):=\left(\sum_{\substack{K \in \mathbb{D}^{\alpha}\\ K \subseteq \Omega}} \left| \langle F \circ \phi, h_K \rangle \right|^2\right)^{\frac{1}{2}}. \label{eq:oscOmega}
\end{equation}
First of all, we would like to take advantage of Proposition \ref{long_tall_pre} which relies heavily on the preservation of the Sobolev norm under rotations and suggests 
\begin{enumerate}
\item[(a)]
a separate treatment of sufficiently large or small eccentricities from intermediate eccentricities of $K$;
\item[(b)]
a natural localization of the oscillation on an arbitrary open set to rectangles which are more regular from both geometric and analytic perspectives for the application of Proposition \ref{long_tall_pre} and for further estimates.
\end{enumerate}
\begin{definition} \label{def_max_rectangle}
A rectangle $K_0$ 
is \textbf{maximal} if it is maximal with respect to inclusion in $\Omega$, written as $K_0 \subseteq_{\max} \Omega$. 
We would denote the maximal rectangles by $K_{\max}$.
\end{definition}

Then localization of (\ref{eq:oscOmega}) to maximal rectangles can be achieved as follows.
\begin{align} \label{localization_osc}
\osc_{\Omega}(F \circ \phi) =&\left(\sum_{\substack{K_{\max} \in \mathbb{D}^{\alpha}\\ K_{\max} \subseteq \Omega}}\sum_{\substack{K \in \mathbb{D}^{\alpha} \\ K \subseteq K_{\max}} }\left| \langle F \circ \phi, h_K \rangle \right|^2\right)^{\frac{1}{2}}
\end{align}
In general, the maximal rectangles are nested so that $\displaystyle \sum_{K_{\max}\subseteq \Omega}|K_{\max}|$ is not comparable to $|\Omega|$ apriori. To resolve this issue, we will further decompose the maximal rectangles so that we can apply Journ\'e's lemma and understand the estimate we could aim for. 
\begin{definition}\label{K_l_definition}
For an integer $l \in \mathbb{N}$,  define a sub-collection of the maximal rectangles $\mathcal{K}_l$ as
\begin{equation}
\mathcal{K}_l := \{K_0 \in \mathbb{D}^\alpha: K_0 \subseteq_{\max} \Omega \ \  \text{and} \ \ 2^{l}K_0 \subseteq \widetilde{\widetilde{\Omega}}, 2^{l+1} K_0\nsubseteq \widetilde{\widetilde{\Omega}} \},
\end{equation}
where 
\begin{equation*}
\widetilde \Omega:=\{(x,y), MM(\mathbbm{1}_\Omega)(x,y)> \epsilon\} 
\end{equation*}
and
\begin{equation*}
\widetilde{\widetilde \Omega}:=\{(x,y), MM(\mathbbm{1}_{\widetilde{\Omega}})(x,y)> \epsilon\}.
\end{equation*}
\end{definition}
\begin{remark}
According to Definition \ref{K_l_definition}, we obtain a decomposition on the collection of maximal dyadic rectangles, denoted by $\mathcal{K}_{\max}$. In particular, 
\begin{equation} \label{max_rec_decomposition}
\mathcal{K}_{\max} = \bigcup_{l \in \mathbb{N}} \mathcal{K}_l.
\end{equation}
\end{remark}
One important observation about the relationship between $\epsilon>0$ and $l \in \mathbb{N}$ is summarized in the following lemma.
\begin{lemma} \label{l_epsilon}
Let $\epsilon>0$ and $l \in \mathbb{N}$ denote the enlargement factor and decomposition levels as specified in Definition \ref{K_l_definition}. Then the following relation is always true:
\begin{equation}
2^{- 2 l} \lesssim \epsilon.
\end{equation}
\end{lemma}

\begin{proof}
 Indeed if $2^{l} \leq \frac{1}{4} \sqrt{\epsilon}$ then for any maximal rectangle $K_0$ and for every point $(x,y)\in 2^{l+1} K_0$, we will have
 $$ MM(\mathbbm{1}_\Omega)(x,y) \geq \frac{1}{|2^{l+1} K_0|} \int_{2^{l+1} K_0} \mathbbm{1}_{\Omega}(s,t) \, dsdt \geq 2^{-2(l+1)} $$
 since we work in dimension $2$. So that yields $MM(\mathbbm{1}_\Omega)(x,y)\geq 2\epsilon>\epsilon$ and hence $(x,y)\in \widetilde{\widetilde \Omega}$. We deduce that $2^{l+1} K_0 \subset \widetilde{\widetilde \Omega}$ which is in contradiction with the definition of $K_0$ beeing in $\mathcal{K}_l$.
\end{proof}

\subsection{Localization and decomposition of $\osc_{\Omega}(F \circ \phi)$ for dyadic rectangles with sufficiently large and small eccentricities}
We can then rewrite (\ref{localization_osc}) based on the decomposition (\ref{max_rec_decomposition}) as follows. 
\begin{align} \label{decomposition_osc}
\osc_{\Omega}(F \circ \phi) =&\left(\sum_{l \in \mathbb{N}}\ \sum_{\substack{K_{\max} \in \mathbb{D}^{\alpha}\\ K_{\max} \subseteq \Omega \\ K_{\max} \in \mathcal{K}_l}}\ \sum_{\substack{K \in \mathbb{D}^{\alpha} \\ K \subseteq K_{\max}}} \left| \langle F \circ \phi, h_K \rangle \right|^2\right)^{\frac{1}{2}}
\end{align}
Thanks to Journ\'e lemma  \cite{JourneLemma}, the nested sum of maximal rectangles in each sub-collection $\mathcal{K}_l$ with $l \in \mathbb{N}$ fixed is comparable to the measure of $\Omega$ with a loss of an arbitrarily small power of $2^l$.

We first state Journ\'e's lemma:

\begin{lemma}[Journ\'e's lemma] 
Let $U \subseteq \mathbb{R}^2$ be an open set and $k$ a fixed non-negative integer. Assume that $\mathcal{R}$ is a collection of dyadic rectangles which are maximal with respect to inclusion and which all lie in $U$.  Let 
\begin{equation*}
\tilde{U} := \{MM \mathbbm{1}_{U} > \frac{1}{2} \}
\end{equation*}
and
\begin{equation*}
\tilde{\tilde{U}} := \{MM \mathbbm{1}_{\tilde{U}} > \frac{1}{2} \}.
\end{equation*}
Suppose that for each $R \in \mathcal{R}$, $2^k R \subseteq \tilde{\tilde{U}}$ and that $k$ is maximal with this property. Then for every $\nu>0$, 
$$ \sum_{R \in \mathcal{R}} |R| \lesssim_{\nu} 2^{l\nu} |U|.$$
\end{lemma}

We refer the reader to \cite[Tome II - Section 3.6]{MuscaluSchlag} for a proof of Journ\'e's lemma (initially detailed in \cite{JourneLemma}). We remark that for $U := \Omega$, and the enlargements $\tilde{\Omega}$ and $\tilde{\tilde{\Omega}}$ are slightly different from $\tilde{U}$ and $\tilde{\tilde{U}}$ since the fraction $\frac{1}{2}$ is replaced by $\epsilon >0$, which suggests that we would need a more general statement of Journ\'e's lemma in the current setting. We will show how such generalization is implemented to generate the desired estimate in the following lemma.

\begin{lemma} \label{lemma:journe}
For every $\nu>0$ arbitrarily small and every integer $l\geq 1$ (with $2^{-2l} \lesssim \epsilon$), we have
$$ \sum_{K_{\max} \in \mathcal{K}_l} |K_{\max}| \lesssim \big(\sqrt{\epsilon} . 2^{l}\big)^\nu |\Omega|.$$
\end{lemma}
\begin{proof}
 Indeed for $K_{\max} \in \mathcal{K}_l$, we have that 
$$ |\widetilde \Omega \cap 2^{l+1} K_{\max}| \leq \epsilon |2^{l+1} K_{\max}|$$
which can be rewritten as
$$ |\widetilde \Omega \cap 2^{l+1} K_{\max}| \leq \frac{1}{2} |\sqrt{\epsilon} 2^{l+2} K_{\max}|$$
and that yields (because $\epsilon$ is a small parameter and so $2\sqrt{\epsilon} \leq 1$)
$$ |\Omega \cap \sqrt{\epsilon} 2^{l+2} K_{\max}| \leq \frac{1}{2} |\sqrt{\epsilon} 2^{l+2} K_{\max}|.$$
Indeed this is the main property in the Journ\'e's Lemma. With $k_0$ such that $\sqrt{\epsilon}2 \simeq 2^{k_0}$, we see that for $K_{\max} \in \mathcal{K}_l$ then $K_{\max} \in \mathcal{C}_{l+k_0}$ where $\mathcal{C}_k$ is the maximal class of Journ\'e's Lemma. So by the previous lemma, one gets
\begin{align*}
 \sum_{K_{\max} \in \mathcal{K}_l} |K_{\max}| & \leq \sum_{K_{\max}\in \mathcal{C}_{l+k_0}} |K_{\max}| \\
 & \lesssim  2^{(l+k_0)\nu} |\Omega| \lesssim \big(\sqrt{\epsilon} . 2^{l}\big)^\nu |\Omega|.
\end{align*}
\end{proof}

Inspired by Journ\'e's Lemma, we can apply Proposition \ref{long_tall_pre} to the oscillation of $F$ tested on the wavelets $K$ which are contained in a fixed maximal rectangle $K_{\max} \in \mathcal{K}_l$ and satisfy the eccentricity condition quantified by $2^{-l}$, which yields the following corollary.
\begin{corollary}\label{long_tall_total} For a fixed maximal rectangle $K_{\max} \in \mathcal{K}_l$ with $ |K_{\max}| = K_{\max}^1 \times K_{\max}^2$, we have
\begin{align} 
\sum_{\substack{K \in \mathbb{D}^{\alpha} \\ \substack{K \subseteq K_{\max} \\ \varepsilon_K \leq 2^{-l+3}}}} \left|\langle F \circ \phi, h_K \rangle \right|^2 \lesssim 2^{-2\gamma l } |K_{\max}|^{} \|F\|_{W^{s,p}}^2, 
\end{align}
and similarly
\begin{align} 
\sum_{\substack{K \in \mathbb{D}^{\alpha} \\ K \subseteq K_{\max} \\ \varepsilon_K \geq 2^{l-3}}}\left|\langle F \circ \phi, h_K \rangle \right|^2 \lesssim 2^{- 2\gamma l} |K_{\max}|^{} \|F\|_{W^{s,p}}^2, 
\end{align}
for an arbitrary exponent $\gamma\in(0,\delta)$ with $\delta:=\min(s-\frac{2}{p}, \frac{1}{p})$.
\end{corollary}
As a consequence of Corollary \ref{long_tall_total} and Journ\'e's lemma, for $0 < \gamma < \min(s-\frac{2}{p}, \frac{1}{p})$ and $\nu >0$,
\begin{equation} \label{result_each_level}
\sum_{\substack{K_{\max} \in \mathbb{D}^{\alpha}\\ K_{\max} \subseteq \Omega \\ K_{\max} \in \mathcal{K}_l}} \ \sum_{\substack{K \in \mathbb{D}^{\alpha} \\ \substack{K \subseteq K_{\max} \\ \varepsilon_K \leq 2^{-l+3}}}} \left|\langle F \circ \phi, h_K \rangle \right|^2 \lesssim 2^{-2\gamma l } \sum_{K_{\max} \in \mathcal{K}_l}|K_{\max}|^{} \|F\|_{W^{s,p}}^2 \lesssim \epsilon^{\frac{\nu}{2}}2^{-l(2\gamma-\nu)}|\Omega|\|F\|_{W^{s,p}}^2.
\end{equation}
By choosing $\nu>0$ such that $2\gamma - \nu >0$, we can majorize the oscillation of $F$ on a restricted subset of wavelets $K \subseteq \Omega$ by
\begin{equation} 
\left(\sum_{l \in \mathbb{N}} \ \sum_{\substack{K_{\max} \in \mathbb{D}^{\alpha} \\ K_{\max} \subseteq \Omega \\ K_{\max} \in \mathcal{K}_l}} \ \sum_{\substack{K \in \mathbb{D}^{\alpha} \\ \substack{K \subseteq K_{\max} \\ \varepsilon_K \leq 2^{-l+3}}}} \left|\langle F \circ \phi, h_K \rangle \right|^2\right)^{\frac{1}{2}} \lesssim \epsilon^{\frac{\gamma}{2}}|\Omega|^{\frac{1}{2}}\|F\|_{W^{s,p}}, \label{error1}
\end{equation}
where the inequality follows from (\ref{result_each_level}) and the relation between $l$ and $\epsilon$ highlighted in Lemma \ref{l_epsilon}.
By the exactly same argument,
\begin{equation} 
\left(\sum_{l \in \mathbb{N}} \ \sum_{\substack{K_{\max} \in \mathbb{D}^{\alpha}\\ K_{\max} \subseteq \Omega \\K_{\max} \in \mathcal{K}_l}} \ \sum_{\substack{K \in \mathbb{D}^{\alpha} \\ \substack{K \subseteq K_{\max} \\ \varepsilon_K \geq 2^{l-3}}}} \left|\langle F \circ \phi, h_K \rangle \right|^2\right)^{\frac{1}{2}} \lesssim \epsilon^{\frac{\gamma}{2}}|\Omega|^{\frac{1}{2}}\|F\|_{W^{s,p}}.
\label{error2}
\end{equation}

\subsection{Localization and decomposition of $\osc_{\Omega}(F \circ \phi)$ for dyadic rectangles with intermediate eccentricities}
We start with the following crucial observation.
\begin{obs} \label{R_restrict}
Fix an integer $l \in \mathbb{N}$. Suppose that $K _{\max}\in \mathcal{K}_l$ (Definition \ref{K_l_definition}) and $R$ is a rectangle with sides parallel to the coordinate axes and $|R| := 2^{r_1} \times 2^{r_2}$. Further assume that $R \nsubseteq \widetilde{\phi^\theta(\Omega)}$. Then at least one of the following necessary conditions hold:
\begin{equation}\label{R_1}
2^{-l}2^{r_1} \geq \frac{1}{2}\min(\frac{K_{\max}^1}{|\cos\theta|}, \frac{K_{\max}^2}{|\sin \theta|})
\qquad \textrm{or} \qquad 
2^{-l}2^{r_2} \geq \frac{1}{2}\min(\frac{K_{\max}^1}{|\sin \theta|}, \frac{K_{\max}^2}{|\cos\theta|}).
\end{equation}
Moreover $l$ has to satisfy the condition $l \geq L$ where $L$ is defined to be $2^{-L} := \epsilon^{1/2}$.
\end{obs}

\begin{proof}
Fix $l \in \mathbb{N}$. The equivalence of the following two conditions is highlighted and heavily used in the proof. In particular, since $\phi^\theta$ is a $1$-to-$1$ map,
\begin{equation*}
 2^{l}K_{\max} \subseteq \tilde{\Omega}
\end{equation*}
if and only if 
\begin{equation*}
\phi^{\theta}(2^{l}K_{\max}) \subseteq \phi^{\theta}(\widetilde{\Omega}).
\end{equation*}
Indeed, the reason is that for $(x,y)\in\BBR^2$
\begin{align*}
MM^{\theta}[\mathbbm{1}_{\phi^{\theta}(\Omega)}](x,y) & := \sup_{K: (x,y) \in \phi^{\theta}(K)} \frac{1}{|\phi^{\theta}(K)|}\int_{\phi^{\theta}(K)}\mathbbm{1}_{\phi^{\theta}(\Omega)}(t,s) dtds \nonumber \\
& = \sup_{K: \phi^{-\theta}(x,y) \in K}\frac{1}{|K|} \int_{K} \mathbbm{1}_{\Omega}(\tilde{t}, \tilde{s})d\tilde{t}d\tilde{s} = MM[\mathbbm{1}_{\Omega}](\phi^{-\theta}(x,y))
\end{align*}
where the equality follows from the change of variables
\begin{equation*}
(\tilde{t},\tilde{s}) := \phi^{-\theta}(t,s) = [\phi^\theta]^{-1}(t,s)
\end{equation*}
and the fact that $\phi^\theta$ is a rotation and so preserves the measure. As a consequence, this computation shows that if $(\tilde{x}, \tilde{y})$ satisfies that there exists a rectangle $K$ such that 
\begin{equation} \label{cont_cond}
 (\tilde{x}, \tilde{y}) \in K \nonumber \qquad \textrm{and} \qquad
 \frac{|K \cap \Omega|}{|K|} > \epsilon,
\end{equation}
then $(x,y) := \phi^{\theta}(\tilde{x}, \tilde{y})$ satisfies that
\begin{equation} \label{containment_cond}
MM^{\theta}[\mathbbm{1}_{\phi^{\theta}(\Omega)}](x,y)  > \epsilon,
\end{equation}
or equivalently, $(x,y) \in \widetilde{\phi^\theta(\Omega)}$. One recalls that $2^{l}K_{\max} \subseteq \tilde{\Omega}$ means for any $(\tilde{x}, \tilde{y}) \in 2^{l}K_{\max}$, there exists $K$ containing $(\tilde{x},\tilde{y})$ satisfying (\ref{containment_cond}). Therefore, one can conclude that 
\begin{equation} \label{dil_in}
\phi^{\theta}(2^{l}K_{\max}) \subseteq \widetilde{\phi^{\theta}(\Omega)}.
\end{equation} 
A consequence of (\ref{dil_in}) is that every rectangle $R$ with $R \cap \left(\widetilde{\phi^{\theta}(\Omega)}\right)^c \neq \emptyset$ satisfies
\begin{equation} \label{int_inside}
R \cap \left(\phi^{\theta}(2^{l}K_{\max})\right)^c \neq \emptyset.
\end{equation}
Furthermore, for any $K \subseteq K_{\max}$,  $\langle h_R \circ \phi^\theta, h_K \rangle \neq 0 $ implies that
\begin{equation} \label{int_outside}
R \cap \phi^{\theta}(K_{\max}) \neq \emptyset.
\end{equation}
The geometry of $R$, $\phi^{\theta}(K_{\max}) $ and $\phi^{\theta}(2^{l}K_{\max})$ are shown in Figures \ref{figure:RR} and \ref{figureR-2R}. \\
To quantify the geometric depiction, one represents $\phi^{\theta}(K_{\max})$ and $\phi^{\theta}(2^lK_{\max})$ in coordinates. A point $(x,y)\in \BBR$ will be represented by the $2D$ vector $\begin{bmatrix} 
x \\
y 
\end{bmatrix}$ and a rectangle will be described by its four corners. By translation invariance, we can assume that one of the corner of $K_{\max}$ is $0$ and so by denoting $\kappa_1,\kappa_2$ its two length, let us suppose that 
\begin{align*}
& K_{\max} = 
\begin{bmatrix} 
0 & \kappa_1 & \kappa_1& 0\\
0 & 0 &  \kappa_2 & \kappa_2
\end{bmatrix},\\
& 2^l K_{\max} = 
\frac{1}{2}\begin{bmatrix} 
-2^l \kappa_1 & (2^l+1)\kappa_1 &(2^l+1) \kappa_1&-2^l \kappa_1 \\
-2^l \kappa_2 & -2^l \kappa_2  &  (2^l+1)\kappa_2 & (2^l+1) \kappa_2
\end{bmatrix},
\end{align*}
and the rotation $\phi^\theta$ is given by the matrix 
\begin{equation*}
\phi^{\theta} = 
\begin{bmatrix}
\cos \theta & -\sin \theta \\
\sin \theta & \cos \theta
\end{bmatrix}.
\end{equation*}
Then
\begin{align*}
\phi^{\theta}(K_{\max}) := & \phi^{\theta} \cdot
\begin{bmatrix} 
0 & \kappa_1 & \kappa_1& 0\\
0 & 0 &  \kappa_2 & \kappa_2
\end{bmatrix}
= 
\begin{bmatrix}
0 & \kappa_1 \cos \theta & \kappa_1 \cos \theta - \kappa_2 \sin \theta  & -\kappa_2 \sin \theta \\
0 &\kappa_1 \sin \theta & \kappa_1 \sin \theta + \kappa_2 \cos \theta & \kappa_2 \cos \theta
\end{bmatrix};
\end{align*}
and
{\fontsize{9.5}{9.5}
\begin{align*}
& \phi^{\theta}(2^lK_{\max}) \\
 := &
\frac{1}{2} \phi^{\theta} \cdot
\begin{bmatrix} 
-2^l \kappa_1 & (2^l+1)\kappa_1 &(2^l+1) \kappa_1&-2^l \kappa_1 \\
-2^l \kappa_2 & -2^l \kappa_2  &  (2^l+1)\kappa_2 & (2^l+1) \kappa_2
\end{bmatrix} \nonumber\\
 = & {\frac{1}{2} \begin{bmatrix}
 -2^l \kappa_1 \cos \theta + 2^l \kappa_2 \sin \theta & (2^l+1)\kappa_1 \cos \theta + 2^l \kappa_2 \sin\theta & (2^l+1)\kappa_1 \cos \theta - (2^l+1)\kappa_2 \sin \theta   & -2^l\kappa_1 \cos \theta - (2^l+1)\kappa_2 \sin \theta \\
-2^l \kappa_1 \sin \theta -2^l \kappa_2 \cos \theta &(2^l+1) \kappa_1 \sin \theta -2^l \kappa_2 \cos\theta & (2^l+1)\kappa_1 \sin \theta + (2^l+1)\kappa_2 \cos \theta & -2^l \kappa_1 \sin \theta + (2^l+1)\kappa_2 \cos \theta
\end{bmatrix}}
\end{align*}}
Moreover, one observes that for \eqref{int_inside} and \eqref{int_outside} to happen simultaneously, $R$ has to intersect at least one side of $\phi^\theta(K_{\max})$ and at least one side of $\phi^{\theta}(2^lK_{\max})$. More precisely one same side of $R$ will have to intersect a side of $\phi^\theta(K_{\max})$ and a side of $\phi^{\theta}(2^lK_{\max})$ ; or if it is not the case then it has to be true for $2R$. See the following configuration for example in Figure \ref{figureR-2R}.

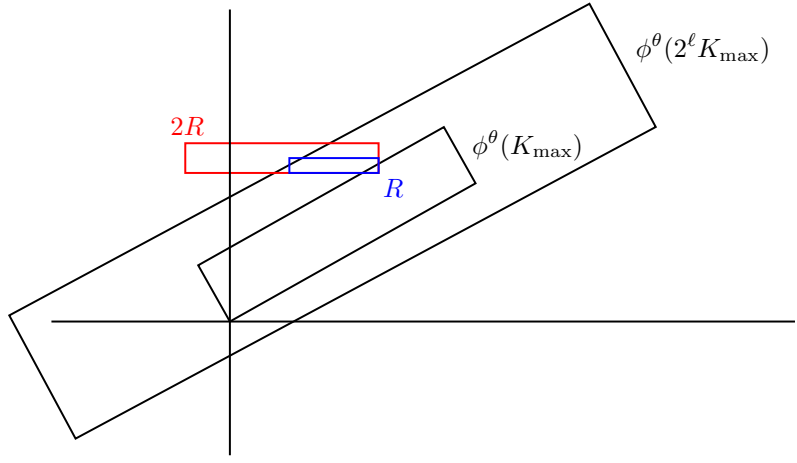
\begin{figure}

\centering

\tikzset{every picture/.style={line width=0.75pt}} 

\begin{tikzpicture}[x=0.75pt,y=0.75pt,yscale=-1,xscale=1,scale=0.75]

\draw   (128.69,182.15) -- (293.89,89.14) -- (315.2,126.99) -- (150,220) -- cycle ;
\draw   (1.65,215.87) -- (391.8,6.15) -- (436.37,89.06) -- (46.23,298.79) -- cycle ;
\draw  [color={red}  ,draw opacity=1 ] (120,100) -- (250,100) -- (250,120) -- (120,120) -- cycle ;
\draw    (30,220) -- (530,220) ;
\draw    (150,10) -- (150,310) ;
\draw  [color={blue}  ,draw opacity=1 ] (190,110) -- (250,110) -- (250,120) -- (190,120) -- cycle ;

\draw (249,70) node [anchor=north west][inner sep=0.75pt]   [align=left] {$ $};
\draw (311,89) node [anchor=north west][inner sep=0.75pt]   [align=left] {$\displaystyle \phi ^{\theta }(K_{\max})$};
\draw (421,24) node [anchor=north west][inner sep=0.75pt]   [align=left] {$\displaystyle \phi ^{\theta }(2^{\ell} K_{\max})$};
\draw (108,80) node [anchor=north west][inner sep=0.75pt]  [color={red}  ,opacity=1 ] [align=left] {$\displaystyle 2R$};
\draw (251,122) node [anchor=north west][inner sep=0.75pt]  [color={blue}  ,opacity=1 ] [align=left] {$\displaystyle R$};

\end{tikzpicture}

\caption{Position of \textcolor{blue}{$\displaystyle R$} and \textcolor{red}{$\displaystyle 2R$}.} \label{figureR-2R}
\end{figure}

Suppose that $0 < \theta< \frac{\pi}{2}$ and assume without loss of generality that one side of $2R$ intersects at least one of the rotated sides, denoted by $l_1$ and $l_2$, of $\phi^{\theta}(K_{\max})$ and intersects at least one of the rotated sides, denoted by $l_1'$ and $l_2'$, of $\phi^{\theta}(2^lK_{\max})$, where
\begin{align}
l_1&:= \left\{\phi^{\theta}\cdot
\begin{bmatrix}
0\\
\tilde{y}
\end{bmatrix},\ 0\leq \tilde y \leq \kappa_2\right\}
= 
\left\{\begin{bmatrix}
-\tilde{ y} \sin \theta \\
\tilde{y} \cos \theta
\end{bmatrix}, \ 0\leq \tilde y \leq \kappa_2\right\} \nonumber \\
l_2 &:=\left\{\phi^{\theta}\cdot
\begin{bmatrix}
\tilde{x}\\
\kappa_2
\end{bmatrix},\ 0\leq \tilde x \leq \kappa_1\right\} 
= 
\left\{\begin{bmatrix}
\tilde{x}\cos \theta - \kappa_2 \sin \theta \\
\tilde{x}\sin\theta+ \kappa_2 \cos \theta
\end{bmatrix},\ 0\leq \tilde x \leq \kappa_1\right\}. \nonumber
\end{align}
And similarly,
\begin{align}
l_1'&:=\left\{\phi^{\theta}\cdot
\begin{bmatrix}
-2^l \kappa_1\\
\tilde{y}
\end{bmatrix},\ -2^l \kappa_2 \leq y \leq (2^l+1)\kappa_2\right\} 
= 
\left\{\begin{bmatrix}
- 2^l \kappa_1 \cos \theta - \tilde{y} \sin \theta \\
-2^l \kappa_1 \sin \theta + \tilde{y} \cos \theta
\end{bmatrix},\ -2^l \kappa_2 \leq y \leq (2^l+1)\kappa_2\right\} \nonumber \\
l_2' &:=\left\{\phi^{\theta}\cdot
\begin{bmatrix}
\tilde{x}\\
(2^l+1)\kappa_2
\end{bmatrix},\  -2^l \kappa_1 \leq x \leq (2^{l}+1)\kappa_1\right\}
= 
\left\{\begin{bmatrix}
\tilde{x}\cos \theta -(2^l+1) \kappa_2 \sin \theta \\
\tilde{x}\sin\theta+ (2^l+1)\kappa_2 \cos \theta
\end{bmatrix},\ -2^l \kappa_1 \leq x \leq (2^{l}+1)\kappa_1 \right\}. 
\end{align}
There are $4$ different cases, that we are going to study:
\begin{enumerate}
\item[(a)]
Suppose that a horizontal side of $2R$ intersects $l_1$ and $l_1'$. So for a fix $c$, there exist $x_1,x_2$ such that $(x_1,c)\in l_1$ and $(x_2,x)\in l_1'$ which means for some suitable $\tilde y_1,\tilde y_2$
\begin{align*}
\begin{bmatrix}
x_1 \\
c
\end{bmatrix}
=
\phi^{\theta}\cdot
\begin{bmatrix}
0 \\
\tilde{y}_1
\end{bmatrix} 
= 
\begin{bmatrix}
-\tilde{ y}_1 \sin \theta \\
\tilde{y}_1 \cos \theta
\end{bmatrix} \qquad \textrm{and} \qquad
\begin{bmatrix}
x_2 \\
c
\end{bmatrix}
=
\phi^{\theta}\cdot
\begin{bmatrix}
-2^l \kappa_1\\
\tilde{y}_2
\end{bmatrix} 
= 
\begin{bmatrix}
- 2^l \kappa_1 \cos \theta - \tilde{y}_2 \sin \theta \\
-2^l \kappa_1 \sin \theta + \tilde{y}_2 \cos \theta
\end{bmatrix}
\end{align*}
A straightforward computation shows that
\begin{equation*}
2 R_1 \geq |x_2-x_1| = \frac{2^l\kappa_1}{|\cos\theta|}.
\end{equation*}
\item[(b)]
Suppose that a horizontal side of $2R$ intersects $l_2$ and $l_2'$. One carries out the same computation as in (a) and yields that
\begin{equation*}
2 R_1 \geq |x_2-x_1| = \frac{2^l\kappa_2}{|\sin\theta|}.
\end{equation*}
\item[(c)] 
Suppose that a horizontal side of $2R$ intersects $l_1$ and $l_2'$, as in Figure \ref{figure:RR}.
\begin{figure}
\centering

\tikzset{every picture/.style={line width=0.75pt}} 

\begin{tikzpicture}[x=0.75pt,y=0.75pt,yscale=-1,xscale=1,scale=0.75]

\draw   (150,180) -- (322.12,180) -- (322.12,220) -- (150,220) -- cycle ;
\draw   (128.69,182.15) -- (293.89,89.14) -- (315.2,126.99) -- (150,220) -- cycle ;
\draw   (1.65,215.87) -- (391.8,6.15) -- (436.37,89.06) -- (46.23,298.79) -- cycle ;
\draw  [color={red}  ,draw opacity=1 ] (10,190) -- (140,190) -- (140,200) -- (10,200) -- cycle ;
\draw    (30,220) -- (530,220) ;
\draw    (150,10) -- (150,310) ;

\draw (326,200) node [anchor=north west][inner sep=0.75pt]   [align=left] {$\displaystyle K_{\max}$};
\draw (311,89) node [anchor=north west][inner sep=0.75pt]   [align=left] {$\displaystyle \phi ^{\theta}(K_{\max})$};
\draw (421,24) node [anchor=north west][inner sep=0.75pt]   [align=left] {$\displaystyle \phi ^{\theta }(2^{\ell} K_{\max})$};
\draw (261,110) node [anchor=north west][inner sep=0.75pt]   [align=left] {$\displaystyle l_{2}$};
\draw (285,30) node [anchor=north west][inner sep=0.75pt]   [align=left] {$\displaystyle l'_{2}$};
\draw (41,260) node [anchor=north west][inner sep=0.75pt]   [align=left] {$\displaystyle l'_{1}$};
\draw (131,210) node [anchor=north west][inner sep=0.75pt]   [align=left] {$\displaystyle l_{1}$};
\draw (21,170) node [anchor=north west][inner sep=0.75pt]  [color={red}  ,opacity=1 ] [align=left] {$\displaystyle 2R$};

\end{tikzpicture}

\caption{Positions of $K_{\max}$, $\phi^\theta(K_{\max})$, $\phi^\theta(2^l K_{\max})$ and $\textcolor{red}{2R}$ in case $(c)$.} \label{figure:RR}
\end{figure}
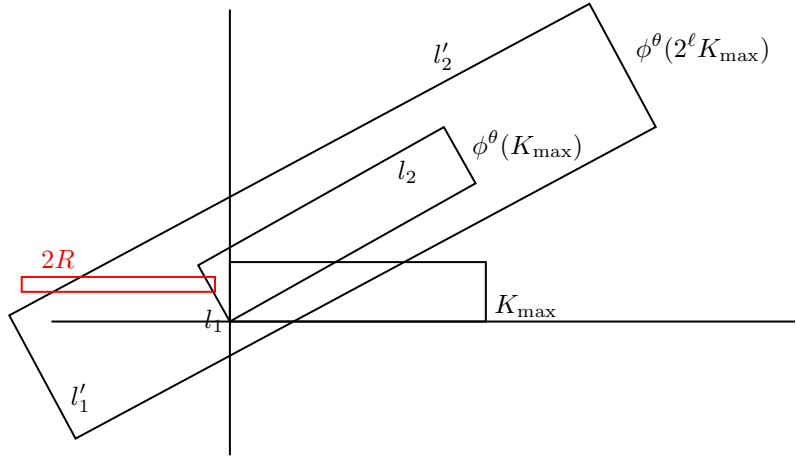

Then by a similar reasoning in (a) and (b), one obtains
\begin{equation*}
|x_2 - x_1| = 2^l\frac{\tilde{x_2}}{|\cos\theta|}
\end{equation*}
where $\tilde{x_2}$ ranges between $-2^l \kappa_1$ and $-2^l \kappa_2 \frac{\cos \theta}{\sin \theta}$. It is trivial that  $2^l\frac{\tilde{x_2}}{|\cos\theta|}$ is a monotonic function of $\tilde{x_2}$. Suppose, without loss of generality, that $0 \leq \theta \leq \frac{\pi}{2}$, and if
\begin{equation*}
-2^l \kappa_1 \geq-2^l \kappa_2 \frac{\cos \theta}{\sin \theta} \iff \frac{\kappa_1}{|\cos \theta|} \leq \frac{\kappa_2}{|\sin\theta|}
\end{equation*}
then 
\begin{equation*}
2R_1 \geq |x_2 - x_1| \geq 2^{l}\frac{\kappa_1}{|\cos \theta|};
\end{equation*}
If otherwise 
\begin{equation*}
-2^l \kappa_1 \leq -2^l \kappa_2 \frac{\cos \theta}{\sin \theta} \iff \frac{\kappa_1}{|\cos \theta|} \geq \frac{\kappa_2}{|\sin\theta|}
\end{equation*}
then 
\begin{equation*}
2R_1 \geq |x_2 - x_1| \geq 2^{l}\frac{\kappa_2}{|\sin \theta|}.
\end{equation*}
\item[(d)]
Suppose that a horizontal side of $R$ intersects $l_2$ and $l_1'$, by the same argument in (c), one deduces that
\begin{equation*}
2R_1 \geq \min(\frac{2^l\kappa_1}{|\cos \theta|}, \frac{2^l\kappa_2}{\sin \theta|}).
\end{equation*}
By summarizing all four cases, one concludes that if at least one horizontal side of $2R$ intersects $\phi^{\theta}(K_{\max})$ and $(\phi^{\theta}(2^lK_{\max}))^c$, then
\begin{equation*}
2 R_1 \geq \min(\frac{2^l\kappa_1}{|\cos \theta|}, \frac{2^l\kappa_2}{\sin \theta|}).
\end{equation*}
And parallel arguments shows that if at least one vertical side of $2R$ intersects $\phi^{\theta}(K_{\max})$ and $(\phi^{\theta}(2^lK_{\max}))^c$, then
\begin{equation*}
2 R_2 \geq \min(\frac{2^l\kappa_1}{|\sin \theta|}, \frac{2^l\kappa_2}{\cos \theta|}).
\end{equation*}
\end{enumerate}
This completes the proof of the observation.
\end{proof}

Based on Observation \ref{R_restrict}, one can assume without loss of generality that (\ref{R_1}) holds and in particular
\begin{equation}
2^{-l}2^{r_1} \geq \frac{1}{2}\min(\frac{K_{\max}^1}{|\cos\theta|}, \frac{K_{\max}^2}{|\sin \theta|}) = \frac{1}{2}\frac{K_{\max}^1}{|\cos\theta|}. \label{eq:r1}
\end{equation}
\begin{remark}
One first notices that the two conditions in (\ref{R_1}) are symmetric in the sense that if we interchange the role of $2^{r_1}$ and $2^{r_2}$, the argument developed below would still go through. 
Furthermore, what we have assumed in (\ref{eq:r1}) is that 
$$
\frac{K^2_{\max}}{|\sin \theta|} \geq \frac{K^1_{\max}}{|\cos \theta|}.
$$
One key feature used by (\ref{eq:r1}) is that for any $K \subseteq K_{\max}$ with $|K|= 2^{k_1} \times 2^{k_2}$,
one has $2^{k_1} \leq K^1_{\max}$, and thus
\begin{equation} \label{condition_k_1}
2^{k_1} \lesssim 2^{-l}2^{r_1}.
\end{equation}
If instead 
$$\frac{K^2_{\max}}{|\sin \theta|} \leq \frac{K^1_{\max}}{|\cos \theta|},$$
then we would proceed the discussion with the following condition on the shapes of the dyadic rectangles involved:
\begin{equation} \label{condition_k_2}
2^{k_2} \lesssim 2^{-l}2^{r_1}.
\end{equation}
It is not difficult to observe the symmetry between (\ref{condition_k_1}) and (\ref{condition_k_2}) in the argument by switching $2^{k_1}$ and $2^{k_2}$. 
\end{remark}

Under the assumption (\ref{eq:r1}), we obtain the following corollary which can be viewed as a quantification and generalization to the example in Section \ref{sec:motivation} (3).

\begin{corollary} \label{cor:r1} 
Suppose that $R$ and $K_{\max}$ satisfy the conditions stated in Observation \ref{R_restrict} and in particular (\ref{eq:r1}). Then for any $K_s \subseteq K_{\max}$ with $|K_s| = 2^{k_1} \times 2^{k_2}$ such that $2^{-l+3} \leq \varepsilon_{K_s}:= \frac{2^{k_2}}{2^{k_1}} \leq 2^{l-3}$, we have
\begin{equation} 2^{k_1}+2^{k_2} \leq \frac{1}{4} 2^{r_1}. \label{eqcor:r1} \end{equation}
\end{corollary}

\begin{proof} As pointed out from the beginning, $\epsilon$ is considered as a small parameter, since the desired estimate is ``trivial'' for $\epsilon \simeq 1$. So by Lemma \ref{l_epsilon}, we can then assume that $l\geq 10$.  From that, \eqref{eq:r1} implies
$$ \min(2^{k_1},2^{k_2}) \leq 2. 2^{-l} 2^{r_1} \leq \frac{1}{20} 2^{r_1}$$
and so with $8.2^{-l} \leq 2^{k_2-k_1} \leq \frac{1}{8}2^{l}$, one deduces that
$$ \max(2^{k_1},2^{k_2})\leq \frac{1}{8} 2^{r_1}$$
and \eqref{eqcor:r1} thus follows.
\end{proof}

Corollary \ref{cor:r1} provides a desirable localization of $\osc_{\Omega}(F \circ \phi)$ in the sense that if $F\circ \phi$ adopts a Haar expansion for some $h_R$ with $R \nsubseteq \widetilde{\phi^\theta(\Omega)}$, then the linear form $\langle \langle F, h_R \rangle h_R\circ \phi, h_{K_s} \rangle$ can be well localized to a dyadic children of $R$ if $K_s$ satisfies the condition specified in Corollary \ref{cor:r1}. 

In order to achieve such localization, we need to restrict to rectangles whose the eccentricity belongs to $[2^{-l+3},2^{l-3}]$.

If $  2^{-l+3} \leq \varepsilon_{K_{\max}} \leq 2^{l-3}$, then there is nothing to do and we set $K_{\max,1}=K_{\max}$.
Assume now that $\varepsilon_{K_{\max}}>2^{l-3}$. We decompose the segment $K^2_{\max}$ by dyadic disjoint segments of length $2^{l-2} |K^1_{\max}|$ (and so larger than $2^{-l+2}|K^1_{\max}|$), which creates a decomposition of $$\displaystyle K^2_{\max}= \bigsqcup_{i=1}^{N} K^2_{\max,i}$$ and also a decomposition of 
$$\displaystyle K_{\max}:= \bigsqcup_{i=1}^{N} K_{\max,i}:= \bigsqcup_{i=1}^{N} K_{\max}^1 \times K_{\max,i}^2.$$ 
Then for every $K=K^1 \times K^2\subset K_{\max}$ with the eccentricity $\varepsilon_{K_{\max,i}} \in [2^{-l+3},2^{l-3}]$, $|K^2| \leq 2^{l-3} |K^1| \leq 2^{l-3} |K_{\max}^1|$. So by the dyadic property of the grid, $K^2$ has to be included in one of the intervals $K_{\max,i}^2$ and so $K$ is included in one of the rectangles $K_{\max,i}$ for some unique $i$. \\
If the initial eccentricity $\varepsilon_{K_{\max}}<2^{-l+3}$, we perform the same decomposition by replacing $K^2_{\max}$ with $K^1_{\max}$.
In all three situations, we have decomposed $K_{\max}$ as a union of disjoint rectangles $(K_{\max,i})_{i}$.

\begin{remark}\label{inclu_submax}
We will call the rectangles $K_{\max,i}$ \textit{sub-maximal rectangles}. 
Each of the sub-maximal rectangles $K_{\max,i}$ has an eccentricity $\varepsilon_{K_{\max,i}} \in[2^{-l+3},2^{l-3}]$ and one has the following property: for every $K=K^1 \times K^2\subset K_{\max}$ with the eccentricity $\varepsilon_{K_{\max,i}} \in[2^{-l+3},2^{l-3}]$, $K$ is included in one and only one of the  sub-maximal rectangles.
\end{remark}

To estimate the remaining cases, we will develop a proof described in Figure \ref{outline_proof}. As indicated in the flow-chart, we have performed localization and decomposition of maximal rectangles in this section. Moreover,  the estimate for $\mathcal{E}_1$ term, which consists of two cases, have been treated in this section as well.  

We will further decompose $F$ into two parts based on its wavelet expansion. The first part $\mathcal{M}$ would generate the first term on the right hand side of (\ref{main_pos_bound}) while the second part $\mathcal{E}_2$ contributes to the second term involving the Sobolev norm of $F$. The notations $\mathcal{M}$ and $\mathcal{E}_i$ for $i = 1,2$ represent for main term and error term respectively. 

\begin{figure}

\centering

\tikzset{every picture/.style={line width=0.75pt}} 
        
\hspace*{-2cm}%
\begin{tikzpicture}[x=0.75pt,y=0.75pt,yscale=-1,xscale=1]

\draw    (276.5,156) -- (276.87,196.17) ;
\draw [shift={(276.88,198.17)}, rotate = 269.48] [color={rgb, 255:red, 0; green, 0; blue, 0 }  ][line width=0.75]    (10.93,-3.29) .. controls (6.95,-1.4) and (3.31,-0.3) .. (0,0) .. controls (3.31,0.3) and (6.95,1.4) .. (10.93,3.29)   ;
\draw    (298.57,268.91) -- (466.57,315.47) ;
\draw [shift={(468.5,316)}, rotate = 195.49] [color={rgb, 255:red, 0; green, 0; blue, 0 }  ][line width=0.75]    (10.93,-3.29) .. controls (6.95,-1.4) and (3.31,-0.3) .. (0,0) .. controls (3.31,0.3) and (6.95,1.4) .. (10.93,3.29)   ;
\draw    (256.5,267) -- (121.99,311.44) ;
\draw [shift={(120.09,312.07)}, rotate = 341.72] [color={rgb, 255:red, 0; green, 0; blue, 0 }  ][line width=0.75]    (10.93,-3.29) .. controls (6.95,-1.4) and (3.31,-0.3) .. (0,0) .. controls (3.31,0.3) and (6.95,1.4) .. (10.93,3.29)   ;
\draw    (513.5,391) -- (544.77,471.14) ;
\draw [shift={(545.5,473)}, rotate = 248.68] [color={rgb, 255:red, 0; green, 0; blue, 0 }  ][line width=0.75]    (10.93,-3.29) .. controls (6.95,-1.4) and (3.31,-0.3) .. (0,0) .. controls (3.31,0.3) and (6.95,1.4) .. (10.93,3.29)   ;
\draw    (494.91,389.38) -- (244.41,467.41) ;
\draw [shift={(242.5,468)}, rotate = 342.7] [color={rgb, 255:red, 0; green, 0; blue, 0 }  ][line width=0.75]    (10.93,-3.29) .. controls (6.95,-1.4) and (3.31,-0.3) .. (0,0) .. controls (3.31,0.3) and (6.95,1.4) .. (10.93,3.29)   ;
\draw    (280.94,55.02) -- (280.94,82.85) ;
\draw [shift={(280.94,84.85)}, rotate = 270] [color={rgb, 255:red, 0; green, 0; blue, 0 }  ][line width=0.75]    (10.93,-3.29) .. controls (6.95,-1.4) and (3.31,-0.3) .. (0,0) .. controls (3.31,0.3) and (6.95,1.4) .. (10.93,3.29)   ;
\draw    (274.5,268) -- (274.5,313) ;
\draw [shift={(274.5,315)}, rotate = 270] [color={rgb, 255:red, 0; green, 0; blue, 0 }  ][line width=0.75]    (10.93,-3.29) .. controls (6.95,-1.4) and (3.31,-0.3) .. (0,0) .. controls (3.31,0.3) and (6.95,1.4) .. (10.93,3.29)   ;

\draw    (244.04,27.22) -- (312.04,27.22) -- (312.04,50.22) -- (244.04,50.22) -- cycle  ;
\draw (257.04,31.62) node [anchor=north west][inner sep=0.75pt]  [font=\footnotesize]  {\hspace{-0.2cm} $\osc_{\Omega } F \circ \phi $};
\draw  [color={rgb, 255:red, 208; green, 2; blue, 27 }  ,draw opacity=1 ]  (299.56,57.3) -- (552.56,57.3) -- (552.56,82.3) -- (299.56,82.3) -- cycle  ;
\draw (312.56,61.3) node [anchor=north west][inner sep=0.75pt]  [color={rgb, 255:red, 208; green, 2; blue, 27 }  ,opacity=1 ] [align=left] {localization onto maximal rectangles };
\draw    (186.94,200.55) -- (380.94,200.55) -- (380.94,263.55) -- (186.94,263.55) -- cycle  ;
\draw (199.94,204.95) node [anchor=north west][inner sep=0.75pt]  [font=\footnotesize]  {$\left(\displaystyle \sum_{l\ \in \mathbb{N}}\sum_{\substack{K_{\max} \in \mathcal{K}_{l}}}\left( osc_{{K}_{\max}} F \circ \phi \right)^{2}\right)^{\frac{1}{2}}$};
\draw    (194.94,90.29) -- (365.94,90.29) -- (365.94,150.29) -- (194.94,150.29) -- cycle  ;
\draw (207.94,94.69) node [anchor=north west][inner sep=0.75pt]  [font=\footnotesize]  {$\left(\displaystyle \sum_{K_{\max} \subseteq \Omega }\left( osc_{{K}_{\max}} F \circ \phi \right)^{2}\right)^{\frac{1}{2}}$};
\draw  [color={rgb, 255:red, 208; green, 2; blue, 27 }  ,draw opacity=1 ]  (303.09,162.94) -- (558.82,162.94) -- (558.82,187.94) -- (303.09,187.94) -- cycle  ;
\draw (316.16,166.94) node [anchor=north west][inner sep=0.75pt]  [color={rgb, 255:red, 208; green, 2; blue, 27 }  ,opacity=1 ,xslant=-0.03] [align=left] {decomposition of maximal rectangles};
\draw (181.26,193.44) node [anchor=north west][inner sep=0.75pt]    {$$};
\draw  [fill={rgb, 255:red, 248; green, 231; blue, 28 }  ,fill opacity=0.16 ]  (0,321.53) -- (225.81,321.53) -- (225.81,390.53) -- (0,390.53) -- cycle  ;
\draw (0,325.93) node [anchor=north west][inner sep=0.75pt]  [font=\scriptsize]  {$\quad \left(\displaystyle \sum_{l\in \mathbb{N}} \sum_{K_{\max} \in \mathcal{K}_{l}} \sum_{ \substack{
K\subseteq K_{\max}\\
\varepsilon _{K} \leq 2^{-l+3} }} |\langle F \circ \phi ,h_{K} \rangle |^{2}\right)^{\frac{1}{2}}$};
\draw  [color={rgb, 255:red, 208; green, 2; blue, 27 }  ,draw opacity=1 ]  (364.71,270.98) -- (686.71,270.98) -- (686.71,295.98) -- (364.71,295.98) -- cycle  ;
\draw (347.71,275.98) node [anchor=north west][inner sep=0.75pt]  [color={rgb, 255:red, 208; green, 2; blue, 27 }  ,opacity=1 ] [align=left] { \qquad decompositions of wavelets relative to eccentricities};
\draw    (460.29,321.53) -- (750.29,321.53) -- (750.29,390.53) -- (460.29,390.53) -- cycle  ;
\draw (460.29,325.62) node [anchor=north west][inner sep=0.75pt]  [font=\scriptsize]  {$\quad \left(\displaystyle \sum_{l\in \mathbb{N}} \sum_{K_{\max} \in \mathcal{K}_{l}}\sum_{\substack{K_{\max,i} \\ \subseteq K_{\max}}} \sum_{ \substack{K\subseteq K_{\max,i}\\
2^{-l+3} < \varepsilon _{K} < 2^{l-3}}}
 |\langle F \circ \phi ,h_{K} \rangle |^{2}\right)^{\frac{1}{2}}$};
\draw  [fill={rgb, 255:red, 248; green, 231; blue, 28 }  ,fill opacity=0.16 ]  (230.21,321.86) -- (455.21,321.86) -- (455.21,390.53) -- (230.21,390.53) -- cycle  ;
\draw (230.21,326.26) node [anchor=north west][inner sep=0.75pt]  [font=\scriptsize]  {$\quad \left(\displaystyle \sum_{l\in \mathbb{N}}\sum_{K_{\max} \in \mathcal{K}_{l}} \sum_{ \substack{
K\subseteq K_{\max}\\
\varepsilon _{K} \geq 2^{l-3}}} 
|\langle F \circ \phi ,h_{K} \rangle |^{2}\right)^{\frac{1}{2}}$};
\draw  [fill={rgb, 255:red, 80; green, 227; blue, 194 }  ,fill opacity=0.16 ]  (0,474.11) -- (392.64,474.11) -- (392.64,550.11) -- (0,550.11) -- cycle  ;
\draw (-10,478.51) node [anchor=north west][inner sep=0.75pt]  [font=\tiny]  {$\qquad \left(\displaystyle \sum_{l\in \mathbb{N}} \sum_{K_{\max} \in \mathcal{K}_{l}}\sum_{\substack{K_{\max,i} \\ \subseteq K_{\max}}} \sum_{ \substack{
K\subseteq K_{\max,i}\\
2^{-l+3} < \varepsilon _{K} < 2^{l-3}}}
\Bigg| \displaystyle \sum_{ \substack{R\in \mathbb{D} ^{\beta ( K_{\max,i})}\\
R\ \subseteq \widetilde{\widetilde{\phi (\Omega)}}}}
\langle \ F,h_{R} \rangle \langle h_{R} \circ \phi ,h_{K} \rangle \Bigg| ^{2}\right)^{\frac{1}{2}}$};
\draw  [color={rgb, 255:red, 0; green, 0; blue, 0 }  ,draw opacity=1 ][fill={rgb, 255:red, 248; green, 231; blue, 28 }  ,fill opacity=0.16 ]  (393.18,474.11) -- (775.18,474.11) -- (775.18,550.11) -- (393.18,550.11) -- cycle  ;
\draw (371.18,479.88) node [anchor=north west][inner sep=0.75pt]  [font=\tiny]  {$\qquad \left(\displaystyle \sum_{l\in \mathbb{N}} \sum_{K_{\max} \in \mathcal{K}_{l}}\sum_{\substack{K_{\max,i} \\ \subseteq K_{\max}}}  \sum_{\substack{
K\subseteq K_{\max,i}\\
2^{-l+3} < \varepsilon _{K} < 2^{l-3} }}
\Bigg| 
\displaystyle \sum_{\substack{R\in \mathbb{D} ^{\beta ( K_{\max,i})}\\ R\ \nsubseteq \widetilde{\widetilde{\phi (\Omega)}}}}
\langle \ F,h_{R} \rangle \langle h_{R} \circ \phi ,h_{K} \rangle \Bigg| ^{2}\right)^{\frac{1}{2}}$};
\draw  [color={rgb, 255:red, 208; green, 2; blue, 27 }  ,draw opacity=1 ]  (405.4,410.52) -- (689.4,410.52) -- (689.4,440.52) -- (405.4,440.52) -- cycle  ;
\draw (418.4,420.52) node [anchor=north west][inner sep=0.75pt]  [color={rgb, 255:red, 208; green, 2; blue, 27 }  ,opacity=1 ] [align=left] {decomposition of wavelet expansion for $\displaystyle f$};
 Text Node
\draw (115,324.4) node [anchor=north west][inner sep=0.75pt]  [font=\footnotesize,color={rgb, 255:red, 80; green, 227; blue, 194 }  ,opacity=1 ]  {$ \begin{array}{l}
\textcolor[rgb]{0.56,0.07,1}{\mathcal{E}_{1}}\\
\end{array}$};
\draw (330,325.4) node [anchor=north west][inner sep=0.75pt]  [font=\footnotesize,color={rgb, 255:red, 80; green, 227; blue, 194 }  ,opacity=1 ]  {$ \begin{array}{l}
\mathcal{\textcolor[rgb]{0.56,0.07,1}{E}}\textcolor[rgb]{0.56,0.07,1}{_{1}}\\
\end{array}$};
\draw (512,478.4) node [anchor=north west][inner sep=0.75pt]  [font=\footnotesize]  {$\textcolor[rgb]{0.56,0.07,1}{\mathcal{E}_{2}}$};
\draw (180,478.4) node [anchor=north west][inner sep=0.75pt]  [font=\footnotesize,color={rgb, 255:red, 80; green, 227; blue, 194 }  ,opacity=1 ]  {$ \begin{array}{l}
\mathcal{\textcolor[rgb]{0.56,0.07,1}{M}}\textcolor[rgb]{0.56,0.07,1}{_{}}\\
\end{array}$};

\end{tikzpicture}

\caption{Outline for Proof of Theorem \ref{main_thm_pos}. Following this decomposition, the study of $\osc_{\Omega}(F \circ \phi)$ is reduced to the control of the four coloured quantities. The two terms $\mathcal{E}_1$ have been estimated in \eqref{error1} and \eqref{error2}. The next section is devoted to the study of  $\mathcal{M}$ and $\mathcal{E}_2$.} \label{outline_proof}

\end{figure}
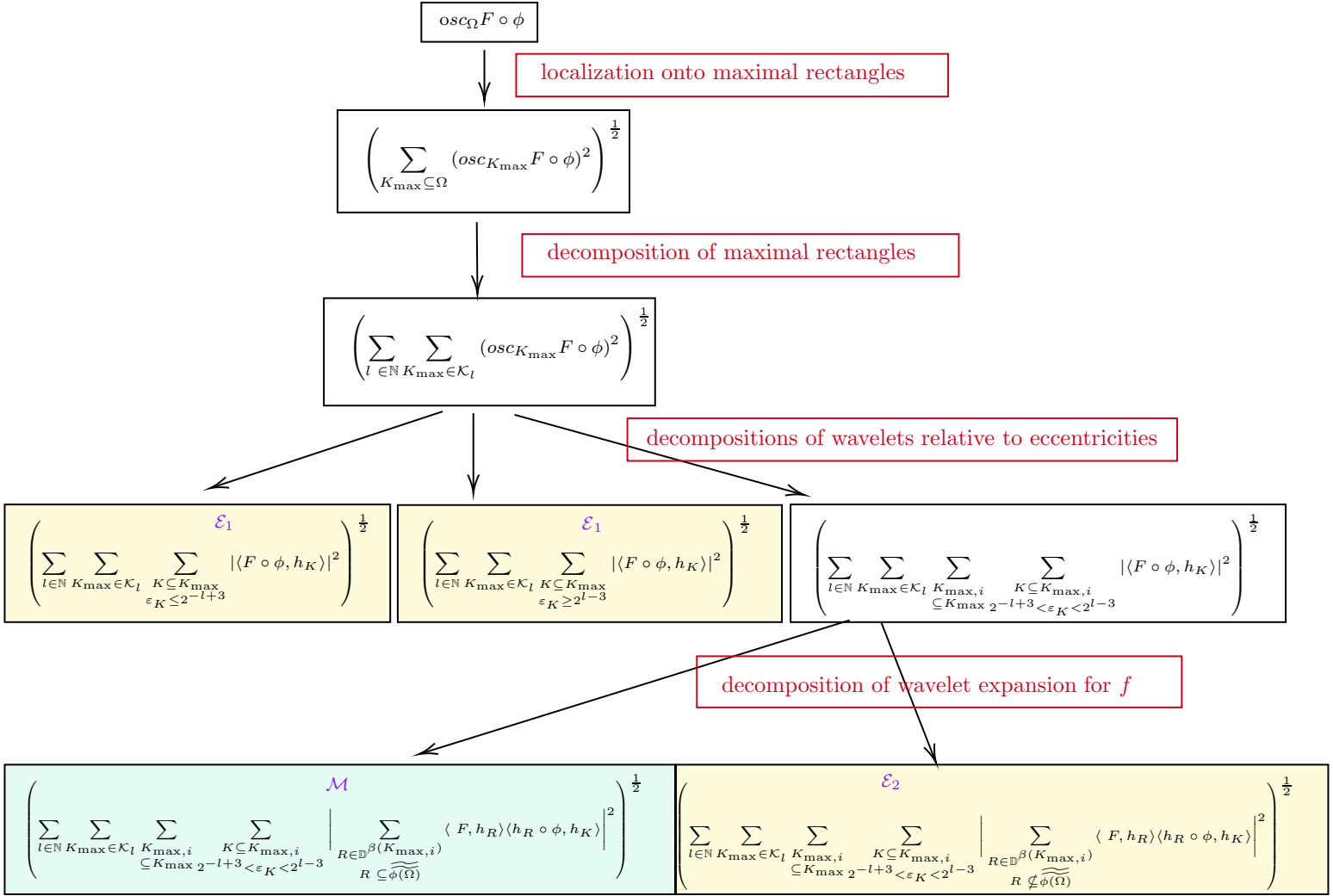


\section{Proof of Theorem \ref{main_thm_pos} - Main Term}\label{sec:main}

It remains to prove that
\begin{equation} \label{main_thm_intermediate}
\left(\sum_{l \in \mathbb{N}}\sum_{\substack{K_{\max} \in \mathbb{D}^{\alpha} \\K_{\max} \subseteq \Omega \\ K_{\max} \in \mathcal{K}_l}}\sum_{\substack{K_{\max,i} \\ \subseteq K_{\max}}} \sum_{\substack{K \in \mathbb{D}^{\alpha} \\ K\subseteq K_{\max,i} \\ 2^{-l+3} < \varepsilon_K < 2^{l-3} }}|\langle F \circ \phi, h_K\rangle|^2\right)^{\frac{1}{2}} \lesssim \epsilon^{-1}\log \epsilon^{-1}|\Omega|^{\frac{1}{2}}\|F\|_{\BMO} + \epsilon^{-\tilde{\delta}} |\Omega|^{\frac{1}{2}}\|F\|_{W^{s,p}}.
\end{equation} 
for some $\tilde{\delta} := \tilde{\delta}(s,p) >0 $ to be determined later.

\begin{notation}
For the sake of simplicity, we would omit the indication that $K_{\max}$, $K_{\max,i}$ and $K$ are dyadic rectangles in the dyadic grid $\mathbb{D}^{\alpha}$.
\end{notation}

We recall that the estimate of term $\mathcal{E}_1$ exploits purely the regularity of $F$ and holds uniformly with respect to the angle of rotation $\theta$. To study the left hand side of (\ref{main_thm_intermediate}), we would take advantage of both the analytic and geometric properties of the interaction of $ F \circ \phi$ with the wavelets $h_K $, which have been discussed in Section 5 when $F$ is a single smooth wavelet or a sum of smooth wavelets of fixed scale. In those simple cases, geometric properties can be investigated more thoroughly. When $F$ is a generic function, it is difficult to understand the geometry of $F \circ \phi$ and the wavelets directly. However, we can expand $F$ in its wavelet expansion to create a similar setting as in Section 5. In particular, we choose a dyadic grid $\mathbb{D}^{\beta}$ and rewrite 
\begin{equation*}
F = \sum_{R \in \mathbb{D}^{\beta}} \langle F,h_R \rangle h_R.
\end{equation*}
Due to Remark (\ref{haar_composition}), we can further rewrite
\begin{equation*}
F \circ \phi = \sum_{R \in \mathbb{D}^{\beta}} \langle F,h_R \rangle h_R \circ \phi
\end{equation*}
so that 
\begin{equation} \label{mot_wavelet_expansion}
\langle F \circ \phi, h_K\rangle = \sum_{R \in \mathbb{D}^{\beta}} \langle F,h_R \rangle \langle h_R \circ \phi, h_K\rangle.
\end{equation}
The right hand side of (\ref{mot_wavelet_expansion}) is more desirable in the sense that we separate the interaction of $F \circ \phi$ with $h_K$ into two parts (as indicated in Figure \ref{fig:strategy_intermediate}):
\begin{enumerate}
\item[(a)]
we study the interaction of $F$ with a single wavelet $h_R$, which allows the analytic implication on the estimate of $\langle F, h_R \rangle $;
\item[(b)] 
we explore the geometric implication from the interaction of wavelets with other rotated wavelets $\langle h_R \circ \phi, h_K \rangle$.
\end{enumerate}
\begin{figure}
\centering

\tikzset{every picture/.style={line width=0.75pt}} 

\begin{tikzpicture}[x=0.75pt,y=0.75pt,yscale=-1,xscale=1]

\draw    (200,201) -- (272.54,268.64) ;
\draw [shift={(274,270)}, rotate = 223] [color={rgb, 255:red, 0; green, 0; blue, 0 }  ][line width=0.75]    (10.93,-3.29) .. controls (6.95,-1.4) and (3.31,-0.3) .. (0,0) .. controls (3.31,0.3) and (6.95,1.4) .. (10.93,3.29)   ;
\draw [color={rgb, 255:red, 208; green, 2; blue, 27 }  ,draw opacity=1 ]   (309,270) -- (369.68,200.51) ;
\draw [shift={(371,199)}, rotate = 491.13] [color={rgb, 255:red, 208; green, 2; blue, 27 }  ,draw opacity=1 ][line width=0.75]    (10.93,-3.29) .. controls (6.95,-1.4) and (3.31,-0.3) .. (0,0) .. controls (3.31,0.3) and (6.95,1.4) .. (10.93,3.29)   ;
\draw [color={rgb, 255:red, 208; green, 2; blue, 27 }  ,draw opacity=1 ] [dash pattern={on 0.84pt off 2.51pt}]  (209,190) .. controls (319.4,246.64) and (357.92,200.06) .. (366.13,185.63) ;
\draw [shift={(367,184)}, rotate = 476.57] [color={rgb, 255:red, 208; green, 2; blue, 27 }  ,draw opacity=1 ][line width=0.75]    (10.93,-3.29) .. controls (6.95,-1.4) and (3.31,-0.3) .. (0,0) .. controls (3.31,0.3) and (6.95,1.4) .. (10.93,3.29)   ;

\draw (381,174.4) node [anchor=north west][inner sep=0.75pt]    {$h_{K}$};
\draw (283,264.4) node [anchor=north west][inner sep=0.75pt]    {$h_{R}$};
\draw (198,173.4) node [anchor=north west][inner sep=0.75pt]    {$F$};
\draw (322,216.4) node [anchor=north west][inner sep=0.75pt]  [color={rgb, 255:red, 208; green, 2; blue, 27 }  ,opacity=1 ]  {$\circ \phi $};
\draw (122,233) node [anchor=north west][inner sep=0.75pt]   [align=left] {Analytic properties};
\draw (351,237.5) node [anchor=north west][inner sep=0.75pt]   [align=left] {Geometric properties};

\end{tikzpicture}
\caption{Strategy for the estimates of $\mathcal{M}$ and $\mathcal{E}_2$.} \label{fig:strategy_intermediate}

\end{figure}
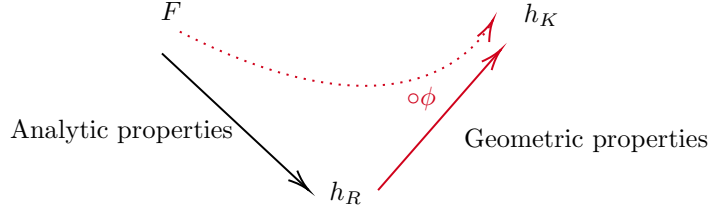

Then we can rewrite the left hand side of (\ref{main_thm_intermediate}) as
\begin{equation} \label{trivial_main_total}
\left(\sum_{l \in \mathbb{N}}\ \sum_{\substack{K_{\max} \subseteq \Omega \\ K_{\max} \in \mathcal{K}_l}} \ \sum_{\substack{K_{\max,i} \\ \subseteq K_{\max}}} \sum_{\substack{K\subseteq K_{\max,i} \\ 2^{-l+3} < \varepsilon_K < 2^{l-3} }}\left|\sum_{R\in \mathbb{D}^{\beta}} \langle F, h_R\rangle \langle h_R \circ \phi, h_K\rangle\right|^2\right)^{\frac{1}{2}}. 
\end{equation}

Inspired by the discussion in Section 5, we can separate into two cases when the wavelets $h_R$ are supported in some enlargement of $\phi(\Omega)$ (corresponding to the blue rectangle in Figure \ref{figure:motivation}) and when the supports of the wavelets lie outside the enlargement of $\phi(\Omega)$ (which correspond to the purple and red rectangles in Figure \ref{figure:motivation}). The first case allows a straightforward estimate which relies heavily on the geometric property while the second case needs a more careful treatment and both the analytic and geometric aspects would be necessary. The rest of this section would be devoted to the first case and we would focus on the more difficult second case in the next section.

We now define enlarged subsets relatively to $\epsilon$ of $\phi(\Omega)$ by
\begin{equation*} 
\widetilde{\phi(\Omega)} := \left\{(x,y)\in\BBR^2,\ M M^{\theta}(\mathbbm{1}_{\phi(\Omega)}) > \epsilon \right\},
\end{equation*}
and
\begin{equation} \label{eq:Omegatildetilde}
\widetilde{\widetilde{\phi(\Omega)}} := \left\{(x,y)\in\BBR^2,\ M M^{\theta}(\mathbbm{1}_{\widetilde{\phi(\Omega)}}) > \epsilon \right\},
\end{equation}
where $MM^\theta$ is the rotated version of the strong biparameter maximal function: for $h\in L^1_{loc}(\BBR^2)$
\begin{align*}
MM^{\theta}(h)(x,y) & := \sup_{K: (x,y) \in \phi^{\theta}(K)}\frac{1}{|\phi^{\theta}(K)|}\int_{\phi^{\theta}(K)} |h(t,s)| dtds  \\
& = \sup_{K: \phi^{-\theta}(x,y) \in K}\frac{1}{|K|} \int_{K} |h \circ \phi^\theta (\tilde{t}, \tilde{s})|d\tilde{t}d\tilde{s} 
\\
& = MM[h \circ \phi^\theta]( \phi^{-\theta}(x,y)). 
\end{align*}
In particular using that $MM: L(1+\log^{+})L \rightarrow L^{1,\infty}$ (see \cite{JMZ,CF}), we easily deduce that for $0 < \epsilon < e^{-1}$,
\begin{equation} \label{eq:first_enlargement}
|\widetilde{\phi(\Omega)}| \lesssim  \int \frac{\mathbbm{1}_{\phi(\Omega)}(x,y)}{\epsilon}(1+ \log^{+} \frac{\mathbbm{1}_{\phi(\Omega)}(x,y)}{\epsilon}) dxdy = \int \frac{\mathbbm{1}_{\phi(\Omega)}(x,y)}{\epsilon}(1+\log^{+} \epsilon^{-1})dxdy \lesssim \epsilon^{-1}\log \epsilon^{-1}|\phi(\Omega)|
\end{equation}
and by the same argument,
\begin{equation}
|\widetilde{\widetilde{\phi(\Omega)}}|\lesssim \epsilon^{-1}\log \epsilon^{-1} |\widetilde{\phi(\Omega)}|.
\label{eq:double_enlargement}
\end{equation}
By combining (\ref{eq:first_enlargement}) and (\ref{eq:double_enlargement}) and using the assumption that $\phi$ is measure-preserving, 
\begin{equation}
|\widetilde{\widetilde{\phi(\Omega)}}| \lesssim \epsilon^{-2} (\log \epsilon^{-1})^2|\phi(\Omega)| = \epsilon^{-2} (\log \epsilon^{-1})^2 |\Omega|.
\end{equation}
Then we would like to decompose (\ref{trivial_main_total}) depending on the supports of the wavelets $h_R$ into the cases when $R \subseteq \widetilde{\widetilde{\phi(\Omega)}}$ and $R \nsubseteq \widetilde{\widetilde{\phi(\Omega)}}$. However, before performing such decomposition, we would like to rewrite (\ref{trivial_main_total}) with a more careful choice of dyadic grids when we expand $F \circ \phi$ in term of the Haar basis.

\medskip

Instead of picking an arbitrary dyadic grid $\mathbb{D}^{\beta}$ as in (\ref{trivial_main_total}), we choose the dyadic grid more carefully so that the error term estimate could be simplified significantly. In particular, for each fixed $K$, there would be one and only one $K_{\max}$ such that $K \subseteq K_{\max}$ due to the maximality of $K_{\max}$. Moreover, for $K$ satisfying the eccentricity condition $2^{-l+3} \leq \varepsilon_{K} \leq 2^{l-3}$ , $K$ is uniquely contained in the sub-maximal rectangle $K_{\max,i} \subseteq K_{\max}$ by Remark \ref{inclu_submax},
which guarantees a disjoint decomposition of $K$ with respect to its inclusion to its correspondent sub-maximal rectangle.
Moreover, for every $K_{\max,i}$, let $\beta(\phi(K_{\max,i}))$ represent the choice of a particular dyadic grid $\mathbb{D}^{\beta(\phi(K_{\max,i}))}$ such that there exists a dyadic rectangle $Q \in \mathbb{D}^{\beta(\phi(K_{\max,i}))}$ satisfying the following two conditions (see Lemma \ref{lemma:grid}): 
\begin{enumerate}
\item[(a)]
\begin{equation}
\phi(K_{\max,i}) \subseteq Q=Q_1\times Q_2; \label{eq:Q}
\end{equation}
\item[(b)]
Let $|K_{\max,i}| = K_{\max,i}^1 \times K_{\max,i}^2$ denote the horizontal and vertical side lengths of the maximal rectangle $K_{\max,i}$. Then
\begin{equation} \label{choice_grid}
\begin{split}
& Q_1 \leq 2( K_{\max,i}^1|\cos \theta| + K_{\max,i}^2 |\sin \theta|) \\
& Q_2 \leq 2( K_{\max,i}^1|\sin \theta| + K_{\max,i}^2 |\cos \theta|).
\end{split}
\end{equation}
\end{enumerate}

\begin{remark} \label{rem:62bis}
This choice of such dyadic grid and such rectangle $Q$ is important, since it will allow us to quantify a situation where we have perfect cancellation in the inner product $\langle h_R \circ \phi, h_K\rangle$. In particular, the projection on the $x$-variable ${\mathbb P}_x [\phi(K_{\max,i})]$ is an interval of length $\leq K^1_{\max,i}+K^2_{\max,i} \leq \frac{1}{4} 2^{r_1}$ (due to Corollary \ref{cor:r1}). Since it is included in $Q_1$  with \eqref{eq:Q}, one deduces that $|Q_1| \leq \frac{1}{2} 2^{r_1}$. By the choice of the dyadic grid (through the parameter $\beta(K_{\max,i})$), $R$ and $Q$ belong to the same dyadic grid. Consequently we deduce that if $R \cap \phi(K_{\max,i}) \neq \emptyset$ and $R \nsubseteq \widetilde{\phi(\Omega)}$ then
\begin{equation}  {\mathbb P}_x [\phi(K_{\max,i})] \subset R_1^{left} \qquad \textrm{or} \qquad {\mathbb P}_x [\phi(K_{\max,i})] \subset R_1^{right}. \label{eq:R1} 
\end{equation}
\end{remark}

Now for each fixed $K_{\max,i}$, we can rewrite the inner sum of (\ref{trivial_main_total}) in terms of the Haar wavelet expression of $F$ with the dyadic grid $\mathbb{D}^{\beta(\phi(K_{\max,i}))}$ so that
\begin{align*} \label{fine_decomp}
& \left(\sum_{l \in \mathbb{N}} \ \sum_{\substack{K_{\max} \subseteq \Omega \\ K_{\max} \in \mathcal{K}_l}} \  \sum_{\substack{K_{\max,i} \\ \subseteq K_{\max}}}\sum_{\substack{K \subseteq K_{\max,i} \\ 2^{-l+3} < \varepsilon_K < 2^{l-3}}} \Bigg| \sum_{\substack{R \in \mathbb{D}^{\beta(\phi(K_{\max,i}))} \\ R \subseteq \widetilde{\widetilde{\phi(\Omega)}}}} \langle F, h_R \rangle \langle h_R \circ \phi, h_K \rangle \Bigg|^2\right)^{\frac{1}{2}}
\end{align*}
We further apply the decomposition of the wavelets depending their supports and derive the terms $\mathcal{M}$ and $\mathcal{E}_2$ in Figure \ref{outline_proof}. More precisely,

\begin{align*}
 &\left(\sum_{l \in \mathbb{N}} \ \sum_{\substack{K_{\max} \subseteq \Omega \\ K_{\max} \in \mathcal{K}_l}} \  \sum_{\substack{K_{\max,i} \\ \subseteq K_{\max}}}\sum_{\substack{K \subseteq K_{\max,i} \\ 2^{-l+3} < \varepsilon_K < 2^{l-3}}} \Bigg| \sum_{\substack{R \in \mathbb{D}^{\beta(\phi(K_{\max,i}))} \\ R \subseteq \widetilde{\widetilde{\phi(\Omega)}}}} \langle F, h_R \rangle \langle h_R \circ \phi, h_K \rangle + \sum_{\substack{R \in \mathbb{D}^{\beta(\phi(K_{\max,i}))} \\ R \nsubseteq \widetilde{\widetilde{\phi(\Omega)}}}} \langle F, h_R \rangle \langle h_R \circ \phi, h_K \rangle \Bigg|^2\right)^{\frac{1}{2}} \nonumber\\
 \leq & \left(\sum_{l \in \mathbb{N}} \ \sum_{\substack{K_{\max} \subseteq \Omega \\ K_{\max} \in \mathcal{K}_l}} \ \sum_{\substack{K_{\max,i} \\ \subseteq K_{\max}}} \sum_{\substack{K \subseteq K_{\max,i} \\ 2^{-l+3} < \varepsilon_K < 2^{l-3}}} \Bigg| \sum_{\substack{R \in \mathbb{D}^{\beta(\phi(K_{\max,i}))} \\ R \subseteq \widetilde{\widetilde{\phi(\Omega)}}}} \langle F, h_R \rangle \langle h_R \circ \phi, h_K \rangle \Bigg|^2\right)^{\frac{1}{2}} + \nonumber\\
 & \left( \sum_{l \in \mathbb{N}} \ \sum_{\substack{K_{\max} \subseteq \Omega \\ K_{\max} \in \mathcal{K}_l}} \  \sum_{\substack{K_{\max,i} \\ \subseteq K_{\max}}}\sum_{\substack{K \subseteq K_{\max,i} \\ 2^{-l+3} < \varepsilon_K < 2^{l-3}}} \Bigg| \sum_{\substack{R \in \mathbb{D}^{\beta(\phi(K_{\max,i}))} \\ R \nsubseteq \widetilde{\widetilde{\phi(\Omega)}}}} \langle F, h_R \rangle \langle h_R \circ \phi, h_K \rangle \Bigg|^2 \right)^{\frac{1}{2}} \nonumber\\
 \leq & \underbrace{\left(\sum_{\substack{K_{\max} \subseteq \Omega \\ }}\sum_{i} \sum_{K \subseteq K_{\max,i}} \Bigg| \sum_{\substack{R \in \mathbb{D}^{\beta(\phi(K_{\max,i}))} \\ R \subseteq \widetilde{\widetilde{\phi(\Omega)}}}} \langle F, h_R \rangle \langle h_R \circ \phi, h_K \rangle \Bigg|^2\right)^{\frac{1}{2}}}_{\mathcal{M}} + \\
& \underbrace{\left( \sum_{l \in \mathbb{N}}\ \sum_{\substack{K_{\max} \subseteq \Omega \\ K_{\max} \in \mathcal{K}_l}}  \sum_{\substack{K_{\max,i} \\ \subseteq K_{\max}}} \sum_{\substack{K \subseteq K_{\max,i} \\ 2^{-l+3} < \varepsilon_K < 2^{l-3}}} \Bigg| \sum_{\substack{R \in \mathbb{D}^{\beta(\phi(K_{\max,i}))} \\ R \nsubseteq \widetilde{\widetilde{\phi(\Omega)}}}} \langle F, h_R \rangle \langle h_R \circ \phi, h_K \rangle \Bigg|^2 \right)^{\frac{1}{2}}}_{\mathcal{E}_2},
\end{align*}
where the second inequality follows from the triangle inequality and the last inequality follows from dropping the condition on the wavelets $K$ and thus adding more terms.
\begin{remark} \label{rem:62}
As pointed out earlier (Remark \ref{rem:62bis}), such decomposition is performed so that the error term would be well-behaved. More precisely, for any $R \in \mathbb{D}^{\beta(\phi(K_{\max,i}))}$, if $|R| = R_1 \times R_2$ satisfies the condition that 
\begin{equation}
R_1 \geq 4(K_{\max,i}^1\sin \theta + K_{\max,i}^2 \cos \theta) 
\end{equation}
and
\begin{equation}
R_2 \geq 4(K_{\max,i}^1\cos \theta + K_{\max,i}^2 \sin \theta )
\end{equation}
then by (\ref{choice_grid}) and the choice of the dyadic grid $\mathbb{D}^{\beta(\phi(K_{\max,i}))}$, we deduce $\phi(K_{\max,i})$ is contained in a dyadic son of $R$ so that $h_R$ is constant on $\phi(K_{\max,i})$ and in particular on $\phi(K)$ for any $K \subseteq K_{\max,i}$. Thus
\begin{equation}
\langle h_R \circ \phi, h_K \rangle = \langle h_R, h_K \circ \phi^{-1} \rangle = 0.
\end{equation}
As a consequence of the careful choice of dyadic grid with respect to each maximal rectangle, we avoid the appearance of wavelets for $F$ which are not well-localizable. 
\end{remark}

We would now develop the estimate for the main term $\mathcal{M}$. Let us define
\begin{equation}
F^{\beta} :=\sum_{\substack{R \in \mathbb{D}^{\beta} \\ R \subseteq \widetilde{\widetilde{\phi(\Omega)}}}}\langle F, h_R \rangle h_R.
\end{equation}
We also recall that there are finitely many choices of dyadic grids, let's denoting it by $n$ (which depends only on the dimension and of the systems of shifted dyadic grids). Suppose each dyadic grid corresponds to a parameter denoted by $\beta_j$, then $(\beta_j)_{j=1}^n$ represents the collection of all possible dyadic grids. Then for any $K_{\max,i}$ and correspondent $\beta(\phi(K_{\max,i}))$, $ \beta(\phi(K_{\max,i})) = \beta_{j}$ for some $1 \leq j \leq n$. Let $K_{\max,i} \mapsto \beta_j$ denote such correspondence. We can then group those belonging to the same dyadic grid as follows.
{\fontsize{9.5}{9.5}
\begin{align} \label{M_estimate}
\mathcal{M} \leq & \left\Vert \sum_{j=1}^n  \ \sum_{\substack{K_{\max} \in \mathbb{D}^{\alpha} \\ K_{\max} \subseteq \Omega \\ }} \sum_{\substack{K_{\max,i} \\ \subseteq K_{\max}}} \sum_{K_{\max,i} \mapsto \beta_j}\sum_{\substack{K \in \mathbb{D}^{\alpha} \\ \substack{K \subseteq K_{\max,i} \\  }}} \langle F^{\beta_j} \circ \phi, h_K \rangle  h_K \right\Vert_2 \nonumber\\
\leq &\sum_{j=1}^n   \left\Vert \sum_{\substack{K_{\max} \in \mathbb{D}^{\alpha} \\ K_{\max} \subseteq \Omega \\ }} \sum_{\substack{K_{\max,i} \\ \subseteq K_{\max}}} \sum_{K_{\max,i} \mapsto \beta_j}\sum_{\substack{K \in \mathbb{D}^{\alpha} \\ \substack{K \subseteq K_{\max,i} \\  }}}\langle F^{\beta_j} \circ \phi, h_K \rangle  h_K \right\Vert_2, 
\end{align}}
where for each $1 \leq j \leq n$, we can apply the square function estimate and the fact that $\phi$ preserves the $L^2$-norm to derive
\begin{align*}
 \left\Vert \sum_{\substack{K_{\max} \in \mathbb{D}^{\alpha} \\ K_{\max} \subseteq \Omega \\ }} \sum_{\substack{K_{\max,i} \\ \subseteq K_{\max}}} \sum_{K_{\max,i} \mapsto \beta_j}\sum_{\substack{K \in \mathbb{D}^{\alpha} \\ \substack{K \subseteq K_{\max,i} \\  }}} \langle F^{\beta_j} \circ \phi, h_K \rangle  h_K \right\Vert_2 \lesssim \|F^{\beta_j} \circ \phi\|_2 =  \|F^{\beta_j}\|_2.
\end{align*}
By recalling the definition of $F^{\beta_j}$ and with \eqref{eq:double_enlargement}, we have
\begin{align}
 \|F^{\beta_j}\|_2 = \left(\sum_{\substack{R \in \mathbb{D}^{\beta_j} \\ R \subseteq \widetilde{\widetilde{\phi(\Omega)}}}}|\langle F, h_R \rangle|^2 \right)^{\frac{1}{2}} \leq   \|F\|_{\BMO} |\widetilde{\widetilde{\phi(\Omega)}}|^{\frac{1}{2}} \lesssim \epsilon^{-1} \log \epsilon^{-1}|\Omega|^{\frac{1}{2}}\|F\|_{\BMO}. \label{eq:termI}
\end{align}
Since \eqref{eq:termI} holds for any $j$, we can plug it into (\ref{M_estimate}) and conclude that
\begin{equation} \label{main_term_result}
\mathcal{M} \lesssim \epsilon^{-1} \log \epsilon^{-1} |\Omega|^{\frac{1}{2}}\|F\|_{\BMO}.
\end{equation}
This completes the estimate for the main term and it remains to control the error term $\mathcal{E}_2$ which will be far more technical.

\section{Proof of Theorem \ref{main_thm_pos} - Error Term}

\subsection{Reduction of $\mathcal{E}_2$ and analytic implications} \label{error_analytic}
As commented at the beginning of Section \ref{sec:main} and indicated in Figure \ref{fig:strategy_intermediate}, we need to apply the analytic property - Lemma \ref{lemma:regularity} - for the estimate involving $\langle F, h_R \rangle$ and the geometric property for the estimate involving $\langle h_R \circ \phi, h_K \rangle$. Section \ref{error_analytic} would contribute to the analytic perspective and reduces the estimate of the error term $\mathcal{E}_2$ to a careful study of geometric implications from interactions between wavelets and rotated wavelets. 

We first apply the dualization and study the corresponding multilinear form for $\mathcal{E}_2$. In particular, for some function $\eta \in L^2$ with $\|\eta \|_2 = 1$,
\begin{align}
\|\mathcal{E}_2\|_2 \leq |\langle \mathcal{E}_2,\eta\rangle| = &  \left| \sum_{l \geq 0} \sum_{\substack{K_{\max} \subseteq \Omega \\ K_{\max} \in \mathcal{K}_{l}}} \sum_{\substack{K_{\max,i} \\ \subseteq K_{\max}}}\sum_{\substack{K \subseteq K_{\max,i} \\ 2^{-l+3} <\varepsilon_K < 2^{l-3}}} \sum_{\substack{R \in \mathbb{D}^{\beta(\phi(K_{\max,i}))}\\ R \nsubseteq \widetilde{\phi(\Omega)} \\ }}\langle F, h_R \rangle \langle h_R \circ \phi, h_K \rangle  \langle h_K, \eta \rangle \right| \nonumber\\
\leq &\sum_{l \geq 0} \ \sum_{\substack{K_{\max} \subseteq \Omega \\ K_{\max} \in \mathcal{K}_l}} \ \sum_{\substack{K_{\max,i} \\ \subseteq K_{\max}}}\sum_{\substack{K \subseteq K_{\max,i} \\ 2^{-l+3} <\varepsilon_K < 2^{l-3}}} \sum_{\substack{R \in \mathbb{D}^{\beta(\phi(K_{\max,i}))}\\ R \nsubseteq \widetilde{\phi(\Omega)} \\ }} |\langle F, h_R \rangle| |\langle h_R \circ \phi, h_K \rangle| |\langle h_K, \eta \rangle|.   \nonumber
\end{align}
It suffices to prove the following for every fixed $l$ such that $2^l \geq \epsilon^{-\frac{1}{2}}$ (see Lemma \ref{l_epsilon}), $\mathcal{E}_2^l$ defined as
\begin{equation} \label{goal}
\mathcal{E}_2^l := \sum_{\substack{ K_{\max} \subseteq \Omega \\ K _{\max} \in \mathcal{K}_l}}\sum_{\substack{K_{\max,i} \\ \subseteq K_{\max}}}\sum_{\substack{K \subseteq K_{\text{max,i}} \\ 2^{-l+3} < \varepsilon_K < 2^{l-3}}} \sum_{\substack{R \in \mathbb{D}^{\beta(\phi(K_{\max,i})} \\ R \nsubseteq \widetilde{\phi(\Omega)} }} |\langle F, h_R \rangle| |\langle h_R \circ \phi, h_K \rangle| |\langle h_K, \eta \rangle|.
\end{equation}
satisfies the estimate
\begin{equation}\label{goal_estimate}
\mathcal{E}_2^l  \lesssim 2^{-l\tilde{\delta}}|\Omega|^{\frac{1}{2}} \|F\|_{W^{s,p}}.
\end{equation}
for some $\tilde{\delta} >0$. 
Indeed the series of $2^{-l\tilde{\delta}}$ is convergent and bounded by $\epsilon^{\tilde{\delta}/2}$ (up to some constant), which completes the proof of the estimate for $\mathcal{E}_2$.

First of all, we would like to decompose $\mathcal{E}^l_2$ into the sums where $K \subseteq K_{\max,i}$ and $R$ are of fixed sizes so that both analytic and geometric properties can be exploited. Define $\mathcal{EXP}$ as the set of $(r_1,r_2,k_1,k_2)$ (also depending on $K_{\max,i}$ and $l$ but we will keep this independence implicit in the notation) as follows:
$$ \mathcal{EXP}:= \left\{(r_1,r_2,k_1,k_2),\ 2^{k_1}\leq K_{\max,i}^1,\ 2^{k_2}\leq K_{\max,i}^2,\ 2^{-l+3} \leq 2^{k_2-k_1} \leq 2^{l-3}\right\},$$
Then $\mathcal{E}^l_2$ can be rewritten as
\begin{equation*}
\mathcal{E}^l_2 = \sum_{\substack{ K_{\max} \subseteq \Omega \\ K _{\max} \in \mathcal{K}_l}}\sum_{\substack{K_{\max,i} \\ \subseteq K_{\max}}}\sum_{(r_1, r_2,k_1,k_2) \in \mathcal{EXP}}\sum_{\substack{K \subseteq K_{\max,i} \\ |K| = 2^{k_1} \times 2^{k_2}}} \sum_{\substack{R \in \mathbb{D}^{\beta(\phi(K_{\max,i})} \\ R \nsubseteq \widetilde{\phi(\Omega)} \\ |R| = 2^{r_1} \times 2^{r_2}}} |\langle F, h_R \rangle| |\langle h_R \circ \phi, h_K \rangle| |\langle h_K, \eta \rangle|.
\end{equation*}
By H\"older's inequality, for any $p, p' >1$ with $\frac{1}{p}+\frac{1}{p'} =1$,
{\fontsize{9}{9}
\begin{align}\label{E_2_step_1}
\mathcal{E}^l_2 \leq&
  \sum_{\substack{K_{\max} \in \mathbb{D}^{\alpha} \\ K_{\max} \subseteq \Omega \\ K _{\max} \in \mathcal{K}_l}}\sum_{\substack{K_{\max,i} \\ \subseteq K_{\max}}}\sum_{(r_1, r_2,k_1,k_2) \in \mathcal{EXP}}\sum_{\substack{K \in \mathbb{D}^{\alpha} \\ K \subseteq K_{\max,i} \\ |K| = 2^{k_1} \times 2^{k_2}}} \left(\sum_{\substack{R \in \mathbb{D}^{\beta(\phi(K_{\max,i}))} \\ R \nsubseteq \widetilde{\phi(\Omega)} \\ |R| = 2^{r_1} \times 2^{r_2}}}  |\langle F, h_R \rangle|^p \right)^{\frac{1}{p}} \left(\sum_{\substack{R \in \mathbb{D}^{\beta(\phi(K_{\max,i}))} \\ R \nsubseteq \widetilde{\phi(\Omega)} \\ |R| = 2^{r_1} \times 2^{r_2} }} |\langle h_R \circ \phi, h_K \rangle|^{p'} \right)^{\frac{1}{p'}} |\langle \eta, h_K \rangle|.
\end{align}}
Now we are ready to apply the analytic property which is given by Lemma \ref{lemma:regularity} and the regularity of $F$, which gives
\begin{equation}\label{E_2_analytic}
\left(\sum_{\substack{R \in \mathbb{D}^{\beta(\phi(K_{\max,i}))} \\ R \nsubseteq \widetilde{\phi(\Omega)} \\ |R| = 2^{r_1} \times 2^{r_2}}}  |\langle F, h_R \rangle|^p \right)^{\frac{1}{p}} \lesssim  \|F\|_{W^{s,p}} 2^{r_1(\frac{1}{2}+ \gamma_1)}2^{r_2(\frac{1}{2}+\gamma_2)} 
\end{equation}
for $\gamma_1, \gamma_2 \in (-\delta, \delta)$ with $\delta:= \min(\frac{1}{p}, s-\frac{2}{p})$. By plugging the estimate (\ref{E_2_analytic}) into (\ref{E_2_step_1}), we have
{\fontsize{9}{9}
\begin{align*}
&\|F\|_{W^{s,p}} \sum_{\substack{K_{\max,i} \subseteq \Omega \\ K _{\max} \in \mathcal{K}_l}}\sum_{\substack{K_{\max,i} \\ \subseteq K_{\max}}}\sum_{(r_1, r_2,k_1,k_2) \in \mathcal{EXP}} 2^{r_1(\frac{1}{2} +\gamma_1)+r_2(\frac{1}{2}+\gamma_2)} \sum_{\substack{K \subseteq K_{\max,i} \\ |K| = 2^{k_1} \times 2^{k_2}}} \left(\sum_{\substack{R \in \mathbb{D}^{\beta(\phi(K_{\max,i}))} \\ R \nsubseteq \widetilde{\phi(\Omega)} \\ |R| = 2^{r_1} \times 2^{r_2}}} |\langle h_R \circ \phi, h_K \rangle|^{p'} \right)^{\frac{1}{p'}} |\langle \eta, h_K \rangle| \nonumber\\
\lesssim & \|F\|_{W^{s,p}} \left[\sum_{\substack{K_{\max} \subseteq \Omega \\ K _{\max} \in \mathcal{K}_l}}\sum_{\substack{K_{\max,i} \\ \subseteq K_{\max}}}\sum_{(r_1, r_2,k_1,k_2) \in \mathcal{EXP}} 2^{2r_1(\frac{1}{2}+ \gamma_1)}2^{2r_2(\frac{1}{2}+ \gamma_2)}\sum_{\substack{K \subseteq K_{\max,i} \\ |K| = 2^{k_1} \times 2^{k_2}}} \left(\sum_{\substack{R \in \mathbb{D}^{\beta(\phi(K_{\max,i}))} \\ R \nsubseteq \widetilde{\phi(\Omega)} \\ |R| = 2^{r_1} \times 2^{r_2}}} |\langle h_R \circ \phi, h_K \rangle|^{p'} \right)^{\frac{2}{p'}} \right]^{\frac{1}{2}}  \nonumber \\
&\quad  \quad \quad \quad \cdot \left(\sum_{\substack{K_{\max} \subseteq \Omega \\ K _{\max} \in \mathcal{K}_l}}\sum_{\substack{K_{\max,i} \\ \subseteq K_{\max}}}\sum_{(r_1, r_2,k_1,k_2) \in \mathcal{EXP}}\sum_{\substack{K \subseteq K_{\max,i} \\ |K| = 2^{k_1} \times 2^{k_2}}}|\langle \eta, h_K \rangle|^2 \right)^{\frac{1}{2}}  \nonumber \\
 \lesssim & \|F\|_{W^{s,p}}\left[\sum_{\substack{K_{\max} \subseteq \Omega \\ K_{\max} \in \mathcal{K}_l}} \sum_{\substack{K_{\max,i} \\ \subseteq K_{\max}}}\sum_{(r_1, r_2,k_1,k_2) \in \mathcal{EXP}}2^{2r_1(\frac{1}{2}+ \gamma_1)}2^{2r_2(\frac{1}{2}+ \gamma_2)}\sum_{\substack{K \subseteq K_{\max,i} \\ |K| = 2^{k_1} \times 2^{k_2}}} \left(\sum_{\substack{R \in \mathbb{D}^{\beta(\phi(K_{\max,i}))} \\R \nsubseteq \widetilde{\phi(\Omega)} \\ |R| = 2^{r_1} \times 2^{r_2}}} |\langle h_R \circ \phi, h_K \rangle|^{p'} \right)^{\frac{2}{p'}} \right]^{\frac{1}{2}}, \nonumber
\end{align*}}
where the second inequality follows from Cauchy-Schwartz inequality and the last inequality uses the square function estimate for $\eta$ and the fact that $\|\eta\|_2=1$. 

\subsection{Geometric implications}
It suffices to prove that for any fixed $l \in \mathbb{N}$, $K_{\max}\in {\mathcal{K}}_l$ and $K_{\max,i} \subseteq K_{\max}$, 
\begin{equation} \label{subpart}
\sum_{(r_1, r_2,k_1,k_2) \in \mathcal{EXP}}2^{2r_1(\frac{1}{2}+ \gamma_1)}2^{2r_2(\frac{1}{2}+ \gamma_2)}  \sum_{\substack{K \subseteq K_{\max,i} \\ |K| = 2^{k_1} \times 2^{k_2}}} \left( \sum_{\substack{R \in \mathbb{D}^{\beta(\phi(K_{\max,i}))} \\ R \nsubseteq \widetilde{\phi(\Omega)} \\ |R| = 2^{r_1} \times 2^{r_2}}} |\langle h_R \circ \phi, h_K \rangle|^{p'} \right)^{\frac{2}{p'}} \lesssim 2^{-l\mu}|K_{\max,i}|
\end{equation}
for some $\mu >0$. Indeed, together with Journ\'e lemma (Lemma \ref{lemma:journe}) that will imply 
\begin{align*}
& \sum_{\substack{K_{\max} \subseteq \Omega \\ K_{\max} \in \mathcal{K}_l}}\sum_{\substack{K_{\max,i} \\ \subseteq K_{\max}}}\sum_{(r_1, r_2,k_1,k_2) \in \mathcal{EXP}_i}2^{2r_1(\frac{1}{2}+ \gamma_1)}2^{2r_2(\frac{1}{2}+ \gamma_2)}  \sum_{\substack{K \subseteq K_{\max,i} \\ |K| = 2^{k_1} \times 2^{k_2}}} \left( \sum_{\substack{R \in \mathbb{D}^{\beta(\phi(K_{\max,i}))} \\R \nsubseteq \widetilde{\phi(\Omega)} \\ |R| = 2^{r_1} \times 2^{r_2}}} |\langle h_R \circ \phi, h_K \rangle|^{p'} \right)^{\frac{2}{p'}} \\
& \qquad \lesssim \sum_{\substack{K_{\max} \subseteq \Omega \\ K_{\max} \in \mathcal{K}_l}}\sum_{\substack{K_{\max,i} \\ \subseteq K_{\max}}} 2^{-l\mu}|K_{\max,i}| \lesssim 
\sum_{\substack{K_{\max} \subseteq \Omega \\ K_{\max} \in \mathcal{K}_l}} 2^{-l\mu}|K_{\max}| \\
& \qquad \lesssim \epsilon^{\frac{\nu}{2}}2^{-l(\mu-\nu)}|\Omega|,
\end{align*}
for an arbitrarily small $\nu>0$ with $\nu < \mu$. As a result, one completes the proof as follows: 
\begin{equation} \label{E_2_to_fill}
\mathcal{E}_2 \lesssim \|F\|_{W^{s,p}} \left( \epsilon^{\frac{\nu}{2}}\sum_{l: 2^{-l} \lesssim \epsilon^{\frac{1}{2}}}2^{-l(\mu-\nu)}|\Omega|\right)^{\frac{1}{2}} \lesssim  \|F\|_{W^{s,p}}|\Omega|^{\frac{1}{2}}\left( \epsilon^{\frac{\nu}{2}} \epsilon^{\frac{1}{2}(\mu-\nu)}\right)^{\frac{1}{2}} = \epsilon^{\frac{\mu}{4}}\|F\|_{W^{s,p}}|\Omega|^{\frac{1}{2}}.
\end{equation}

Now we will focus the study on $\Gamma(K_{\max,i},r_1,r_2,k_1,k_2)$ defined as
\begin{equation} \Gamma(K_{\max,i},r_1,r_2,k_1,k_2):= \sum_{\substack{K \subseteq K_{\max,i} \\ |K| = 2^{k_1} \times 2^{k_2}}} \left(\sum_{\substack{R \in \mathbb{D}^{\beta(\phi(K_{\max,i}))} \\R \nsubseteq \widetilde{\phi(\Omega)} \\ |R| = 2^{r_1} \times 2^{r_2}}} |\langle h_R \circ \phi, h_K \rangle|^{p'} \right)^{\frac{2}{p'}} . \label{eq:Gamma} \end{equation}
for a fix rectangle $K_{\max,i}$ and fixed parameters $r_1,r_2,k_1,k_2$. This is indeed the subject which we will give a delicate treatment depending on the geometry of rotations. More precisely, we will study how rotated wavelets produce cancellations when tested on wavelets with sides parallel to the axes, which will generate a more accurate estimate than the trivial estimate with a multilinear Kakeya perspective described in Section 5 (see (\ref{multilinear_Kakeya})). We could give a naive estimate for $\langle h_R \circ \phi, h_K \rangle$ as follows: 
\begin{equation} \label{trivial_Gamma}
|\langle h_R \circ \phi, h_K \rangle| \leq |R|^{-\frac{1}{2}}|K|^{-\frac{1}{2}}|\phi(K) \cap R|.
\end{equation}
Evidently (\ref{trivial_Gamma}) is not optimal for similar reasons discussed in Section 5. Firstly, it does not pay attention to the oscillatory behavior of the wavelets $h_R$ and $h_K$. Moreover, it does not use the geometric implication from the rotation map $\phi$. In other words, (\ref{trivial_Gamma}) would still hold if $h_R$ and $h_K$ are replaced by positive $L^2$-normalized functions supported on $R$ and $K$ respectively, and if $\phi$ is an arbitrary bi-Lipschitz measure-preserving map. 

The rest of the section will be organized as follows. In Subsection \ref{subsec:cancelation}, we will establish two main ingredients  - a precise and uniform estimate for $|\langle h_R \circ \phi, h_K \rangle|$ in Proposition \ref{prop:1} and the mechanism to count the numbers of rectangles $(R,K)$ involved in the double sum $\Gamma(K_{\max,i},r_1,r_2,k_1,k_2)$ (Proposition \ref{prop:number} for the counting of $R$ with $K$ fixed and Proposition \ref{prop_sparse_count} for the counting of $K$ with $R$ fixed first). We will then combine these two ingredients to obtain the main cancellation estimates which are summarized in Table \ref{cancellation_est_table}.
Finally we will apply the cancellation estimates to conclude the discussion of $\Gamma(K_{\max,i},r_1,r_2,k_1,k_2)$ and derive the estimate (\ref{subpart}) in Subsection \ref{subsection:endproof}.

\vskip .25in

\subsubsection{Two main ingredients relying on cancellation} \label{subsec:cancelation}

We will now start the discussion of an important cancellation estimate which employs heavily the oscillation of Haar wavelets and the geometry of the rotation map $\phi$, which was not taken into consideration in the trivial estimate (\ref{trivial_Gamma}). Let $S,T$ denote two dyadic rectangles and $h_S, h_T$ the Haar wavelets corresponding to $S$ and $T$. Heuristically, suppose that 
\begin{equation}\label{mot_cancel}
|\langle h_S \circ \phi, h_T \rangle| \lesssim |S|^{-\frac{1}{2}}|T|^{-\frac{1}{2}} \cdot \text{(shaded area in Figure \ref{fig:intuition_cancel})}. 
\end{equation} 
Clearly (\ref{mot_cancel}) is more optimal than (\ref{trivial_Gamma}) since the shaded area is smaller than the area of $S \cap \phi(T)$ used in the trivial estimate. We will now make the estimate (\ref{mot_cancel}) more precise in Proposition \ref{prop:1}.

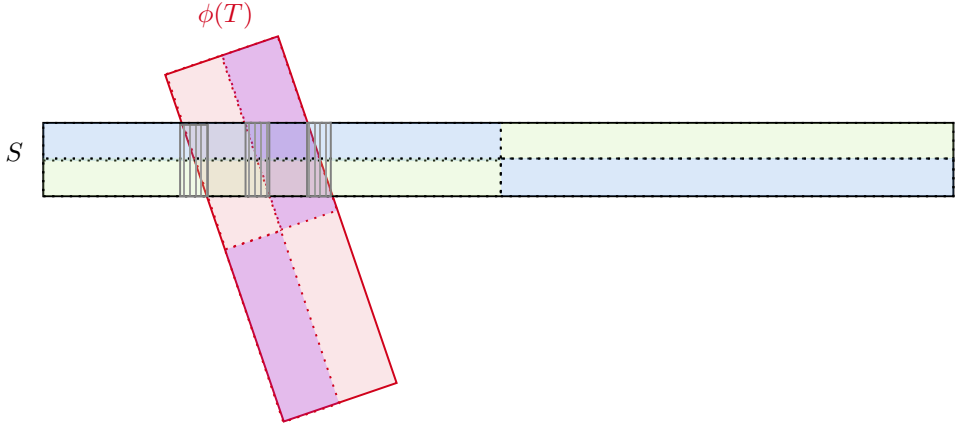
\begin{figure}
\centering

 
\tikzset{
pattern size/.store in=\mcSize, 
pattern size = 5pt,
pattern thickness/.store in=\mcThickness, 
pattern thickness = 0.3pt,
pattern radius/.store in=\mcRadius, 
pattern radius = 1pt}
\makeatletter
\pgfutil@ifundefined{pgf@pattern@name@_sq6xocrjp}{
\pgfdeclarepatternformonly[\mcThickness,\mcSize]{_sq6xocrjp}
{\pgfqpoint{-\mcThickness}{-\mcThickness}}
{\pgfpoint{\mcSize}{\mcSize}}
{\pgfpoint{\mcSize}{\mcSize}}
{
\pgfsetcolor{\tikz@pattern@color}
\pgfsetlinewidth{\mcThickness}
\pgfpathmoveto{\pgfpointorigin}
\pgfpathlineto{\pgfpoint{0}{\mcSize}}
\pgfusepath{stroke}
}}
\makeatother

 
\tikzset{
pattern size/.store in=\mcSize, 
pattern size = 5pt,
pattern thickness/.store in=\mcThickness, 
pattern thickness = 0.3pt,
pattern radius/.store in=\mcRadius, 
pattern radius = 1pt}
\makeatletter
\pgfutil@ifundefined{pgf@pattern@name@_u85uq9yk3}{
\pgfdeclarepatternformonly[\mcThickness,\mcSize]{_u85uq9yk3}
{\pgfqpoint{-\mcThickness}{-\mcThickness}}
{\pgfpoint{\mcSize}{\mcSize}}
{\pgfpoint{\mcSize}{\mcSize}}
{
\pgfsetcolor{\tikz@pattern@color}
\pgfsetlinewidth{\mcThickness}
\pgfpathmoveto{\pgfpointorigin}
\pgfpathlineto{\pgfpoint{0}{\mcSize}}
\pgfusepath{stroke}
}}
\makeatother

 
\tikzset{
pattern size/.store in=\mcSize, 
pattern size = 5pt,
pattern thickness/.store in=\mcThickness, 
pattern thickness = 0.3pt,
pattern radius/.store in=\mcRadius, 
pattern radius = 1pt}
\makeatletter
\pgfutil@ifundefined{pgf@pattern@name@_vrhgvrm6z}{
\pgfdeclarepatternformonly[\mcThickness,\mcSize]{_vrhgvrm6z}
{\pgfqpoint{-\mcThickness}{-\mcThickness}}
{\pgfpoint{\mcSize}{\mcSize}}
{\pgfpoint{\mcSize}{\mcSize}}
{
\pgfsetcolor{\tikz@pattern@color}
\pgfsetlinewidth{\mcThickness}
\pgfpathmoveto{\pgfpointorigin}
\pgfpathlineto{\pgfpoint{0}{\mcSize}}
\pgfusepath{stroke}
}}
\makeatother
\tikzset{every picture/.style={line width=0.75pt}} 

\begin{tikzpicture}[x=0.75pt,y=0.75pt,yscale=-1,xscale=1]

\draw  [color={rgb, 255:red, 208; green, 2; blue, 27 }  ,draw opacity=1 ][fill={rgb, 255:red, 208; green, 2; blue, 27 }  ,fill opacity=0.1 ] (236.89,39.31) -- (296.69,214.26) -- (239.79,233.71) -- (180,58.75) -- cycle ;
\draw  [color={rgb, 255:red, 208; green, 2; blue, 27 }  ,draw opacity=1 ][dash pattern={on 0.84pt off 2.51pt}] (180.52,58.36) -- (208.74,48.79) -- (238.71,137.19) -- (210.5,146.75) -- cycle ;
\draw  [color={rgb, 255:red, 208; green, 2; blue, 27 }  ,draw opacity=1 ][fill={rgb, 255:red, 144; green, 19; blue, 254 }  ,fill opacity=0.2 ][dash pattern={on 0.84pt off 2.51pt}] (208.26,49.02) -- (236.68,39.31) -- (266.7,127.22) -- (238.27,136.93) -- cycle ;
\draw  [color={rgb, 255:red, 208; green, 2; blue, 27 }  ,draw opacity=1 ][fill={rgb, 255:red, 144; green, 19; blue, 254 }  ,fill opacity=0.2 ][dash pattern={on 0.84pt off 2.51pt}] (210.12,147.04) -- (238.34,137.51) -- (267.75,224.54) -- (239.53,234.08) -- cycle ;
\draw   (118.5,83) -- (577.5,83) -- (577.5,120) -- (118.5,120) -- cycle ;
\draw  [dash pattern={on 0.84pt off 2.51pt}] (118.5,101) -- (349.25,101) -- (349.25,120) -- (118.5,120) -- cycle ;
\draw  [fill={rgb, 255:red, 74; green, 144; blue, 226 }  ,fill opacity=0.2 ][dash pattern={on 0.84pt off 2.51pt}] (118.5,83) -- (349.25,83) -- (349.25,102) -- (118.5,102) -- cycle ;
\draw  [dash pattern={on 0.84pt off 2.51pt}] (349.25,101) -- (577.5,101) -- (577.5,120) -- (349.25,120) -- cycle ;
\draw  [fill={rgb, 255:red, 184; green, 233; blue, 134 }  ,fill opacity=0.2 ][dash pattern={on 0.84pt off 2.51pt}] (349.25,83) -- (577.5,83) -- (577.5,101) -- (349.25,101) -- cycle ;
\draw  [fill={rgb, 255:red, 184; green, 233; blue, 134 }  ,fill opacity=0.2 ][dash pattern={on 0.84pt off 2.51pt}] (118.5,101) -- (349.25,101) -- (349.25,120) -- (118.5,120) -- cycle ;
\draw  [fill={rgb, 255:red, 74; green, 144; blue, 226 }  ,fill opacity=0.2 ][dash pattern={on 0.84pt off 2.51pt}] (349.25,101) -- (577.5,101) -- (577.5,120) -- (349.25,120) -- cycle ;
\draw  [color={rgb, 255:red, 128; green, 128; blue, 128 }  ,draw opacity=1 ][pattern=_sq6xocrjp,pattern size=2.25pt,pattern thickness=0.75pt,pattern radius=0pt, pattern color={rgb, 255:red, 155; green, 155; blue, 155}] (187.5,84) -- (201.5,84) -- (201.5,120) -- (187.5,120) -- cycle ;
\draw  [color={rgb, 255:red, 128; green, 128; blue, 128 }  ,draw opacity=1 ][pattern=_u85uq9yk3,pattern size=2.25pt,pattern thickness=0.75pt,pattern radius=0pt, pattern color={rgb, 255:red, 155; green, 155; blue, 155}] (220.5,83) -- (232.5,83) -- (232.5,120) -- (220.5,120) -- cycle ;
\draw  [color={rgb, 255:red, 128; green, 128; blue, 128 }  ,draw opacity=1 ][pattern=_vrhgvrm6z,pattern size=2.25pt,pattern thickness=0.75pt,pattern radius=0pt, pattern color={rgb, 255:red, 155; green, 155; blue, 155}] (251.5,83) -- (263.5,83) -- (263.5,120) -- (251.5,120) -- cycle ;

\draw (195,20.4) node [anchor=north west][inner sep=0.75pt]  [color={rgb, 255:red, 208; green, 2; blue, 27 }  ,opacity=1 ]  {$\phi ( T)$};
\draw (98,91.4) node [anchor=north west][inner sep=0.75pt]    {$S$};

\end{tikzpicture} \caption{Intuition for the cancellation estimate} \label{fig:intuition_cancel}

\end{figure}
For a fixed rectangle $T = [a_1,a_2] \times [b_1,b_2]$. We consider the following 6 segments\begin{enumerate}
\item 3 vertical segments
$$
l^v_1 := 
\left\{\begin{bmatrix}
\frac{a_1+a_2}{2} \\
\tilde{y}
\end{bmatrix}, \ b_1 \leq \tilde{y} \leq b_2\right\},
\qquad 
l^v_2 := 
\left\{\begin{bmatrix}
a_2 \\
\tilde{y}
\end{bmatrix}, \ b_1 \leq \tilde{y} \leq b_2\right\}
\qquad \textrm{and} \qquad
l^v_{3}:=
\left\{\begin{bmatrix}
a_1 \\
\tilde{y}
\end{bmatrix}, \ b_1 \leq \tilde{y} \leq b_2\right\}
$$ 
\item 3 horizontal segments
$$
l^h_1:=
\left\{\begin{bmatrix}
\tilde{x} \\
\frac{b_1+b_2}{2}
\end{bmatrix},\ a_1 \leq \tilde{x} \leq a_2\right\},
\qquad
l^h_2 :=
\left\{\begin{bmatrix}
\tilde{x} \\
b_2
\end{bmatrix},\ a_1 \leq \tilde{x} \leq a_2\right\}
\qquad \textrm{and} \qquad
l^h_3 :=
\left\{\begin{bmatrix}
\tilde{x} \\
b_1
\end{bmatrix},\ a_1 \leq \tilde{x} \leq a_2\right\}.
$$
\end{enumerate}
and the $9$ ``{\it tops}'' defined as
$$ \tops (T):=\Big\{(a,b),\ a\in\{a_1,\frac{a_1+a_2}{2},a_2\},\ b\in\{b_1,\frac{b_1+b_2}{2},b_2\} \Big\}.$$
We also extend these notions to rotated rectangles.

If needed, we denote by $l^h_i(T)$ or $l^v_i(T)$ these segments associated to the rectangle $T$. And finally we set the 
boundary
$$ \partial T:= l^v_1 \cup l^v_2 \cup l^v_3 \cup l^h_1 \cup l^h_2 \cup l^h_3.$$
This subset is important because one notices that $h_T$ is a step function, which changes values between $|T|^{-\frac{1}{2}}, -|T|^{-\frac{1}{2}} $ and $0$ according to these segments. We can further quantify the step function associated to the rotated rectangle $\phi(T)$ with segments as follows. For $1 \leq i \leq 3$, define 
\begin{equation}
h_{\phi(l^h_i)}(x,y) := |T|^{-\frac{1}{2}} \cdot sgn(y - (x \tan \theta + C_{l^h_i})) \mathbbm{1}_{\phi(T)}(x,y),
\end{equation}
where $y = x \tan \theta + C_{l^h_i}$ parametrizes the extended rotated horizontal segment $l^h_i$ while the characteristic function localizes the step function to the support of $\phi(T)$. Similarly define, for $1 \leq j \leq 3$, 
\begin{equation}
h_{\phi(l^v_j)}(x,y) :=|T|^{-\frac{1}{2}} \cdot sgn(y - (x \cot \theta + C_{l^v_j}))\mathbbm{1}_{\phi(T)}(x,y),
\end{equation}
where $y = x \cot \theta + C_{l^v_j}$ parametrizes the extended rotated vertical segment $l^v_j$.
One simple observation is that $$\langle h_S \circ \phi, h_T \rangle \neq 0$$ if and only if $S$ meets the set $\phi(\partial T)$. We could further decompose $\langle h_S \circ \phi, h_T \rangle$ in terms of whether $S$ intersects with the vertical or horizontal segments of $T$ and more precisely in terms of the inner product between $h_S$ and $h_{\phi(l^v_j)}$ or $h_{\phi(l^h_i)}$.
\begin{definition}
Let $\{ l_j^v(T): 1 \leq j \leq 3 \}$ denote the (extended) vertical segments of a rectangle $T$ and $\{ l_j^h(T): 1 \leq j \leq 3 \}$ the (extended) horizontal segments of $T$. Define
\begin{equation*}
\Delta^v_T(S) := \sum_{j=1}^3|\langle h_S, h_{\phi(l^v_j)} \rangle|
\end{equation*}
and
\begin{equation*}
\Delta^h_T(S) := \sum_{i=1}^3|\langle h_S, h_{\phi(l^h_i)} \rangle|.
\end{equation*}
\end{definition}

We separate the cases when 
\begin{enumerate}
\item
the dyadic rectangle $S$ only intersects with the rotated vertical or horizontal segments of $T$, as indicated by $S = S^1$ in Figure \ref{fig:mot_prop:1} or Figure \ref{fig:mot_prop:1(1)};
\item
the dyadic rectangle $S$ intersects with both rotated and vertical segments of $T$ and the vertical boundaries of $S$ do not intersect with $\phi(T)$, as indicated by $S = S^2$ in Figure \ref{fig:mot_prop:1}. In this case, the decomposition of $\langle h_S \circ \phi, h_T \rangle$ into $\Delta^v_T(S)$ and $\Delta^h_T(S)$ can be achieved and the estimate for each piece will be introduced;
\item
it is also marked in Figure \ref{fig:mot_prop:1} the situation when a wavelet, namely $S^3$, intersects with a top of the rotated rectangle $\phi(T)$, in which case a simpler estimate similar as (\ref{trivial_Gamma}) applies and will be discussed in Remark \ref{remark_top}. 
\end{enumerate}

\begin{figure}
\centering

 
\tikzset{
pattern size/.store in=\mcSize, 
pattern size = 5pt,
pattern thickness/.store in=\mcThickness, 
pattern thickness = 0.3pt,
pattern radius/.store in=\mcRadius, 
pattern radius = 1pt}
\makeatletter
\pgfutil@ifundefined{pgf@pattern@name@_tifirdnag}{
\pgfdeclarepatternformonly[\mcThickness,\mcSize]{_tifirdnag}
{\pgfqpoint{-\mcThickness}{-\mcThickness}}
{\pgfpoint{\mcSize}{\mcSize}}
{\pgfpoint{\mcSize}{\mcSize}}
{
\pgfsetcolor{\tikz@pattern@color}
\pgfsetlinewidth{\mcThickness}
\pgfpathmoveto{\pgfpointorigin}
\pgfpathlineto{\pgfpoint{0}{\mcSize}}
\pgfusepath{stroke}
}}
\makeatother

 
\tikzset{
pattern size/.store in=\mcSize, 
pattern size = 5pt,
pattern thickness/.store in=\mcThickness, 
pattern thickness = 0.3pt,
pattern radius/.store in=\mcRadius, 
pattern radius = 1pt}
\makeatletter
\pgfutil@ifundefined{pgf@pattern@name@_tnmk38a98}{
\pgfdeclarepatternformonly[\mcThickness,\mcSize]{_tnmk38a98}
{\pgfqpoint{-\mcThickness}{-\mcThickness}}
{\pgfpoint{\mcSize}{\mcSize}}
{\pgfpoint{\mcSize}{\mcSize}}
{
\pgfsetcolor{\tikz@pattern@color}
\pgfsetlinewidth{\mcThickness}
\pgfpathmoveto{\pgfpointorigin}
\pgfpathlineto{\pgfpoint{0}{\mcSize}}
\pgfusepath{stroke}
}}
\makeatother

 
\tikzset{
pattern size/.store in=\mcSize, 
pattern size = 5pt,
pattern thickness/.store in=\mcThickness, 
pattern thickness = 0.3pt,
pattern radius/.store in=\mcRadius, 
pattern radius = 1pt}
\makeatletter
\pgfutil@ifundefined{pgf@pattern@name@_vpt0j7e5b}{
\pgfdeclarepatternformonly[\mcThickness,\mcSize]{_vpt0j7e5b}
{\pgfqpoint{-\mcThickness}{-\mcThickness}}
{\pgfpoint{\mcSize}{\mcSize}}
{\pgfpoint{\mcSize}{\mcSize}}
{
\pgfsetcolor{\tikz@pattern@color}
\pgfsetlinewidth{\mcThickness}
\pgfpathmoveto{\pgfpointorigin}
\pgfpathlineto{\pgfpoint{0}{\mcSize}}
\pgfusepath{stroke}
}}
\makeatother

 
\tikzset{
pattern size/.store in=\mcSize, 
pattern size = 5pt,
pattern thickness/.store in=\mcThickness, 
pattern thickness = 0.3pt,
pattern radius/.store in=\mcRadius, 
pattern radius = 1pt}
\makeatletter
\pgfutil@ifundefined{pgf@pattern@name@_jrwbomdft}{
\pgfdeclarepatternformonly[\mcThickness,\mcSize]{_jrwbomdft}
{\pgfqpoint{-\mcThickness}{-\mcThickness}}
{\pgfpoint{\mcSize}{\mcSize}}
{\pgfpoint{\mcSize}{\mcSize}}
{
\pgfsetcolor{\tikz@pattern@color}
\pgfsetlinewidth{\mcThickness}
\pgfpathmoveto{\pgfpointorigin}
\pgfpathlineto{\pgfpoint{0}{\mcSize}}
\pgfusepath{stroke}
}}
\makeatother

 
\tikzset{
pattern size/.store in=\mcSize, 
pattern size = 5pt,
pattern thickness/.store in=\mcThickness, 
pattern thickness = 0.3pt,
pattern radius/.store in=\mcRadius, 
pattern radius = 1pt}
\makeatletter
\pgfutil@ifundefined{pgf@pattern@name@_z3ryrpx55}{
\pgfdeclarepatternformonly[\mcThickness,\mcSize]{_z3ryrpx55}
{\pgfqpoint{-\mcThickness}{-\mcThickness}}
{\pgfpoint{\mcSize}{\mcSize}}
{\pgfpoint{\mcSize}{\mcSize}}
{
\pgfsetcolor{\tikz@pattern@color}
\pgfsetlinewidth{\mcThickness}
\pgfpathmoveto{\pgfpointorigin}
\pgfpathlineto{\pgfpoint{0}{\mcSize}}
\pgfusepath{stroke}
}}
\makeatother

 
\tikzset{
pattern size/.store in=\mcSize, 
pattern size = 5pt,
pattern thickness/.store in=\mcThickness, 
pattern thickness = 0.3pt,
pattern radius/.store in=\mcRadius, 
pattern radius = 1pt}
\makeatletter
\pgfutil@ifundefined{pgf@pattern@name@_8k1jhhl1y}{
\pgfdeclarepatternformonly[\mcThickness,\mcSize]{_8k1jhhl1y}
{\pgfqpoint{-\mcThickness}{-\mcThickness}}
{\pgfpoint{\mcSize}{\mcSize}}
{\pgfpoint{\mcSize}{\mcSize}}
{
\pgfsetcolor{\tikz@pattern@color}
\pgfsetlinewidth{\mcThickness}
\pgfpathmoveto{\pgfpointorigin}
\pgfpathlineto{\pgfpoint{0}{\mcSize}}
\pgfusepath{stroke}
}}
\makeatother

 
\tikzset{
pattern size/.store in=\mcSize, 
pattern size = 5pt,
pattern thickness/.store in=\mcThickness, 
pattern thickness = 0.3pt,
pattern radius/.store in=\mcRadius, 
pattern radius = 1pt}
\makeatletter
\pgfutil@ifundefined{pgf@pattern@name@_ffepub23j}{
\pgfdeclarepatternformonly[\mcThickness,\mcSize]{_ffepub23j}
{\pgfqpoint{-\mcThickness}{-\mcThickness}}
{\pgfpoint{\mcSize}{\mcSize}}
{\pgfpoint{\mcSize}{\mcSize}}
{
\pgfsetcolor{\tikz@pattern@color}
\pgfsetlinewidth{\mcThickness}
\pgfpathmoveto{\pgfpointorigin}
\pgfpathlineto{\pgfpoint{0}{\mcSize}}
\pgfusepath{stroke}
}}
\makeatother

 
\tikzset{
pattern size/.store in=\mcSize, 
pattern size = 5pt,
pattern thickness/.store in=\mcThickness, 
pattern thickness = 0.3pt,
pattern radius/.store in=\mcRadius, 
pattern radius = 1pt}
\makeatletter
\pgfutil@ifundefined{pgf@pattern@name@_8u3ndlwxt}{
\pgfdeclarepatternformonly[\mcThickness,\mcSize]{_8u3ndlwxt}
{\pgfqpoint{-\mcThickness}{-\mcThickness}}
{\pgfpoint{\mcSize}{\mcSize}}
{\pgfpoint{\mcSize}{\mcSize}}
{
\pgfsetcolor{\tikz@pattern@color}
\pgfsetlinewidth{\mcThickness}
\pgfpathmoveto{\pgfpointorigin}
\pgfpathlineto{\pgfpoint{0}{\mcSize}}
\pgfusepath{stroke}
}}
\makeatother

 
\tikzset{
pattern size/.store in=\mcSize, 
pattern size = 5pt,
pattern thickness/.store in=\mcThickness, 
pattern thickness = 0.3pt,
pattern radius/.store in=\mcRadius, 
pattern radius = 1pt}
\makeatletter
\pgfutil@ifundefined{pgf@pattern@name@_uyusttltt}{
\pgfdeclarepatternformonly[\mcThickness,\mcSize]{_uyusttltt}
{\pgfqpoint{-\mcThickness}{-\mcThickness}}
{\pgfpoint{\mcSize}{\mcSize}}
{\pgfpoint{\mcSize}{\mcSize}}
{
\pgfsetcolor{\tikz@pattern@color}
\pgfsetlinewidth{\mcThickness}
\pgfpathmoveto{\pgfpointorigin}
\pgfpathlineto{\pgfpoint{0}{\mcSize}}
\pgfusepath{stroke}
}}
\makeatother

 
\tikzset{
pattern size/.store in=\mcSize, 
pattern size = 5pt,
pattern thickness/.store in=\mcThickness, 
pattern thickness = 0.3pt,
pattern radius/.store in=\mcRadius, 
pattern radius = 1pt}
\makeatletter
\pgfutil@ifundefined{pgf@pattern@name@_m4j1isj21}{
\pgfdeclarepatternformonly[\mcThickness,\mcSize]{_m4j1isj21}
{\pgfqpoint{-\mcThickness}{-\mcThickness}}
{\pgfpoint{\mcSize}{\mcSize}}
{\pgfpoint{\mcSize}{\mcSize}}
{
\pgfsetcolor{\tikz@pattern@color}
\pgfsetlinewidth{\mcThickness}
\pgfpathmoveto{\pgfpointorigin}
\pgfpathlineto{\pgfpoint{0}{\mcSize}}
\pgfusepath{stroke}
}}
\makeatother
\tikzset{every picture/.style={line width=0.75pt}} 

\begin{tikzpicture}[x=0.75pt,y=0.75pt,yscale=-1,xscale=1]

\draw   (66.22,237.99) -- (178.5,163.5) -- (359.28,444.01) -- (246.99,518.5) -- cycle ;
\draw  [color={rgb, 255:red, 74; green, 144; blue, 226 }  ,draw opacity=1 ] (43,265.37) -- (648.5,265.37) -- (648.5,293.63) -- (43,293.63) -- cycle ;
\draw  [dash pattern={on 0.84pt off 2.51pt}]  (124.5,201.47) -- (301,480.53) ;
\draw  [dash pattern={on 0.84pt off 2.51pt}]  (156.63,377.21) -- (265.5,305.5) ;
\draw  [color={rgb, 255:red, 208; green, 2; blue, 27 }  ,draw opacity=1 ] (42,165.75) -- (647.5,165.75) -- (647.5,194.5) -- (42,194.5) -- cycle ;
\draw [color={rgb, 255:red, 208; green, 2; blue, 27 }  ,draw opacity=1 ] [dash pattern={on 0.84pt off 2.51pt}]  (345.5,165.5) -- (345.25,196.12) ;
\draw [color={rgb, 255:red, 208; green, 2; blue, 27 }  ,draw opacity=1 ] [dash pattern={on 0.84pt off 2.51pt}]  (43,181.5) -- (646.5,180.5) ;
\draw [color={rgb, 255:red, 74; green, 144; blue, 226 }  ,draw opacity=1 ] [dash pattern={on 0.84pt off 2.51pt}]  (345.5,264) -- (346.5,294.5) ;
\draw [color={rgb, 255:red, 74; green, 144; blue, 226 }  ,draw opacity=1 ] [dash pattern={on 0.84pt off 2.51pt}]  (44,280) -- (647.5,279) ;
\draw  [color={rgb, 255:red, 245; green, 166; blue, 35 }  ,draw opacity=1 ][pattern=_z3ryrpx55,pattern size=3pt,pattern thickness=0.75pt,pattern radius=0pt, pattern color={rgb, 255:red, 245; green, 166; blue, 35}] (83,265.5) -- (102.5,265.5) -- (102.5,293.5) -- (83,293.5) -- cycle ;
\draw  [color={rgb, 255:red, 245; green, 166; blue, 35 }  ,draw opacity=1 ][pattern=_8k1jhhl1y,pattern size=3pt,pattern thickness=0.75pt,pattern radius=0pt, pattern color={rgb, 255:red, 245; green, 166; blue, 35}] (245,265.5) -- (264.5,265.5) -- (264.5,293.5) -- (245,293.5) -- cycle ;
\draw  [color={rgb, 255:red, 245; green, 166; blue, 35 }  ,draw opacity=1 ][pattern=_ffepub23j,pattern size=3pt,pattern thickness=0.75pt,pattern radius=0pt, pattern color={rgb, 255:red, 245; green, 166; blue, 35}] (164,264.5) -- (183.5,264.5) -- (183.5,292.5) -- (164,292.5) -- cycle ;
\draw  [color={rgb, 255:red, 245; green, 166; blue, 35 }  ,draw opacity=1 ][pattern=_8u3ndlwxt,pattern size=3pt,pattern thickness=0.75pt,pattern radius=0pt, pattern color={rgb, 255:red, 245; green, 166; blue, 35}] (180.5,165.5) -- (198,165.5) -- (198,193.5) -- (180.5,193.5) -- cycle ;
\draw  [color={rgb, 255:red, 126; green, 211; blue, 33 }  ,draw opacity=1 ][pattern=_uyusttltt,pattern size=3pt,pattern thickness=0.75pt,pattern radius=0pt, pattern color={rgb, 255:red, 126; green, 211; blue, 33}] (130,166.5) -- (174.5,166.5) -- (174.5,193.5) -- (130,193.5) -- cycle ;
\draw  [color={rgb, 255:red, 144; green, 19; blue, 254 }  ,draw opacity=0.58 ] (43,497.37) -- (648.5,497.37) -- (648.5,525.63) -- (43,525.63) -- cycle ;
\draw  [color={rgb, 255:red, 155; green, 155; blue, 155 }  ,draw opacity=1 ][pattern=_m4j1isj21,pattern size=3pt,pattern thickness=0.75pt,pattern radius=0pt, pattern color={rgb, 255:red, 155; green, 155; blue, 155}] (231,497.5) -- (277.5,497.5) -- (277.5,525.5) -- (231,525.5) -- cycle ;
\draw    (150.5,174.5) .. controls (147.59,156.07) and (147.5,172.45) .. (146.59,153.37) ;
\draw [shift={(146.5,151.5)}, rotate = 447.4] [color={rgb, 255:red, 0; green, 0; blue, 0 }  ][line width=0.75]    (10.93,-3.29) .. controls (6.95,-1.4) and (3.31,-0.3) .. (0,0) .. controls (3.31,0.3) and (6.95,1.4) .. (10.93,3.29)   ;
\draw    (192.5,174.5) .. controls (198.14,159.46) and (200.25,161.21) .. (207.12,154.83) ;
\draw [shift={(208.5,153.5)}, rotate = 495] [color={rgb, 255:red, 0; green, 0; blue, 0 }  ][line width=0.75]    (10.93,-3.29) .. controls (6.95,-1.4) and (3.31,-0.3) .. (0,0) .. controls (3.31,0.3) and (6.95,1.4) .. (10.93,3.29)   ;

\draw (337,136.5) node [anchor=north west][inner sep=0.75pt]  [color={rgb, 255:red, 208; green, 2; blue, 27 }  ,opacity=1 ]  {$S^{2}$};
\draw (335,233.5) node [anchor=north west][inner sep=0.75pt]  [color={rgb, 255:red, 74; green, 144; blue, 226 }  ,opacity=1 ]  {$S^{1}$};
\draw (110,372.5) node [anchor=north west][inner sep=0.75pt]    {$\phi ( T)$};
\draw (195,128.5) node [anchor=north west][inner sep=0.75pt]  [color={rgb, 255:red, 245; green, 166; blue, 35 }  ,opacity=1 ]  {$\Delta ^{v}_{T}( S)$};
\draw (115,127.5) node [anchor=north west][inner sep=0.75pt]  [color={rgb, 255:red, 126; green, 211; blue, 33 }  ,opacity=1 ]  {$\Delta ^{h}_{T}( S)$};
\draw (325.5,530.5) node [anchor=north west][inner sep=0.75pt]  [color={rgb, 255:red, 144; green, 19; blue, 254 }  ,opacity=0.52 ]  {$S^{3}$};

\end{tikzpicture}

\caption{Geometric description of Proposition \ref{prop:1}} \label{fig:mot_prop:1}

\end{figure}
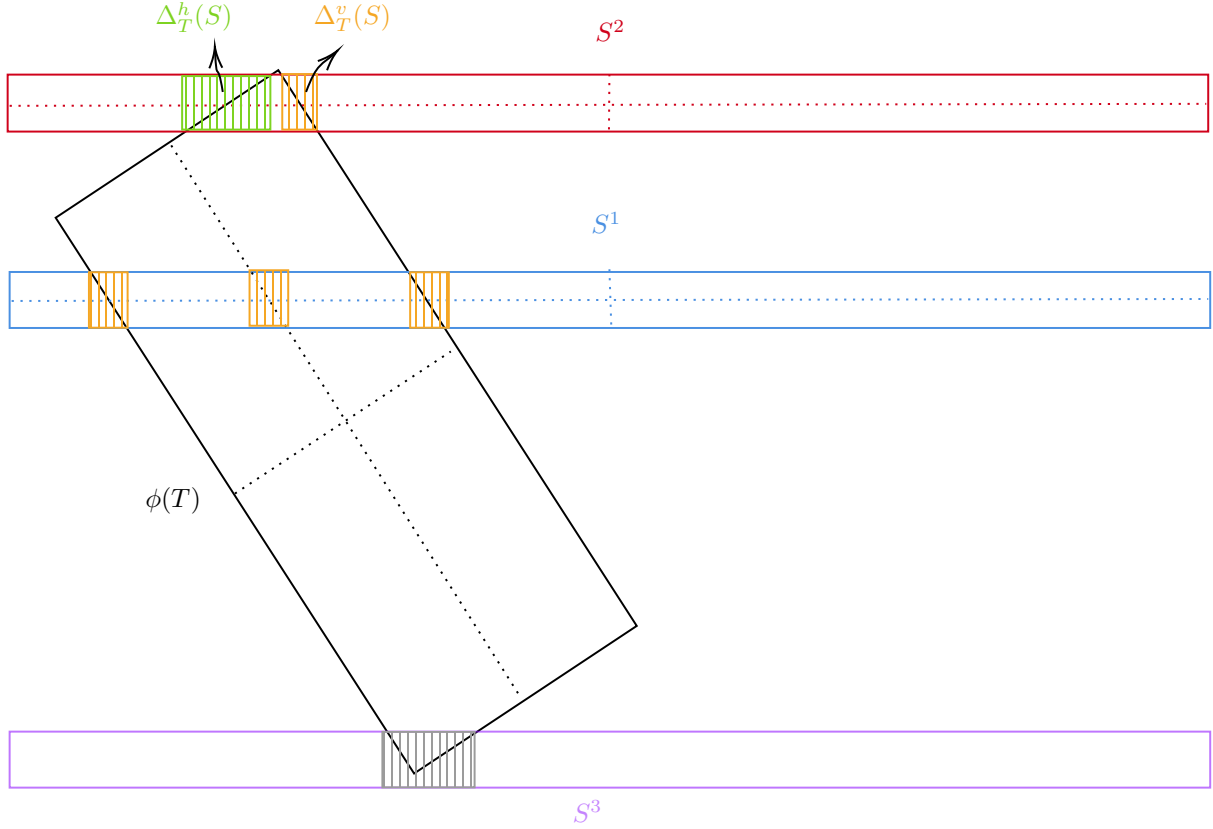

\begin{figure}
\centering

 
\tikzset{
pattern size/.store in=\mcSize, 
pattern size = 5pt,
pattern thickness/.store in=\mcThickness, 
pattern thickness = 0.3pt,
pattern radius/.store in=\mcRadius, 
pattern radius = 1pt}
\makeatletter
\pgfutil@ifundefined{pgf@pattern@name@_p0jwqc9ig}{
\pgfdeclarepatternformonly[\mcThickness,\mcSize]{_p0jwqc9ig}
{\pgfqpoint{-\mcThickness}{-\mcThickness}}
{\pgfpoint{\mcSize}{\mcSize}}
{\pgfpoint{\mcSize}{\mcSize}}
{
\pgfsetcolor{\tikz@pattern@color}
\pgfsetlinewidth{\mcThickness}
\pgfpathmoveto{\pgfpointorigin}
\pgfpathlineto{\pgfpoint{0}{\mcSize}}
\pgfusepath{stroke}
}}
\makeatother
\tikzset{every picture/.style={line width=0.75pt}} 

\begin{tikzpicture}[x=0.75pt,y=0.75pt,yscale=-1,xscale=1]

\draw   (78.83,685.45) -- (99.9,662.04) -- (336.36,874.88) -- (315.29,898.29) -- cycle ;
\draw  [dash pattern={on 0.84pt off 2.51pt}]  (90.73,674.2) -- (324.46,886.12) ;
\draw  [dash pattern={on 0.84pt off 2.51pt}]  (196.28,793.18) -- (218.99,767.9) ;
\draw  [color={rgb, 255:red, 74; green, 144; blue, 226 }  ,draw opacity=1 ] (42.5,693.5) -- (650,693.5) -- (650,750.5) -- (42.5,750.5) -- cycle ;
\draw [color={rgb, 255:red, 74; green, 144; blue, 226 }  ,draw opacity=1 ] [dash pattern={on 0.84pt off 2.51pt}]  (43,721.5) -- (649.5,722.5) ;
\draw [color={rgb, 255:red, 74; green, 144; blue, 226 }  ,draw opacity=1 ] [dash pattern={on 0.84pt off 2.51pt}]  (346,694) -- (346.5,750) ;
\draw  [color={rgb, 255:red, 245; green, 166; blue, 35 }  ,draw opacity=1 ][pattern=_p0jwqc9ig,pattern size=3pt,pattern thickness=0.75pt,pattern radius=0pt, pattern color={rgb, 255:red, 245; green, 166; blue, 35}] (88,693.5) -- (199.5,693.5) -- (199.5,750.5) -- (88,750.5) -- cycle ;

\draw (163,808.5) node [anchor=north west][inner sep=0.75pt]    {$\phi ( T)$};
\draw (338,662.5) node [anchor=north west][inner sep=0.75pt]  [color={rgb, 255:red, 74; green, 144; blue, 226 }  ,opacity=1 ]  {$S^1$};

\end{tikzpicture}\caption{Geometric description of Proposition \ref{prop:1} (1)}\label{fig:mot_prop:1(1)}

\end{figure}
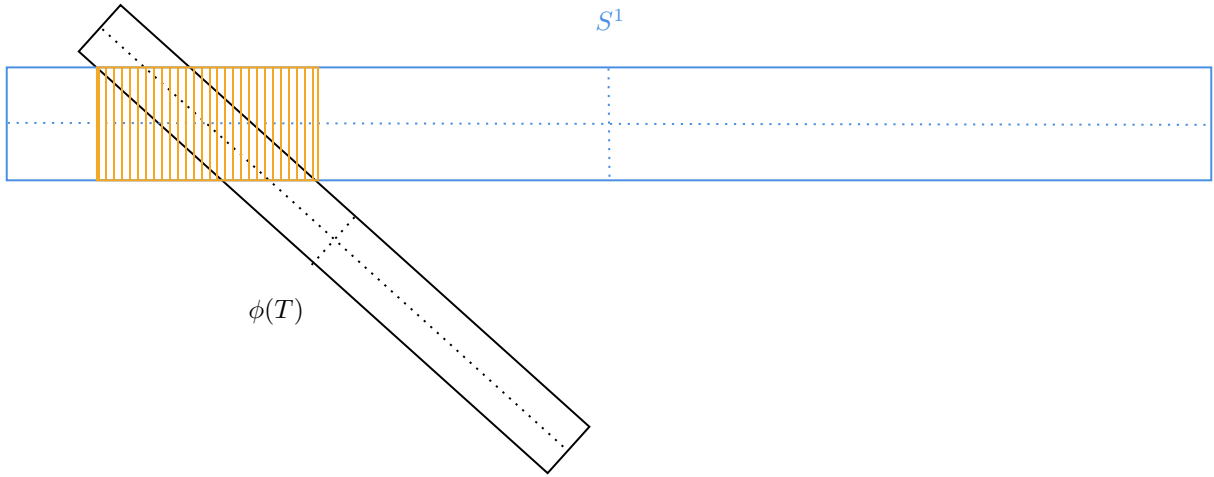

\begin{proposition} \label{prop:1} Consider the rotation $\phi:=\phi^\theta$ for some $\theta\in[0,2\pi]$. Let $S,T$ denote two dyadic rectangles of dimension $S_1\times S_2$ and $T_1\times T_2$ respectively and assume that there are no tops of $\phi(T)$ in $S$ ($\tops (\phi(T)) \cap S=\emptyset$). 
\begin{enumerate}
\item 
If $S$ intersects with only the rotated vertical segments of $T$, then
\begin{equation} \label{only_vert}
|\langle h_S \circ \phi, h_T \rangle|  \leq 3 \frac{\min(S_2^2 |\tan \theta|, S_1^2 |\cot\theta|)}{|T|^{1/2} |S|^{1/2}}.
\end{equation}
\item
If $S$ intersects with only the rotated horizontal segments of $T$, then
\begin{equation} \label{only_hori}
|\langle h_S \circ \phi, h_T \rangle| \leq 3 \frac{\min(S_1^2|\tan \theta|, S_2^2 |\cot \theta|)}{|T|^{1/2} |S|^{1/2}}.
\end{equation}
\item 
If $S$ intersects with both the rotated vertical and horizontal segments of $T$ and the vertical boundaries of $S$ do not intersect with $\phi(T)$ \textbf{or} if $S$ intersects with both the rotated vertical and horizontal segments of $T$ and the horizontal boundaries of $S$ do not intersect with $\phi(T)$, then
\begin{equation} \label{prop:1_decomp}
|\langle h_S \circ \phi, h_T \rangle| \leq \Delta^v_T(S) + \Delta^h_T(S),
\end{equation}
where
$$ \Delta_T^v(S) \leq 3  \frac{\min(S_2^2 |\tan \theta|, S_1^2 |\cot\theta|)}{|T|^{1/2} |S|^{1/2}} \qquad \textrm{and} \qquad \Delta_T^h(S) \leq 3  \frac{\min(S_1^2 |\tan \theta|, S_2^2 |\cot\theta|)}{|T|^{1/2} |S|^{1/2}}.$$
\end{enumerate}
\end{proposition}

\begin{remark} \label{rem:sym} 
Since $\phi=\phi^\theta$ is the rotation of angle $\theta$, then $\phi^{-1}=\phi^{-\theta}$ and 
\begin{equation} \langle h_S, h_T \circ \phi^{-1} \rangle = \langle h_S \circ \phi , h_T  \rangle
\label{eq:inner_p}
\end{equation}
and we see that up to a change of $\theta \to -\theta$ the quantity is symmetric in terms of $S$ and $T$. So the final estimate will be obtained by taking the minimum over different possibilities according that if we want to use cancellation of $h_S$ (along the $x$ coordinate, or $y$ coordinate) or of $h_T$. 
\end{remark}

\begin{proof}[Proof of Proposition \ref{prop:1}] 
We will start with the proof for Case (3) and the estimates for Cases (1) and (2) follow from the same proof with minor modifications. 
\begin{enumerate}[(a)] 
\item Case (3):
The proof consists of two parts - the first part verifies the decomposition (\ref{prop:1_decomp}) and the second part establishes the estimate for $\Delta^v_T(S)$. We remark that the argument for $\Delta^h_T(S)$ follows from the exactly same reasoning with $\theta$ replaced by $\frac{\pi}{2} - \theta$.
\newline
\textbf{Part I.}
Fix any dyadic rectangle $S=I\times J$ of the size 
$|S| = S_1 \times S_2$, in particular $I=(i,i+1)S_1$ and $J=(j,j+1)S_2$ for some integers $i,j$. By translation invariance, we assume that $T$ is has the lower left corner positioned at the origin.\footnote{The general case can be verified by the computation that 
\fontsize{6.5}{6.5}{$$\int h_{T}(\phi^{-1}(x,y)) h_S(x,y)dx dy = \int h_{T_0}(\phi^{-1}(x,y) - t) h_S(x,y)dx dy =\int h_{T_0} \circ \phi^{-1}\left((x,y) - \phi(t)\right) h_S(x,y)dx dy =  \int h_{T_0}\circ \phi^{-1}(x',y') h_S\left((x',y')+ \phi(t)\right)dx' dy',$$} where $T_0$ is the rectangle with the same shape as $T$ with lower left corner at the origin and $(x',y') := (x,y) + \phi(t)$.}

To compute
\begin{align} \label{inner_prod}
\langle h_S, h_T \circ \phi^{-1} \rangle = & \int_{i S_1}^{(i+1)S_1} \int_{j S_2}^{(j+1) S_2} h_I \otimes h_J (x,y) h_T \circ \phi^{-1}(x,y) dy dx,
\end{align}
we recall the $6$ segments: 
\begin{enumerate}
\item 3 vertical segments
$$
l^v_1 := 
\left\{\begin{bmatrix}
\frac{T_1}{2} \\
\tilde{y}
\end{bmatrix}, \ 0 \leq \tilde{y} \leq T_2\right\},
\qquad 
l^v_2 := 
\left\{\begin{bmatrix}
T_1 \\
\tilde{y}
\end{bmatrix}, \ 0 \leq \tilde{y} \leq T_2\right\}
\qquad \textrm{and} \qquad
l^v_{3}:=
\left\{\begin{bmatrix}
0 \\
\tilde{y}
\end{bmatrix}, \ 0 \leq \tilde{y} \leq T_2\right\}
$$ 
\item 3 horizontal segments
$$
l^h_1:=
\left\{\begin{bmatrix}
\tilde{x} \\
\frac{T_2}{2}
\end{bmatrix},\ 0 \leq \tilde{x} \leq T_1\right\},
\qquad
l^h_2 :=
\left\{\begin{bmatrix}
\tilde{x} \\
T_2
\end{bmatrix},\ 0 \leq \tilde{x} \leq T_1\right\}
\qquad \textrm{and} \qquad
l^h_3 :=
\left\{\begin{bmatrix}
\tilde{x} \\
0
\end{bmatrix},\ 0 \leq \tilde{x} \leq T_1\right\}.
$$
\end{enumerate}
Suppose there are no tops of $\phi(T)$ in $S$, the condition in (3) ensures that intersection of $S$ with all the $6$ segments can be written as a disjoint union of rectangles intersecting with each segment (up to endpoints).

In particular, if $S$ intersects with both the rotated vertical and horizontal segments of $T$ and the vertical boundaries of $S$ do not intersect with $\phi(T)$, then 
\begin{align} \label{x_disjoint}
& S \cap \phi(l^v_1 \cup l^v_2 \cup l^v_3 \cup l^h_1 \cup l^h_2 \cup l^h_3) \nonumber\\
= &\left((x_1^1, x_2^1) \times J \cap \phi(l^v_1) \right) \cup \left((x_1^2, x_2^2) \times J \cap \phi(l^v_2) \right) \cup \left((x_1^3, x_2^3) \times J \cap \phi(l^v_3) \right) \cup \nonumber\\
& \left((x_1^4, x_2^4) \times J \cap \phi(l^h_1) \right) \cup \left((x_1^5, x_2^5) \times J \cap \phi(l^h_2) \right) \cup \left((x_1^6, x_2^6) \times J \cap \phi(l^h_3) \right)
\end{align}
where 
\begin{equation*}
(x_1^{l}, x_2^l) \cap (x_1^{l'}, x_2^{l'}) = \emptyset
\end{equation*}
for $l\neq l'$. With a little abuse of notation, the empty interval is denoted by $(x_1,x_2)$ with $x_1 = x_2$.
If $S$ intersects with both the rotated vertical and horizontal segments of $T$ and the horizontal boundaries of $S$ do not intersect with $\phi(T)$, then
\begin{align*}
&S \cap \phi(l^v_1 \cup l^v_2 \cup l^v_3 \cup l^h_1 \cup l^h_2 \cup l^h_3) \nonumber\\
=&\left(I \times [y_1^1,y_2^1] \cap \phi(l^v_1) \right) \cup \left(I \times [y_1^2,y_2^2] \cap \phi(l^v_2)\right) \cup \left(I \times [y_1^3,y_2^3] \cap \phi(l^v_3) \right)\cup \\
& \left(I \times [y_1^4,y_2^4] \cap \phi(l^h_1) \right) \cup \left(I \times [y_1^5,y_2^5] \cap \phi(l^h_2)\right) \cup \left(I \times [y_1^6,y_2^6] \cap \phi(l^h_3) \right)
\end{align*}
where 
\begin{equation*}
(y_1^l, y_2^l) \cap (y_1^{l'}, y_2^{l'}) = \emptyset
\end{equation*}
for $l \neq l'$. 
Let us assume, without loss of generality, that we are in the case (\ref{x_disjoint}). Then (\ref{inner_prod}) can be rewritten as follows. 
\begin{align*}
 \iint h_I \otimes h_J (x,y) h_T \circ \phi^{-1}(x,y) dy dx =\sum_{\ell=1}^6 \int_{x_1^\ell < x < x_2^\ell} h_I \otimes h_J (x,y) h_T \circ \phi^{-1}(x,y) dy dx,
\end{align*}
since when $\displaystyle x\in I \setminus \bigcup_{\ell=1}^6 (x_1^{\ell},x_2^{\ell})$, all the segment $\{x\} \times J$ (up to endpoints) is included in one of the four quarters of $\phi(T)$ or in $\phi(T)^{c}$. So $h_T \circ \phi^{-1}$ is constant and $h_{S}$ has a mean zero on this segment due to the cancellation imposed by the oscillation of $h_J$. Thus the integral
$$ \int_{y\in J} h_S(x,y) h_T \circ \phi(x,y) \, dy=0.$$ 
Moreover, by the disjointness of the decomposition specified in (\ref{x_disjoint}),
\begin{equation}
 \int_{x_1^1 < x < x_2^1} h_I \otimes h_J (x,y) h_T \circ \phi^{-1}(x,y) dy dx =  \int_{x_1^1 < x < x_2^1} h_I \otimes h_J (x,y) h_{\phi(l^v_1)}(x,y) dy dx = \int_{S} h_I \otimes h_J (x,y) h_{\phi(l^v_1)}(x,y) dy dx
\end{equation}
where the last equation again follows from the oscillation of  $h_S$ in $y-$direction for each fixed $x$. The same reasoning can be applied to all the other rectangles in (\ref{x_disjoint}) so that we can derive 
\begin{equation}
\left|\sum_{\ell=1}^3 \int_{x_1^\ell < x < x_2^\ell} h_I \otimes h_J (x,y) h_T \circ \phi^{-1}(x,y) dy dx \right| \leq \Delta^v_T(S) 
\end{equation}
and 
\begin{equation}
\left|\sum_{\ell=3}^6 \int_{x_1^\ell < x < x_2^\ell} h_I \otimes h_J (x,y) h_T \circ \phi^{-1}(x,y) dy dx \right| \leq \Delta^h_T(S) 
\end{equation}
which completes the proof of (\ref{prop:1_decomp}).
\newline
\textbf{Part II.}
We will now prove the estimate for $\Delta^v_T(S)$, in particular
\begin{equation*}
\Delta^v_T(S) \leq 3  \frac{\min(S_2^2 |\tan \theta|, S_1^2 |\cot\theta|)}{|T|^{1/2} |S|^{1/2}}.
\end{equation*}
It suffices to show that for each $1 \leq j \leq 3$,
\begin{equation}\label{vertical_estimate}
|\langle h_S, h_{\phi(l^v_j)} \rangle| \leq \frac{\min(S_2^2 |\tan \theta|, S_1^2 |\cot\theta|)}{|T|^{1/2} |S|^{1/2}}.
\end{equation}
Without loss of generality (by translation), assume that $T$ can be represented by its four corner as
\begin{equation}
T = 
\begin{bmatrix}
0 & T_1 & T_1 & 0 \\
0 & 0 & T_2 & T_2
\end{bmatrix}
\end{equation}
where the size of the dyadic rectangle is represented as
$$|T| = T_1 \times T_2.$$
One notices that the argument for $\theta \in [-\frac{\pi}{2},\frac{\pi}{2}]$ can be applied to the case when $\theta \in [-\pi,-\frac{\pi}{2}] \cup [\frac{\pi}{2},\pi)$. Suppose that $\theta \in [-\frac{\pi}{2},\frac{\pi}{2}]$, then $\phi^{\theta}$ can be expressed as the rotation matrix 
\begin{equation}
\phi^{\theta} = 
\begin{bmatrix}
\cos \theta & -\sin \theta \\
\sin \theta & \cos \theta
\end{bmatrix}
\end{equation}
and the rotated rectangle $\phi(T)$ is indeed
\begin{equation}
\phi(T) = 
\phi^{\theta} \cdot T
= 
\begin{bmatrix}
0 & T_1 \cos \theta & T_1 \cos \theta - T_2 \sin \theta  & -T_2 \sin \theta \\
0 &T_1 \sin \theta & T_1 \sin \theta + T_2 \cos \theta & T_2 \cos \theta
\end{bmatrix}. 
\end{equation}

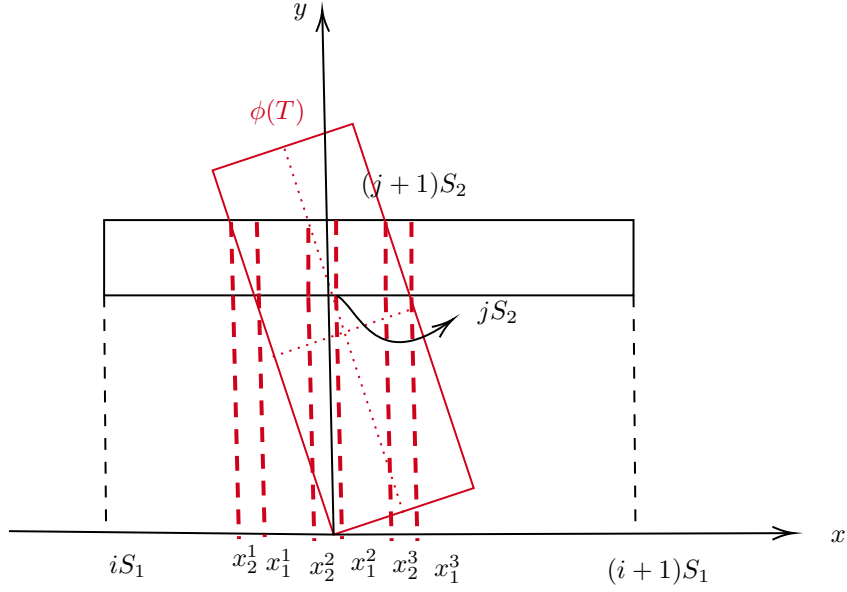
\begin{figure}
\centering

\tikzset{every picture/.style={line width=0.75pt}} 

\begin{tikzpicture}[x=0.75pt,y=0.75pt,yscale=-1,xscale=1]

\draw   (176,172) -- (443,172) -- (443,210) -- (176,210) -- cycle ;
\draw  [color={rgb, 255:red, 208; green, 2; blue, 27 }  ,draw opacity=1 ] (230.72,146.9) -- (291.76,330.73) -- (362.41,307.27) -- (301.36,123.44) -- cycle ;
\draw [color={rgb, 255:red, 208; green, 2; blue, 27 }  ,draw opacity=1 ] [dash pattern={on 0.84pt off 2.51pt}]  (267,136) -- (327,320) ;
\draw [color={rgb, 255:red, 208; green, 2; blue, 27 }  ,draw opacity=1 ] [dash pattern={on 0.84pt off 2.51pt}]  (261,241) -- (331,217) ;
\draw [color={rgb, 255:red, 208; green, 2; blue, 27 }  ,draw opacity=1 ][line width=1.5]  [dash pattern={on 5.63pt off 4.5pt}]  (240,173) -- (241,210) ;
\draw [color={rgb, 255:red, 208; green, 2; blue, 27 }  ,draw opacity=1 ][line width=1.5]  [dash pattern={on 5.63pt off 4.5pt}]  (253,173) -- (254,209) ;
\draw [color={rgb, 255:red, 208; green, 2; blue, 27 }  ,draw opacity=1 ][line width=1.5]  [dash pattern={on 5.63pt off 4.5pt}]  (279,176) -- (279,211) ;
\draw [color={rgb, 255:red, 208; green, 2; blue, 27 }  ,draw opacity=1 ][line width=1.5]  [dash pattern={on 5.63pt off 4.5pt}]  (293,172) -- (293,210) ;
\draw [color={rgb, 255:red, 208; green, 2; blue, 27 }  ,draw opacity=1 ][line width=1.5]  [dash pattern={on 5.63pt off 4.5pt}]  (318,173) -- (318,211) ;
\draw [color={rgb, 255:red, 208; green, 2; blue, 27 }  ,draw opacity=1 ][line width=1.5]  [dash pattern={on 5.63pt off 4.5pt}]  (331,173) -- (331,211) ;
\draw    (291.76,330.73) -- (286.04,68) ;
\draw [shift={(286,66)}, rotate = 448.75] [color={rgb, 255:red, 0; green, 0; blue, 0 }  ][line width=0.75]    (10.93,-3.29) .. controls (6.95,-1.4) and (3.31,-0.3) .. (0,0) .. controls (3.31,0.3) and (6.95,1.4) .. (10.93,3.29)   ;
\draw    (291.76,330.73) -- (522,330.01) ;
\draw [shift={(524,330)}, rotate = 539.8199999999999] [color={rgb, 255:red, 0; green, 0; blue, 0 }  ][line width=0.75]    (10.93,-3.29) .. controls (6.95,-1.4) and (3.31,-0.3) .. (0,0) .. controls (3.31,0.3) and (6.95,1.4) .. (10.93,3.29)   ;
\draw    (128,329) -- (291.76,330.73) ;
\draw [color={rgb, 255:red, 208; green, 2; blue, 27 }  ,draw opacity=1 ][line width=1.5]  [dash pattern={on 5.63pt off 4.5pt}]  (254,209) -- (257,332) ;
\draw [color={rgb, 255:red, 208; green, 2; blue, 27 }  ,draw opacity=1 ][line width=1.5]  [dash pattern={on 5.63pt off 4.5pt}]  (241,210) -- (244,333) ;
\draw [color={rgb, 255:red, 208; green, 2; blue, 27 }  ,draw opacity=1 ][line width=1.5]  [dash pattern={on 5.63pt off 4.5pt}]  (331,211) -- (334,334) ;
\draw [color={rgb, 255:red, 208; green, 2; blue, 27 }  ,draw opacity=1 ][line width=1.5]  [dash pattern={on 5.63pt off 4.5pt}]  (293,210) -- (296,333) ;
\draw [color={rgb, 255:red, 208; green, 2; blue, 27 }  ,draw opacity=1 ][line width=1.5]  [dash pattern={on 5.63pt off 4.5pt}]  (318,211) -- (321,334) ;
\draw [color={rgb, 255:red, 208; green, 2; blue, 27 }  ,draw opacity=1 ][line width=1.5]  [dash pattern={on 5.63pt off 4.5pt}]  (279,211) -- (282,334) ;
\draw    (293,210) .. controls (303.89,214.95) and (311.84,251.26) .. (350.81,222.88) ;
\draw [shift={(352,222)}, rotate = 503.13] [color={rgb, 255:red, 0; green, 0; blue, 0 }  ][line width=0.75]    (10.93,-3.29) .. controls (6.95,-1.4) and (3.31,-0.3) .. (0,0) .. controls (3.31,0.3) and (6.95,1.4) .. (10.93,3.29)   ;
\draw  [dash pattern={on 4.5pt off 4.5pt}]  (176,210) -- (177,330) ;
\draw  [dash pattern={on 4.5pt off 4.5pt}]  (443,210) -- (444,330) ;

\draw (256,336.4) node [anchor=north west][inner sep=0.75pt]    {$x^{1}_{1}$};
\draw (239,334.4) node [anchor=north west][inner sep=0.75pt]    {$x^{1}_{2}$};
\draw (278,338.4) node [anchor=north west][inner sep=0.75pt]    {$x^{2}_{2}$};
\draw (299,336.4) node [anchor=north west][inner sep=0.75pt]    {$x^{2}_{1}$};
\draw (320,337.4) node [anchor=north west][inner sep=0.75pt]    {$x^{3}_{2}$};
\draw (341,339.4) node [anchor=north west][inner sep=0.75pt]    {$x^{3}_{1}$};
\draw (304,146.4) node [anchor=north west][inner sep=0.75pt]    {$( j+1) S_{2}$};
\draw (363,211.4) node [anchor=north west][inner sep=0.75pt]    {$jS_{2}$};
\draw (177,339.4) node [anchor=north west][inner sep=0.75pt]    {$iS_{1}$};
\draw (428,341.4) node [anchor=north west][inner sep=0.75pt]    {$( i+1) S_{1}$};
\draw (541,327.4) node [anchor=north west][inner sep=0.75pt]    {$x$};
\draw (270,61.4) node [anchor=north west][inner sep=0.75pt]    {$y$};
\draw (248,109.4) node [anchor=north west][inner sep=0.75pt]  [color={rgb, 255:red, 208; green, 2; blue, 27 }  ,opacity=1 ]  {$\phi ( T)$};

\end{tikzpicture} \caption{Geometric interpretation of Proof for Proposition \ref{prop:1}}

\end{figure}

\underline{Option 1:} Integrating in $S$ over $J$ with $x \in I$ fixed.\\ We will now prove the following claim.
\begin{claim}
If $x_1,x_2$ comes from the same rotated vertical segment then 
\begin{equation}
|x_1- x_2| \leq \min(S_2 |\tan\theta|,S_1); \label{eq:x1x2-1}
\end{equation}
\end{claim}
Suppose that $|x_1-x_2| \leq S_1$, then for some $\tilde{y}_1, \tilde y_2\in[0,T_2]$ ,
\begin{align}
 \begin{bmatrix}
x_1 \\
jS_2
\end{bmatrix}
= &
\phi^{\theta}\cdot
\begin{bmatrix}
\frac{T_1}{2} \\
\tilde{y}_1
\end{bmatrix} 
= 
\begin{bmatrix}
\frac{T_1}{2}\cos \theta - \tilde{y}_1 \sin \theta \\
\frac{T_1}{2} \sin\theta+ \tilde{y}_1 \cos \theta
\end{bmatrix} \nonumber \\
 \begin{bmatrix}
x_2 \\
(j+1)S_2
\end{bmatrix}
= &
\phi^{\theta}\cdot
\begin{bmatrix}
\frac{T_1}{2} \\
\tilde{y}_2
\end{bmatrix} 
= 
\begin{bmatrix}
\frac{T_1}{2}\cos \theta - \tilde{y}_2 \sin \theta \\
\frac{T_1}{2} \sin\theta+ \tilde{y}_2 \cos \theta
\end{bmatrix}. \nonumber
\end{align}
That generate the relation
$$ \tilde{y_2} - \tilde{y_1} = \frac{S_2}{\cos\theta}$$
and $$ x_2 - x_1 = (\tilde{y_1} - \tilde{y_2}) \sin\theta = -S_2 \tan\theta,$$
which concludes the proof of \eqref{eq:x1x2-1}. The computations are exactly the same for other vertical segments. For the second case, a similar argument holds with $\theta$ replaced by $\frac{\pi}{2}-\theta$.

\medskip

We can now apply the estimate (\ref{eq:x1x2-1}) to majorize the left hand side of (\ref{vertical_estimate}) by
\begin{align}\label{vertical_estimate_2}
\left|\iint_{x_1^j\leq x\leq x_2^j} h_I \otimes h_J (x,y) h_{\phi(l^v_j)}(x,y) dy dx \right| & \leq |[x_1^j,x_2^j] \times J| |S|^{-\frac{1}{2}} |T|^{-\frac{1}{2}} \leq \min(S_2 |\tan\theta|,S_1) \cdot S_2 \cdot |S|^{-\frac{1}{2}} |T|^{-\frac{1}{2}}
\end{align}

\underline{Option 2:} Integrating in $S$ over $I$ with $y \in J$ fixed.\\
By an analogous argument, we can show that by integrating over $x-$ direction, we obtain the following estimate.
\begin{claim}
If $x_1,x_2$ comes from the same rotated vertical segment then 
\begin{equation}\label{eq:x1x2-2}
|x_1- x_2| \leq \min(S_1 |\cot\theta|,S_2); 
\end{equation}
\end{claim}
By applying (\ref{eq:x1x2-2}), we derive the second estimate for the left hand side of (\ref{vertical_estimate}).
\begin{equation}\label{vertical_estimate_1}
\left|\iint_{x_1^j\leq x\leq x_2^j} h_I \otimes h_J (x,y) h_{\phi(l^v_j)}(x,y) dy dx \right|  \leq |[x_1^j,x_2^j] \times J| |S|^{-\frac{1}{2}} |T|^{-\frac{1}{2}} \leq \min(S_1 |\cot\theta|,S_2) \cdot S_1 \cdot |S|^{-\frac{1}{2}} |T|^{-\frac{1}{2}}
\end{equation}
By combining estimates (\ref{vertical_estimate_2}) and (\ref{vertical_estimate_1}), we deduce that
\begin{equation}
\left|\iint_{x_1^j\leq x\leq x_2^j} h_I \otimes h_J (x,y) h_{\phi(l^v_j)}(x,y) dy dx \right|  \leq \min(S_2^2 |\tan \theta|, S_1^2 |\cot\theta|,S_1S_2) \cdot |S|^{-\frac{1}{2}} |T|^{-\frac{1}{2}}.
\end{equation}
We notice that if $S_2^2 |\tan \theta| \geq S_1 S_2$, then $|\tan \theta| \geq \frac{S_1}{S_2}$, which implies that 
$
S_1^2 |\cot \theta| \leq S_1 S_2,
$
in which case 
$$
\min(S_2^2 |\tan \theta|, S_1^2 |\cot\theta|,S_1S_2) = S_1^2 |\cot \theta| .
$$
If $S_1^2 |\cot \theta| \geq S_1 S_2$, then $|\tan \theta| \leq \frac{S_1}{S_2}$, which implies that 
$
S_2^2 |\tan \theta| \leq S_1 S_2,
$
in which case 
$$
\min(S_2^2 |\tan \theta|, S_1^2 |\cot\theta|,S_1S_2) = S_2^2 |\tan\theta| .
$$
As a consequence, we conclude that
$$
|\langle h_S, h_{\phi(l^v_j)} \rangle| \leq \min(S_2^2 |\tan \theta|, S_1^2 |\cot \theta|) \cdot |S|^{-\frac{1}{2}} |T|^{-\frac{1}{2}},
$$
which completes the proof of (\ref{vertical_estimate}). By replacing $\theta$ with $\frac{\pi}{2}- \theta$ and applying the same argument, we can derive the desired estimates for $\Delta^h_T(S)$.
\item 
Cases (1) and (2): 
Both cases can be proven by a similar in (a). Although we can no longer convert the biparameter estimates to the sum of one-parameter estimates. In particular, we can majorize
$$
|\langle h_S \circ \phi, h_T \rangle | \leq \|h_S\|_{\infty} \|h_T\|_{\infty}|\phi(T) \cap S|.
$$
By the same argument in (a), we can show that
\begin{align*}
& |\phi(T) \cap S| \leq 
\begin{cases}
3 \min(S_2^2 |\tan \theta|, S_1^2 |\cot\theta|) \text{ in Case 1}\\
3 \min(S_1^2 |\tan \theta|, S_2^2 |\cot\theta|) \text{ in Case 2}.
\end{cases}
\end{align*}
\end{enumerate}
\end{proof}

\begin{remark} \label{remark_top}
\begin{enumerate}
\item
The assumption that $S$ does not contain any tops of $\phi(T)$ can be easily understood to be necessary for $\theta=0$. If $\theta=0$ and so $\phi=\textrm{Id}$, then if there is no tops of $T$ in $S$, then $S$ can intersect only horizontal (or vertical or none) sides of $T$, which means in particular that we will always be able to use the cancellation of $h_R$ along the $x$-variable or $y$-variable in the whole $R\cap S$. Hence $\langle h_R, h_S\rangle = 0$. However, if we have at least one top of $S$ in $T$ and only one top of $T$ in $S$, then it is easy to construct situations where the inner product is not vanishing (this is particular the case where $S=T$ !). In this case, we cannot use the perfect cancellation estimates in estimating the inner product. 
\item
We can no longer convert the bi-parameter problem to the sum of one-parameter estimates such as (\ref{prop:1_decomp}). Fortunately, such scenario does not happen too often and the following trivial estimate is sufficient:
\begin{equation}
|\langle h_S \circ \phi, h_T\rangle| \lesssim \frac{|\phi(T) \cap S|}{|T|^{1/2} |S|^{1/2}} \leq \frac{\min(|\phi(T)|, |S|)}{|T|^{1/2} |S|^{1/2}}. \label{eq:innertop}
\end{equation}
\end{enumerate}
\end{remark}

Our goal is to obtain an estimate of $\Gamma(K_{\max,i}, r_1,r_2, k_1, k_2)$ defined in \eqref{eq:Gamma} as
$$
\Gamma(K_{\max,i},r_1,r_2,k_1,k_2):= \sum_{\substack{K \subseteq K_{\max,i} \\ |K| = 2^{k_1} \times 2^{k_2}}} \left(\sum_{\substack{R \in \mathbb{D}^{\beta(\phi(K_{\max,i}))} \\R \nsubseteq \widetilde{\phi(\Omega)} \\ |R| = 2^{r_1} \times 2^{r_2}}} |\langle h_R \circ \phi, h_K \rangle|^{p'} \right)^{\frac{2}{p'}} .
$$ 
Having obtained a uniform bound of the inner product $\langle h_R \circ \phi, h_K \rangle$ for $(R,K) \in \Lambda(K_{\max,i},r_1,r_2,k_1,k_2)$, it suffices now to count the number of $R$ and $K$ with fixed sizes, involved in the sum defining $\Gamma(K_{\max,i},r_1,r_2,k_1,k_2)$. The way of counting the couples $(R,K)$ will differ according to the geometry of intersections between $R$ and $\phi(K)$. 
One case is that each rotated $K \subseteq K_{\max,i}$ intersect with $R$'s of fixed dimension evenly and frequently and we will denote it by ``{\bf even case}" from now on. The second case corresponds to when each $\phi(K) \subseteq \phi(K_{\max,i})$ intersects with $R$'s of fixed dimension sparsely, which will be denoted as 
``{\bf sparse case}". Both cases are indicated in Figure \ref{fig:mot_prop:number}. In the ``even case", we would count the number of $R$'s of fixed dimension intersecting with each $\phi(K)$ and then count the number of $K$'s contained in $K_{\max,i}$. Meanwhile, for the ``sparse case", we would like to use a different counting method because some $\phi(K)$'s do no intersect with any $R$ at all, such as $\phi(K^3)$ in Figure \ref{fig:mot_prop:number}, so that the computation for the ``even case" is not ideal any more.
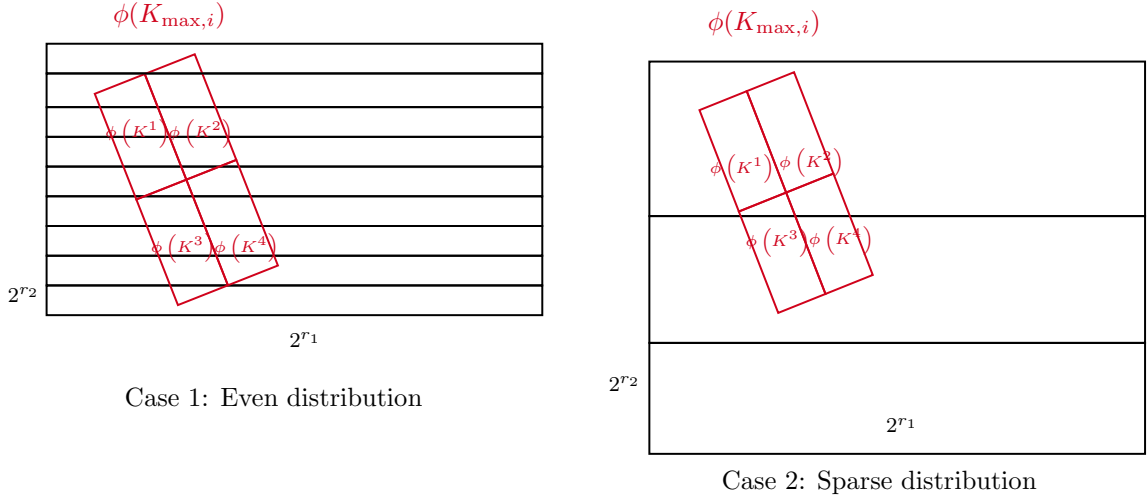
\begin{figure}
\centering

\tikzset{every picture/.style={line width=0.75pt}} 

\begin{tikzpicture}[x=0.75pt,y=0.75pt,yscale=-1,xscale=1]

\draw   (38,46) -- (288,46) -- (288,63) -- (38,63) -- cycle ;
\draw   (38,63) -- (288,63) -- (288,78) -- (38,78) -- cycle ;
\draw   (38,78) -- (288,78) -- (288,93) -- (38,93) -- cycle ;
\draw   (38,108) -- (288,108) -- (288,123) -- (38,123) -- cycle ;
\draw   (38,123) -- (288,123) -- (288,138) -- (38,138) -- cycle ;
\draw   (38,93) -- (288,93) -- (288,108) -- (38,108) -- cycle ;
\draw   (38,138) -- (288,138) -- (288,153) -- (38,153) -- cycle ;
\draw   (38,153) -- (288,153) -- (288,168) -- (38,168) -- cycle ;
\draw  [color={rgb, 255:red, 208; green, 2; blue, 27 }  ,draw opacity=1 ] (112.71,36.32) -- (133.69,89.63) -- (108.48,99.62) -- (87.5,46.3) -- cycle ;
\draw  [color={rgb, 255:red, 208; green, 2; blue, 27 }  ,draw opacity=1 ] (87.51,46.3) -- (108.49,99.62) -- (83.27,109.6) -- (62.3,56.28) -- cycle ;
\draw  [color={rgb, 255:red, 208; green, 2; blue, 27 }  ,draw opacity=1 ] (108.48,99.62) -- (129.46,152.93) -- (104.24,162.91) -- (83.26,109.6) -- cycle ;
\draw  [color={rgb, 255:red, 208; green, 2; blue, 27 }  ,draw opacity=1 ] (133.68,89.63) -- (154.66,142.95) -- (129.45,152.93) -- (108.47,99.61) -- cycle ;

\draw   (38,31) -- (288,31) -- (288,46) -- (38,46) -- cycle ;
\draw   (342,40) -- (592,40) -- (592,118) -- (342,118) -- cycle ;
\draw   (342,118) -- (592,118) -- (592,182) -- (342,182) -- cycle ;
\draw   (342,182) -- (592,182) -- (592,238) -- (342,238) -- cycle ;
\draw  [color={rgb, 255:red, 208; green, 2; blue, 27 }  ,draw opacity=1 ] (414.97,45.37) -- (434.82,96.54) -- (410.98,106.12) -- (391.12,54.94) -- cycle ;
\draw  [color={rgb, 255:red, 208; green, 2; blue, 27 }  ,draw opacity=1 ] (391.13,54.94) -- (410.99,106.11) -- (387.14,115.69) -- (367.28,64.52) -- cycle ;
\draw  [color={rgb, 255:red, 208; green, 2; blue, 27 }  ,draw opacity=1 ] (410.98,106.12) -- (430.83,157.29) -- (406.99,166.86) -- (387.13,115.69) -- cycle ;
\draw  [color={rgb, 255:red, 208; green, 2; blue, 27 }  ,draw opacity=1 ] (434.81,96.54) -- (454.67,147.71) -- (430.82,157.29) -- (410.97,106.12) -- cycle ;

\draw (70,8.4) node [anchor=north west][inner sep=0.75pt]  [color={rgb, 255:red, 208; green, 2; blue, 27 }  ,opacity=1 ]  {$\phi ( K_{\max,i})$};
\draw (65.97,67.46) node [anchor=north west][inner sep=0.75pt]  [font=\tiny,color={rgb, 255:red, 208; green, 2; blue, 27 }  ,opacity=1 ,rotate=-359.86]  {$\phi \left( K^{1}\right)$};
\draw (97.97,67.46) node [anchor=north west][inner sep=0.75pt]  [font=\tiny,color={rgb, 255:red, 208; green, 2; blue, 27 }  ,opacity=1 ,rotate=-359.86]  {$\phi \left( K^{2}\right)$};
\draw (88.97,125.46) node [anchor=north west][inner sep=0.75pt]  [font=\tiny,color={rgb, 255:red, 208; green, 2; blue, 27 }  ,opacity=1 ,rotate=-359.86]  {$\phi \left( K^{3}\right)$};
\draw (120.97,125.46) node [anchor=north west][inner sep=0.75pt]  [font=\tiny,color={rgb, 255:red, 208; green, 2; blue, 27 }  ,opacity=1 ,rotate=-359.86]  {$\phi \left( K^{4}\right)$};
\draw (17,152.4) node [anchor=north west][inner sep=0.75pt]  [font=\footnotesize]  {$2^{r_{2}}$};
\draw (159,175.4) node [anchor=north west][inner sep=0.75pt]  [font=\footnotesize]  {$2^{r_{1}}$};
\draw (370,12.4) node [anchor=north west][inner sep=0.75pt]  [color={rgb, 255:red, 208; green, 2; blue, 27 }  ,opacity=1 ]  {$\phi ( K_{\max,i})$};
\draw (370.97,85.46) node [anchor=north west][inner sep=0.75pt]  [font=\tiny,color={rgb, 255:red, 208; green, 2; blue, 27 }  ,opacity=1 ,rotate=-359.86]  {$\phi \left( K^{1}\right)$};
\draw (405.97,82.46) node [anchor=north west][inner sep=0.75pt]  [font=\tiny,color={rgb, 255:red, 208; green, 2; blue, 27 }  ,opacity=1 ,rotate=-359.86]  {$\phi \left( K^{2}\right)$};
\draw (388.97,122.46) node [anchor=north west][inner sep=0.75pt]  [font=\tiny,color={rgb, 255:red, 208; green, 2; blue, 27 }  ,opacity=1 ,rotate=-359.86]  {$\phi \left( K^{3}\right)$};
\draw (420.97,121.46) node [anchor=north west][inner sep=0.75pt]  [font=\tiny,color={rgb, 255:red, 208; green, 2; blue, 27 }  ,opacity=1 ,rotate=-359.86]  {$\phi \left( K^{4}\right)$};
\draw (320,197.4) node [anchor=north west][inner sep=0.75pt]  [font=\footnotesize]  {$2^{r_{2}}$};
\draw (460,218.4) node [anchor=north west][inner sep=0.75pt]  [font=\footnotesize]  {$2^{r_{1}}$};
\draw (76,203) node [anchor=north west][inner sep=0.75pt]   [align=left] {Case 1: Even distribution };
\draw (377,245) node [anchor=north west][inner sep=0.75pt]   [align=left] {Case 2: Sparse distribution};

\end{tikzpicture}
    
 \caption{Intuition for Counting of $(R,K)$ in $\Gamma(K_{\max,i}, r_1,r_2,k_1,k_2)$} \label{fig:mot_prop:number}

\end{figure} 
\vskip .15in
\noindent
\textbf{Even Case.}
Let us recall the notation of $\Lambda(K_{\max,i},r_1,r_2,k_1,k_2)$ for the set of couples $(R,K)$ such that $K \subseteq K_{\max,i}$, $|K| = 2^{k_1} \times 2^{k_2}$, $R \nsubseteq \widetilde{\phi(\Omega)}$, $|R| = 2^{r_1} \times 2^{r_2}$ with $R\cap \phi(\partial K) \neq \emptyset$ and $K \cap \phi^{-1}(\partial R) \neq \emptyset$.

For $\Gamma(K_{\max,i},r_1,r_2,k_1,k_2)$,  we can estimate $\langle h_R \circ \phi, h_K \rangle$ by (\ref{eq:innertop}) in Remark \ref{remark_top} and Proposition \ref{prop:1} with 
$$S:= R \qquad \textrm{and} \qquad T:= K.$$ 
 We have seen that we will have to consider different situations according to the corners and more precisely if some tops of $\phi(K)$ are included or not in $R$. So let us define, $\Lambda^0(K_{\max,i},r_1,r_2,k_1,k_2)$ (abbreviated as $\Lambda^0$) for the subcollection of $(R,K)$ with $R \cap \tops(\phi(K)) =\emptyset$ and the second subcollection $\Lambda^1(K_{\max,i},r_1,r_2,k_1,k_2)$ (abbreviated as $\Lambda^1$)  with $R \cap \tops(\phi(K)) \neq \emptyset$. Then, we have to estimate the corresponding partial sums 
$$\Gamma(K_{\max,i}, r_1,r_2,k_1,k_2) \leq \Gamma^0 (K_{\max,i},r_1,r_2,k_1,k_2)+\Gamma^1(K_{\max,i},r_1,r_2,k_1,k_2),$$
where 
\begin{equation} \label{def_Gamma^0}
\Gamma^0(K_{\max,i},r_1,r_2,k_1,k_2) := \sum_{\substack{K \subseteq K_{\max,i} \\ |K| = 2^{k_1} \times 2^{k_2}}} \left(\sum_{\substack{R: (R,K) \in \Lambda^0}} |\langle h_R \circ \phi, h_K \rangle|^{p'} \right)^{\frac{2}{p'}} 
\end{equation}
and
\begin{equation*}
\Gamma^1(K_{\max,i},r_1,r_2,k_1,k_2) := \sum_{\substack{K \subseteq K_{\max,i} \\ |K| = 2^{k_1} \times 2^{k_2}}} \left(\sum_{\substack{R: (R,K) \in \Lambda^1}} |\langle h_R \circ \phi, h_K \rangle|^{p'} \right)^{\frac{2}{p'}}. 
\end{equation*}
For $\Gamma^1(K_{\max,i},r_1,r_2,k_1,k_2)$, we apply the estimate (\ref{eq:innertop}) to deduce that
\begin{equation*}
\Gamma^1(K_{\max,i},r_1,r_2,k_1,k_2) \lesssim \sup_{(R,K) \in \Lambda^1} |\langle h_R \circ \phi, h_K \rangle|^2 \cdot \frac{|K_{\max,i}|}{2^{k_1}2^{k_2}} \lesssim \frac{|K|^2}{2^{k_1}2^{k_2}2^{r_1}2^{r_2}}\cdot \frac{|K_{\max,i}|}{2^{k_1}2^{k_2}} = |K_{\max,i}| 2^{-r_1}2^{-r_2}.
\end{equation*}

To derive estimates for $\Gamma^0(K_{\max,i},r_1,r_2,k_1,k_2)$, we apply  (\ref{only_vert}) or/and (\ref{only_hori}) in Proposition \ref{prop:1}. It is therefore natural to expect that the counting estimate useful for $\Gamma^0(K_{\max,i}, r_1, r_2, k_1, k_2)$ can be reduced to the counting of dyadic rectangles $R$ which intersect nontrivially with the rotated segments of $K$. 

We will give the precise definition for the counting number below and introduce Lemma \ref{count_Delta_R} which describes how $\Gamma^0$ can be estimated by appropriate applications of (\ref{only_vert}) and (\ref{only_hori}) and the corresponding counting estimates.
\begin{definition}\label{def_counting_number_R} Fix a rotation map $\phi:= \phi^{\theta}$ of angle $\theta \in [0,2\pi]$ and $K_{\max,i} \in \mathbb{D}^{\alpha}$. 
Define for each $K \subseteq K_{\max,i} \in \mathbb{D}^{\alpha}$, the number of dyadic rectangles with fixed dimension that have non-trivial intersections with the rotated vertical (or horizontal) segments of $K$ as follows:
$$
\mathcal{N}^v_{r_1,r_2}(\phi(K)) := \# \{R: (R, K) \in \Lambda(K_{\max,i}, r_1, r_2, k_1, k_2), R\cap \phi\big(l^v_1(K)\cup l^v_2(K)\cup l^v_3(K)\big)\neq\emptyset\}
$$
and 
$$
\mathcal{N}^h_{r_1,r_2}(\phi(K)) := \# \{R: (R,K) \in  \Lambda(K_{\max,i}, r_1, r_2, k_1, k_2), R\cap \phi\big(l^h_1(K)\cup l^h_2(K)\cup l^h_3(K)\big)\neq\emptyset \}.
$$
\end{definition}
\begin{remark}\label{trivial_top_vert}
Based on the definition of the counting numbers $\mathcal{N}^v_{r_1,r_2}(\phi(K))$ and $\mathcal{N}^h_{r_1,r_2}(\phi(K))$, it is not hard to observe that in the even situation
$$
\min\big(\mathcal{N}^v_{r_1,r_2}(\phi(K)), \mathcal{N}^h_{r_1,r_2}(\phi(K)) \big) \geq \#\text{top}(\phi(K)).
$$
\end{remark}
\begin{lemma}\label{count_Delta_R}
Let us recall the set $\Lambda^0 := \Lambda^0(K_{\max,i},r_1,r_2,k_1,k_2)$ of couples $(R,K)$ such that $K \subseteq K_{\max,i}$, $|K| = 2^{k_1} \times 2^{k_2}$, $R \nsubseteq \widetilde{\phi(\Omega)}$, $|R| = 2^{r_1} \times 2^{r_2}$ with $R\cap \phi(\partial K) \neq \emptyset$, $K \cap \phi^{-1}(\partial R) \neq \emptyset$ and $R \cap \tops(\phi(K)) =\emptyset$.
Then 
\begin{align} \label{Gamma^0_unquantified}
&\Gamma^0(K_{\max,i},r_1,r_2,k_1,k_2) := \sum_{\substack{K \subseteq K_{\max,i} \\ |K| = 2^{k_1} \times 2^{k_2}}} \left(\sum_{\substack{R: (R,K) \in \Lambda^0}} |\langle h_R \circ \phi, h_K \rangle|^{p'} \right)^{\frac{2}{p'}}  \nonumber \\
\leq& \left(\sup_{(R,K) \in \Lambda^0} \frac{\min(R_2^2 |\tan \theta|, R_1^2 |\cot\theta|)}{|K|^{1/2} |R|^{1/2}} \right)^2 \cdot \left(\sup_{\substack{K \subseteq K_{\max,i} \\ |K|= 2^{k_1} \times 2^{k_2}}}\mathcal{N}^v_{r_1,r_2}(\phi(K)) - \#\text{top}(\phi(K))\right)^{\frac{2}{p'}} \cdot \frac{|K_{\max,i}|}{2^{k_1+k_2}}+ \nonumber\\
& \left(\sup_{(R,K) \in \Lambda^0}  \frac{\min(R_1^2 |\tan \theta|, R_2^2 |\cot\theta|)}{|K|^{1/2} |R|^{1/2}}\right)^2 \cdot \left(\sup_{\substack{K \subseteq K_{\max,i} \\ |K|= 2^{k_1} \times 2^{k_2}}}\mathcal{N}^h_{r_1,r_2}(\phi(K)) - \#\text{top}(\phi(K))\right)^{\frac{2}{p'}} \cdot \frac{|K_{\max,i}|}{2^{k_1+k_2}}.
\end{align}
\end{lemma}

\begin{proof}
We will provide a sketch of the proof. We first notice that since there are no tops of $\phi(K)$ intersecting $R$ for any $(R,K) \in \Lambda^0$, 
\begin{align*}
|\langle h_K \circ \phi^{-1}, h_R \rangle| \leq \Delta_K^v(R) + \Delta_K^h(R)
\end{align*}
by the reasoning in the proof of Proposition \ref{prop:1}.
Therefore, for any fixed $K \in $
\begin{align}
 \left(\sum_{\substack{R: (R,K) \in \Lambda^0}} |\langle h_R \circ \phi, h_K \rangle|^{p'} \right)^{\frac{2}{p'}} \leq& \left(\sum_{\substack{R: (R,K) \in \Lambda^0}}  (\Delta^v_K(R) + \Delta^h_K(R))^{p'} \right)^{\frac{2}{p'}} \nonumber \\
 \lesssim_{p'} & \left(\sum_{\substack{R: (R,K) \in \Lambda^0}}  (\Delta^v_K(R))^{p'} + (\Delta^h_K(R))^{p'} \right)^{\frac{2}{p'}} \nonumber \\
 \lesssim &  \left(\sum_{\substack{R: (R,K) \in \Lambda^0}}  (\Delta^v_K(R))^{p'}\right)^{\frac{2}{p'}} +  \left(\sum_{\substack{R: (R,K) \in \Lambda^0}} (\Delta^h_K(R))^{p'} \right)^{\frac{2}{p'}} \label{proof_lemma_Gamma^0}
\end{align}
and the rest of the proof follows from straightforward estimates on (\ref{proof_lemma_Gamma^0}).
\end{proof}
We can also estimate $\langle h_R \circ \phi, h_K \rangle$ for $(R,K) \in \Lambda$ by (\ref{only_vert}) or/and (\ref{only_hori}) with 
$$S:= K,$$ 
$$T:= R,$$
$$\phi := \phi^{-\theta}$$ as illustrated in Proposition \ref{prop:1}. One simple and important observation is that due to Corollary \ref{cor:r1}, for all $K$ such that $(R,K) \in\Lambda(K_{\max,i},r_1,r_2,k_1,k_2)$, $K \cap \text{top}(\phi^{-1}(R)) = \emptyset$. Moreover, the vertical boundary of $R$ never intersects with $K$ for $(R,K) \in \Lambda(K_{\max,i}, r_1, r_2, k_1, k_2)$, hence we only have Case (2) and the estimate (\ref{only_hori}) is sufficient. 
\begin{lemma} \label{count_Delta_K}
\begin{align*}
&\Gamma(K_{\max,i},r_1,r_2,k_1,k_2) \\
\leq &\left(\sup_{(R,K) \in \Lambda} \frac{\min(K_1^2 |\tan \theta|, K_2^2 |\cot\theta|)}{|K|^{1/2} |R|^{1/2}}\right)^2 \cdot \sup_{\substack{K \subseteq K_{\max,i} \\ |K|= 2^{k_1} \times 2^{k_2}}}\left(\mathcal{N}^v_{r_1,r_2}(\phi(K)) + \mathcal{N}^h_{r_1,r_2}(\phi(K)) \right)^{\frac{2}{p'}} \cdot \frac{|K_{\max,i}|}{2^{k_1+k_2}}.
\end{align*}
\end{lemma}

To further quantify $\Gamma^0$ and $\Gamma^1$, we would now establish the bounds for $\mathcal{N}^v_{r_1,r_2}(\phi(K))$ and $\mathcal{N}^h_{r_1,r_2}(\phi(K))$. Figure \ref{fig:mot_prop:number-bis} provides a geometric interpretation of $\mathcal{N}^v_{r_1,r_2}(\phi(K))$ which refers to the number of dyadic rectangles with the orange pattern, and $\mathcal{N}^h_{r_1,r_2}(\phi(K))$ which represents the number of dyadic rectangles with the green pattern.

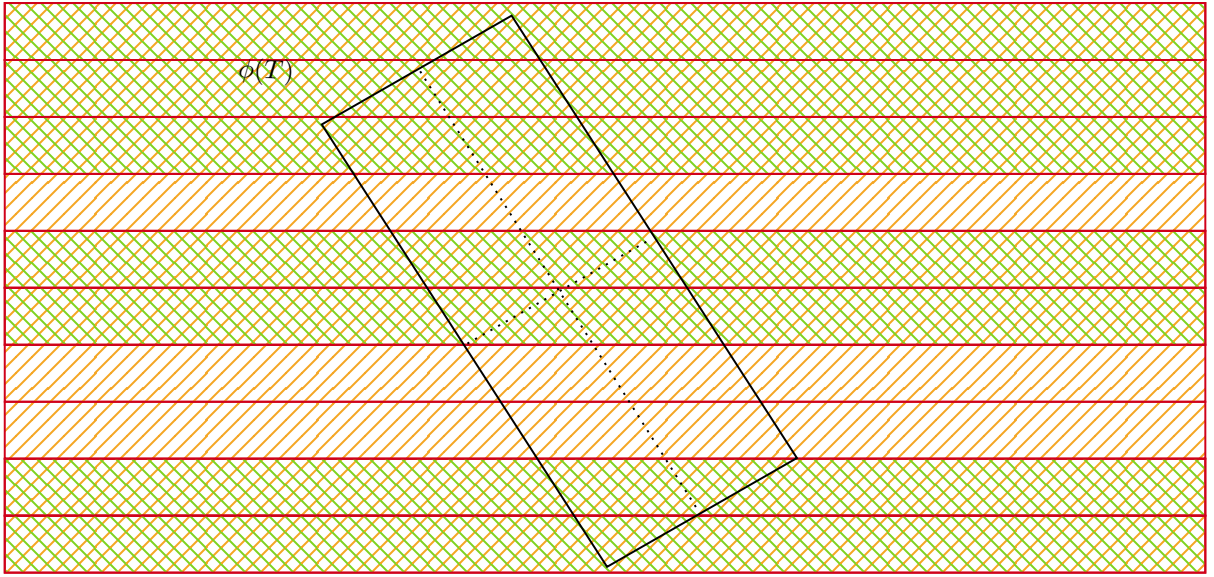
\begin{figure}
\centering

 
\tikzset{
pattern size/.store in=\mcSize, 
pattern size = 5pt,
pattern thickness/.store in=\mcThickness, 
pattern thickness = 0.3pt,
pattern radius/.store in=\mcRadius, 
pattern radius = 1pt}
\makeatletter
\pgfutil@ifundefined{pgf@pattern@name@_oj655gdon}{
\pgfdeclarepatternformonly[\mcThickness,\mcSize]{_oj655gdon}
{\pgfqpoint{0pt}{0pt}}
{\pgfpoint{\mcSize+\mcThickness}{\mcSize+\mcThickness}}
{\pgfpoint{\mcSize}{\mcSize}}
{
\pgfsetcolor{\tikz@pattern@color}
\pgfsetlinewidth{\mcThickness}
\pgfpathmoveto{\pgfqpoint{0pt}{0pt}}
\pgfpathlineto{\pgfpoint{\mcSize+\mcThickness}{\mcSize+\mcThickness}}
\pgfusepath{stroke}
}}
\makeatother

 
\tikzset{
pattern size/.store in=\mcSize, 
pattern size = 5pt,
pattern thickness/.store in=\mcThickness, 
pattern thickness = 0.3pt,
pattern radius/.store in=\mcRadius, 
pattern radius = 1pt}
\makeatletter
\pgfutil@ifundefined{pgf@pattern@name@_b42n3qwnl}{
\pgfdeclarepatternformonly[\mcThickness,\mcSize]{_b42n3qwnl}
{\pgfqpoint{0pt}{0pt}}
{\pgfpoint{\mcSize+\mcThickness}{\mcSize+\mcThickness}}
{\pgfpoint{\mcSize}{\mcSize}}
{
\pgfsetcolor{\tikz@pattern@color}
\pgfsetlinewidth{\mcThickness}
\pgfpathmoveto{\pgfqpoint{0pt}{0pt}}
\pgfpathlineto{\pgfpoint{\mcSize+\mcThickness}{\mcSize+\mcThickness}}
\pgfusepath{stroke}
}}
\makeatother

 
\tikzset{
pattern size/.store in=\mcSize, 
pattern size = 5pt,
pattern thickness/.store in=\mcThickness, 
pattern thickness = 0.3pt,
pattern radius/.store in=\mcRadius, 
pattern radius = 1pt}
\makeatletter
\pgfutil@ifundefined{pgf@pattern@name@_mao6v6af7}{
\pgfdeclarepatternformonly[\mcThickness,\mcSize]{_mao6v6af7}
{\pgfqpoint{0pt}{0pt}}
{\pgfpoint{\mcSize+\mcThickness}{\mcSize+\mcThickness}}
{\pgfpoint{\mcSize}{\mcSize}}
{
\pgfsetcolor{\tikz@pattern@color}
\pgfsetlinewidth{\mcThickness}
\pgfpathmoveto{\pgfqpoint{0pt}{0pt}}
\pgfpathlineto{\pgfpoint{\mcSize+\mcThickness}{\mcSize+\mcThickness}}
\pgfusepath{stroke}
}}
\makeatother

 
\tikzset{
pattern size/.store in=\mcSize, 
pattern size = 5pt,
pattern thickness/.store in=\mcThickness, 
pattern thickness = 0.3pt,
pattern radius/.store in=\mcRadius, 
pattern radius = 1pt}
\makeatletter
\pgfutil@ifundefined{pgf@pattern@name@_yq7yxl0t4}{
\pgfdeclarepatternformonly[\mcThickness,\mcSize]{_yq7yxl0t4}
{\pgfqpoint{0pt}{0pt}}
{\pgfpoint{\mcSize+\mcThickness}{\mcSize+\mcThickness}}
{\pgfpoint{\mcSize}{\mcSize}}
{
\pgfsetcolor{\tikz@pattern@color}
\pgfsetlinewidth{\mcThickness}
\pgfpathmoveto{\pgfqpoint{0pt}{0pt}}
\pgfpathlineto{\pgfpoint{\mcSize+\mcThickness}{\mcSize+\mcThickness}}
\pgfusepath{stroke}
}}
\makeatother

 
\tikzset{
pattern size/.store in=\mcSize, 
pattern size = 5pt,
pattern thickness/.store in=\mcThickness, 
pattern thickness = 0.3pt,
pattern radius/.store in=\mcRadius, 
pattern radius = 1pt}
\makeatletter
\pgfutil@ifundefined{pgf@pattern@name@_e5dh29c07}{
\pgfdeclarepatternformonly[\mcThickness,\mcSize]{_e5dh29c07}
{\pgfqpoint{0pt}{0pt}}
{\pgfpoint{\mcSize+\mcThickness}{\mcSize+\mcThickness}}
{\pgfpoint{\mcSize}{\mcSize}}
{
\pgfsetcolor{\tikz@pattern@color}
\pgfsetlinewidth{\mcThickness}
\pgfpathmoveto{\pgfqpoint{0pt}{0pt}}
\pgfpathlineto{\pgfpoint{\mcSize+\mcThickness}{\mcSize+\mcThickness}}
\pgfusepath{stroke}
}}
\makeatother

 
\tikzset{
pattern size/.store in=\mcSize, 
pattern size = 5pt,
pattern thickness/.store in=\mcThickness, 
pattern thickness = 0.3pt,
pattern radius/.store in=\mcRadius, 
pattern radius = 1pt}
\makeatletter
\pgfutil@ifundefined{pgf@pattern@name@_ohx0c5tvm}{
\pgfdeclarepatternformonly[\mcThickness,\mcSize]{_ohx0c5tvm}
{\pgfqpoint{0pt}{0pt}}
{\pgfpoint{\mcSize+\mcThickness}{\mcSize+\mcThickness}}
{\pgfpoint{\mcSize}{\mcSize}}
{
\pgfsetcolor{\tikz@pattern@color}
\pgfsetlinewidth{\mcThickness}
\pgfpathmoveto{\pgfqpoint{0pt}{0pt}}
\pgfpathlineto{\pgfpoint{\mcSize+\mcThickness}{\mcSize+\mcThickness}}
\pgfusepath{stroke}
}}
\makeatother

 
\tikzset{
pattern size/.store in=\mcSize, 
pattern size = 5pt,
pattern thickness/.store in=\mcThickness, 
pattern thickness = 0.3pt,
pattern radius/.store in=\mcRadius, 
pattern radius = 1pt}
\makeatletter
\pgfutil@ifundefined{pgf@pattern@name@_64wp4ygvt}{
\pgfdeclarepatternformonly[\mcThickness,\mcSize]{_64wp4ygvt}
{\pgfqpoint{0pt}{0pt}}
{\pgfpoint{\mcSize+\mcThickness}{\mcSize+\mcThickness}}
{\pgfpoint{\mcSize}{\mcSize}}
{
\pgfsetcolor{\tikz@pattern@color}
\pgfsetlinewidth{\mcThickness}
\pgfpathmoveto{\pgfqpoint{0pt}{0pt}}
\pgfpathlineto{\pgfpoint{\mcSize+\mcThickness}{\mcSize+\mcThickness}}
\pgfusepath{stroke}
}}
\makeatother

 
\tikzset{
pattern size/.store in=\mcSize, 
pattern size = 5pt,
pattern thickness/.store in=\mcThickness, 
pattern thickness = 0.3pt,
pattern radius/.store in=\mcRadius, 
pattern radius = 1pt}
\makeatletter
\pgfutil@ifundefined{pgf@pattern@name@_k1rz6invw}{
\pgfdeclarepatternformonly[\mcThickness,\mcSize]{_k1rz6invw}
{\pgfqpoint{0pt}{0pt}}
{\pgfpoint{\mcSize+\mcThickness}{\mcSize+\mcThickness}}
{\pgfpoint{\mcSize}{\mcSize}}
{
\pgfsetcolor{\tikz@pattern@color}
\pgfsetlinewidth{\mcThickness}
\pgfpathmoveto{\pgfqpoint{0pt}{0pt}}
\pgfpathlineto{\pgfpoint{\mcSize+\mcThickness}{\mcSize+\mcThickness}}
\pgfusepath{stroke}
}}
\makeatother

 
\tikzset{
pattern size/.store in=\mcSize, 
pattern size = 5pt,
pattern thickness/.store in=\mcThickness, 
pattern thickness = 0.3pt,
pattern radius/.store in=\mcRadius, 
pattern radius = 1pt}
\makeatletter
\pgfutil@ifundefined{pgf@pattern@name@_cqt97nhuv}{
\pgfdeclarepatternformonly[\mcThickness,\mcSize]{_cqt97nhuv}
{\pgfqpoint{0pt}{-\mcThickness}}
{\pgfpoint{\mcSize}{\mcSize}}
{\pgfpoint{\mcSize}{\mcSize}}
{
\pgfsetcolor{\tikz@pattern@color}
\pgfsetlinewidth{\mcThickness}
\pgfpathmoveto{\pgfqpoint{0pt}{\mcSize}}
\pgfpathlineto{\pgfpoint{\mcSize+\mcThickness}{-\mcThickness}}
\pgfusepath{stroke}
}}
\makeatother

 
\tikzset{
pattern size/.store in=\mcSize, 
pattern size = 5pt,
pattern thickness/.store in=\mcThickness, 
pattern thickness = 0.3pt,
pattern radius/.store in=\mcRadius, 
pattern radius = 1pt}
\makeatletter
\pgfutil@ifundefined{pgf@pattern@name@_hww6ky3iq}{
\pgfdeclarepatternformonly[\mcThickness,\mcSize]{_hww6ky3iq}
{\pgfqpoint{0pt}{-\mcThickness}}
{\pgfpoint{\mcSize}{\mcSize}}
{\pgfpoint{\mcSize}{\mcSize}}
{
\pgfsetcolor{\tikz@pattern@color}
\pgfsetlinewidth{\mcThickness}
\pgfpathmoveto{\pgfqpoint{0pt}{\mcSize}}
\pgfpathlineto{\pgfpoint{\mcSize+\mcThickness}{-\mcThickness}}
\pgfusepath{stroke}
}}
\makeatother

 
\tikzset{
pattern size/.store in=\mcSize, 
pattern size = 5pt,
pattern thickness/.store in=\mcThickness, 
pattern thickness = 0.3pt,
pattern radius/.store in=\mcRadius, 
pattern radius = 1pt}
\makeatletter
\pgfutil@ifundefined{pgf@pattern@name@_dl75cn01q}{
\pgfdeclarepatternformonly[\mcThickness,\mcSize]{_dl75cn01q}
{\pgfqpoint{0pt}{0pt}}
{\pgfpoint{\mcSize+\mcThickness}{\mcSize+\mcThickness}}
{\pgfpoint{\mcSize}{\mcSize}}
{
\pgfsetcolor{\tikz@pattern@color}
\pgfsetlinewidth{\mcThickness}
\pgfpathmoveto{\pgfqpoint{0pt}{0pt}}
\pgfpathlineto{\pgfpoint{\mcSize+\mcThickness}{\mcSize+\mcThickness}}
\pgfusepath{stroke}
}}
\makeatother

 
\tikzset{
pattern size/.store in=\mcSize, 
pattern size = 5pt,
pattern thickness/.store in=\mcThickness, 
pattern thickness = 0.3pt,
pattern radius/.store in=\mcRadius, 
pattern radius = 1pt}
\makeatletter
\pgfutil@ifundefined{pgf@pattern@name@_ivue1vqin}{
\pgfdeclarepatternformonly[\mcThickness,\mcSize]{_ivue1vqin}
{\pgfqpoint{0pt}{0pt}}
{\pgfpoint{\mcSize+\mcThickness}{\mcSize+\mcThickness}}
{\pgfpoint{\mcSize}{\mcSize}}
{
\pgfsetcolor{\tikz@pattern@color}
\pgfsetlinewidth{\mcThickness}
\pgfpathmoveto{\pgfqpoint{0pt}{0pt}}
\pgfpathlineto{\pgfpoint{\mcSize+\mcThickness}{\mcSize+\mcThickness}}
\pgfusepath{stroke}
}}
\makeatother

 
\tikzset{
pattern size/.store in=\mcSize, 
pattern size = 5pt,
pattern thickness/.store in=\mcThickness, 
pattern thickness = 0.3pt,
pattern radius/.store in=\mcRadius, 
pattern radius = 1pt}
\makeatletter
\pgfutil@ifundefined{pgf@pattern@name@_rny391xrz}{
\pgfdeclarepatternformonly[\mcThickness,\mcSize]{_rny391xrz}
{\pgfqpoint{0pt}{-\mcThickness}}
{\pgfpoint{\mcSize}{\mcSize}}
{\pgfpoint{\mcSize}{\mcSize}}
{
\pgfsetcolor{\tikz@pattern@color}
\pgfsetlinewidth{\mcThickness}
\pgfpathmoveto{\pgfqpoint{0pt}{\mcSize}}
\pgfpathlineto{\pgfpoint{\mcSize+\mcThickness}{-\mcThickness}}
\pgfusepath{stroke}
}}
\makeatother

 
\tikzset{
pattern size/.store in=\mcSize, 
pattern size = 5pt,
pattern thickness/.store in=\mcThickness, 
pattern thickness = 0.3pt,
pattern radius/.store in=\mcRadius, 
pattern radius = 1pt}
\makeatletter
\pgfutil@ifundefined{pgf@pattern@name@_ro6v5gph0}{
\pgfdeclarepatternformonly[\mcThickness,\mcSize]{_ro6v5gph0}
{\pgfqpoint{0pt}{-\mcThickness}}
{\pgfpoint{\mcSize}{\mcSize}}
{\pgfpoint{\mcSize}{\mcSize}}
{
\pgfsetcolor{\tikz@pattern@color}
\pgfsetlinewidth{\mcThickness}
\pgfpathmoveto{\pgfqpoint{0pt}{\mcSize}}
\pgfpathlineto{\pgfpoint{\mcSize+\mcThickness}{-\mcThickness}}
\pgfusepath{stroke}
}}
\makeatother

 
\tikzset{
pattern size/.store in=\mcSize, 
pattern size = 5pt,
pattern thickness/.store in=\mcThickness, 
pattern thickness = 0.3pt,
pattern radius/.store in=\mcRadius, 
pattern radius = 1pt}
\makeatletter
\pgfutil@ifundefined{pgf@pattern@name@_albnzvzo7}{
\pgfdeclarepatternformonly[\mcThickness,\mcSize]{_albnzvzo7}
{\pgfqpoint{0pt}{-\mcThickness}}
{\pgfpoint{\mcSize}{\mcSize}}
{\pgfpoint{\mcSize}{\mcSize}}
{
\pgfsetcolor{\tikz@pattern@color}
\pgfsetlinewidth{\mcThickness}
\pgfpathmoveto{\pgfqpoint{0pt}{\mcSize}}
\pgfpathlineto{\pgfpoint{\mcSize+\mcThickness}{-\mcThickness}}
\pgfusepath{stroke}
}}
\makeatother

 
\tikzset{
pattern size/.store in=\mcSize, 
pattern size = 5pt,
pattern thickness/.store in=\mcThickness, 
pattern thickness = 0.3pt,
pattern radius/.store in=\mcRadius, 
pattern radius = 1pt}
\makeatletter
\pgfutil@ifundefined{pgf@pattern@name@_h5dogccwd}{
\pgfdeclarepatternformonly[\mcThickness,\mcSize]{_h5dogccwd}
{\pgfqpoint{0pt}{-\mcThickness}}
{\pgfpoint{\mcSize}{\mcSize}}
{\pgfpoint{\mcSize}{\mcSize}}
{
\pgfsetcolor{\tikz@pattern@color}
\pgfsetlinewidth{\mcThickness}
\pgfpathmoveto{\pgfqpoint{0pt}{\mcSize}}
\pgfpathlineto{\pgfpoint{\mcSize+\mcThickness}{-\mcThickness}}
\pgfusepath{stroke}
}}
\makeatother

 
\tikzset{
pattern size/.store in=\mcSize, 
pattern size = 5pt,
pattern thickness/.store in=\mcThickness, 
pattern thickness = 0.3pt,
pattern radius/.store in=\mcRadius, 
pattern radius = 1pt}
\makeatletter
\pgfutil@ifundefined{pgf@pattern@name@_76ml7cdyf}{
\pgfdeclarepatternformonly[\mcThickness,\mcSize]{_76ml7cdyf}
{\pgfqpoint{0pt}{-\mcThickness}}
{\pgfpoint{\mcSize}{\mcSize}}
{\pgfpoint{\mcSize}{\mcSize}}
{
\pgfsetcolor{\tikz@pattern@color}
\pgfsetlinewidth{\mcThickness}
\pgfpathmoveto{\pgfqpoint{0pt}{\mcSize}}
\pgfpathlineto{\pgfpoint{\mcSize+\mcThickness}{-\mcThickness}}
\pgfusepath{stroke}
}}
\makeatother
\tikzset{every picture/.style={line width=0.75pt}} 

\begin{tikzpicture}[x=0.75pt,y=0.75pt,yscale=-1,xscale=1]

\draw  [color={rgb, 255:red, 208; green, 2; blue, 27 }  ,draw opacity=1 ][pattern=_oj655gdon,pattern size=6pt,pattern thickness=0.75pt,pattern radius=0pt, pattern color={rgb, 255:red, 245; green, 166; blue, 35}] (22,992.75) -- (627.5,992.75) -- (627.5,1021.5) -- (22,1021.5) -- cycle ;
\draw  [color={rgb, 255:red, 208; green, 2; blue, 27 }  ,draw opacity=1 ][pattern=_b42n3qwnl,pattern size=6pt,pattern thickness=0.75pt,pattern radius=0pt, pattern color={rgb, 255:red, 245; green, 166; blue, 35}] (22,1050.25) -- (627.5,1050.25) -- (627.5,1079) -- (22,1079) -- cycle ;
\draw  [color={rgb, 255:red, 208; green, 2; blue, 27 }  ,draw opacity=1 ][pattern=_mao6v6af7,pattern size=6pt,pattern thickness=0.75pt,pattern radius=0pt, pattern color={rgb, 255:red, 245; green, 166; blue, 35}] (22,1021.5) -- (627.5,1021.5) -- (627.5,1050.25) -- (22,1050.25) -- cycle ;
\draw  [color={rgb, 255:red, 208; green, 2; blue, 27 }  ,draw opacity=1 ][pattern=_yq7yxl0t4,pattern size=6pt,pattern thickness=0.75pt,pattern radius=0pt, pattern color={rgb, 255:red, 245; green, 166; blue, 35}] (22,1079) -- (627.5,1079) -- (627.5,1107.75) -- (22,1107.75) -- cycle ;
\draw  [color={rgb, 255:red, 208; green, 2; blue, 27 }  ,draw opacity=1 ][pattern=_e5dh29c07,pattern size=6pt,pattern thickness=0.75pt,pattern radius=0pt, pattern color={rgb, 255:red, 245; green, 166; blue, 35}] (22,1107.75) -- (627.5,1107.75) -- (627.5,1136.5) -- (22,1136.5) -- cycle ;
\draw  [color={rgb, 255:red, 208; green, 2; blue, 27 }  ,draw opacity=1 ][pattern=_ohx0c5tvm,pattern size=6pt,pattern thickness=0.75pt,pattern radius=0pt, pattern color={rgb, 255:red, 245; green, 166; blue, 35}] (22,1165.25) -- (627.5,1165.25) -- (627.5,1194) -- (22,1194) -- cycle ;
\draw  [color={rgb, 255:red, 208; green, 2; blue, 27 }  ,draw opacity=1 ][pattern=_64wp4ygvt,pattern size=6pt,pattern thickness=0.75pt,pattern radius=0pt, pattern color={rgb, 255:red, 245; green, 166; blue, 35}] (22,1136.5) -- (627.5,1136.5) -- (627.5,1165.25) -- (22,1165.25) -- cycle ;
\draw  [color={rgb, 255:red, 208; green, 2; blue, 27 }  ,draw opacity=1 ][pattern=_k1rz6invw,pattern size=6pt,pattern thickness=0.75pt,pattern radius=0pt, pattern color={rgb, 255:red, 245; green, 166; blue, 35}] (22,1194) -- (627.5,1194) -- (627.5,1222.75) -- (22,1222.75) -- cycle ;
\draw  [color={rgb, 255:red, 208; green, 2; blue, 27 }  ,draw opacity=1 ][pattern=_cqt97nhuv,pattern size=6pt,pattern thickness=0.75pt,pattern radius=0pt, pattern color={rgb, 255:red, 126; green, 211; blue, 33}] (22,1021.5) -- (627.5,1021.5) -- (627.5,1050.25) -- (22,1050.25) -- cycle ;
\draw  [color={rgb, 255:red, 208; green, 2; blue, 27 }  ,draw opacity=1 ][pattern=_hww6ky3iq,pattern size=6pt,pattern thickness=0.75pt,pattern radius=0pt, pattern color={rgb, 255:red, 126; green, 211; blue, 33}] (22,1079) -- (627.5,1079) -- (627.5,1107.75) -- (22,1107.75) -- cycle ;
\draw  [color={rgb, 255:red, 208; green, 2; blue, 27 }  ,draw opacity=1 ][pattern=_dl75cn01q,pattern size=6pt,pattern thickness=0.75pt,pattern radius=0pt, pattern color={rgb, 255:red, 245; green, 166; blue, 35}] (22,1222.75) -- (627.5,1222.75) -- (627.5,1251.5) -- (22,1251.5) -- cycle ;
\draw  [color={rgb, 255:red, 208; green, 2; blue, 27 }  ,draw opacity=1 ][pattern=_ivue1vqin,pattern size=6pt,pattern thickness=0.75pt,pattern radius=0pt, pattern color={rgb, 255:red, 245; green, 166; blue, 35}] (22,964) -- (627.5,964) -- (627.5,992.75) -- (22,992.75) -- cycle ;
\draw  [color={rgb, 255:red, 208; green, 2; blue, 27 }  ,draw opacity=1 ][pattern=_rny391xrz,pattern size=6pt,pattern thickness=0.75pt,pattern radius=0pt, pattern color={rgb, 255:red, 126; green, 211; blue, 33}] (22,964) -- (627.5,964) -- (627.5,992.75) -- (22,992.75) -- cycle ;
\draw  [color={rgb, 255:red, 208; green, 2; blue, 27 }  ,draw opacity=1 ][pattern=_ro6v5gph0,pattern size=6pt,pattern thickness=0.75pt,pattern radius=0pt, pattern color={rgb, 255:red, 126; green, 211; blue, 33}] (22,992.75) -- (627.5,992.75) -- (627.5,1021.5) -- (22,1021.5) -- cycle ;
\draw  [color={rgb, 255:red, 208; green, 2; blue, 27 }  ,draw opacity=1 ][pattern=_albnzvzo7,pattern size=6pt,pattern thickness=0.75pt,pattern radius=0pt, pattern color={rgb, 255:red, 126; green, 211; blue, 33}] (22,1107.75) -- (627.5,1107.75) -- (627.5,1136.5) -- (22,1136.5) -- cycle ;
\draw  [color={rgb, 255:red, 208; green, 2; blue, 27 }  ,draw opacity=1 ][pattern=_h5dogccwd,pattern size=6pt,pattern thickness=0.75pt,pattern radius=0pt, pattern color={rgb, 255:red, 126; green, 211; blue, 33}] (22,1194) -- (627.5,1194) -- (627.5,1222.75) -- (22,1222.75) -- cycle ;
\draw  [color={rgb, 255:red, 208; green, 2; blue, 27 }  ,draw opacity=1 ][pattern=_76ml7cdyf,pattern size=6pt,pattern thickness=0.75pt,pattern radius=0pt, pattern color={rgb, 255:red, 126; green, 211; blue, 33}] (22,1222.75) -- (627.5,1222.75) -- (627.5,1251.5) -- (22,1251.5) -- cycle ;
\draw   (181.88,1025.28) -- (277.63,970.36) -- (421.53,1193.65) -- (325.78,1248.57) -- cycle ;
\draw  [dash pattern={on 0.84pt off 2.51pt}]  (231.51,998.45) -- (371.9,1220.48) ;
\draw  [dash pattern={on 0.84pt off 2.51pt}]  (255.31,1135.89) -- (348.1,1083.04) ;

\draw (138.67,989.89) node [anchor=north west][inner sep=0.75pt]  [rotate=-1.87]  {$\phi ( T)$};

\end{tikzpicture}

\caption{Geometric description of Proposition \ref{prop:number}} \label{fig:mot_prop:number-bis}
\end{figure}

\begin{proposition} \label{prop:number} 
$\mathcal{N}^v_{r_1,r_2}(\phi(K))$ and $\mathcal{N}^h_{r_1,r_2}(\phi(K))$ (Definition \ref{def_counting_number_R}) satisfy the following estimates:
$$\mathcal{N}^v_{r_1,r_2}(\phi(K)) \leq 6.2^{k_2}\max(\frac{|\cos \theta|}{2^{r_2}}, \frac{|\sin\theta|}{2^{r_1}})+2$$
and
$$\mathcal{N}^h_{r_1,r_2}(\phi(K)) \leq 6.2^{k_1}  \max(\frac{|\cos \theta|}{2^{r_1}}, \frac{|\sin\theta|}{2^{r_2}})+2.$$
\end{proposition}

\begin{proof} [Proof of Proposition \ref{prop:number}]
Let us consider 
\begin{align*}
\mathcal {R}^v := \{R: (R, K) \in \Lambda(K_{\max,i}, r_1, r_2, k_1, k_2), R\cap \phi\big(l^v_1(K)\cup l^v_2(K)\cup l^v_3(K)\big)\neq\emptyset\}, \\
\mathcal {R}^h := \{R: (R,K) \in  \Lambda(K_{\max,i}, r_1, r_2, k_1, k_2), R\cap \phi\big(l^h_1(K)\cup l^h_2(K)\cup l^h_3(K)\big)\neq\emptyset \}.
\end{align*}
so that $\# \mathcal {R}^v = \mathcal{N}^v_{r_1,r_2}(\phi(K))$ and $\# \mathcal {R}^h = \mathcal{N}^h_{r_1,r_2}(\phi(K))$ by Definition \ref{def_counting_number_R}. After translation, one can assume that the rectangle $K$, is described by its four vertices given by 
$$K = \begin{bmatrix}
0 & K_1 & K_1 & 0 \\
0 & 0 & K_2 & K_2
\end{bmatrix}.$$
And so $\phi(K)$ is given by the following four vertices:
\begin{align}
A:= 
\begin{bmatrix}
\cos \theta & -\sin \theta \\
\sin \theta & \cos \theta
\end{bmatrix} 
\begin{bmatrix}
0 \\ 0
\end{bmatrix}
& = 
\begin{bmatrix}
0 \\ 0
\end{bmatrix} \nonumber \\ 
B:= 
\begin{bmatrix}
\cos \theta & -\sin \theta \\
\sin \theta & \cos \theta
\end{bmatrix} 
\begin{bmatrix}
K_1 \\ 0
\end{bmatrix}
& = 
\begin{bmatrix}
K_1 \cos \theta \\ K_1 \sin\theta
\end{bmatrix} \nonumber \\
C:=
\begin{bmatrix}
\cos \theta & -\sin \theta \\
\sin \theta & \cos \theta
\end{bmatrix} 
\begin{bmatrix}
K_1 \\ K_2
\end{bmatrix}
& = 
\begin{bmatrix}
K_1 \cos\theta - K_2 \sin \theta \\ K_1 \sin\theta + K_2 \cos \theta
\end{bmatrix} \nonumber \\ 
D:=
\begin{bmatrix}
\cos \theta & -\sin \theta \\
\sin \theta & \cos \theta
\end{bmatrix} 
\begin{bmatrix}
0 \\ K_2
\end{bmatrix}
& = 
\begin{bmatrix}
-K_2 \sin \theta \\ K_2 \cos \theta
\end{bmatrix}.
\end{align}
We consider rectangles $R\in\mathcal{R}^v$ so it meets only rotated vertical segments, which are $[A,D]$, $[B,C]$ or $[\frac{A+B}{2},\frac{C+D}{2}]$. Consider a rectangle $R$ which intersects $[A,D]$.
Therefore, the projection of the rotated boundary 
$(AD)$ onto $y$-direction has length 
\begin{equation}
|D_y - A_y| := K_2 |\cos \theta|.
\end{equation}
A similar reasoning shows that the length of the projection of $(AD)$ onto $x$-direction is
\begin{equation}
|D_x - A_x| := K_2 |\sin \theta|.
\end{equation}
Hence, the number of rectangles in $\mathcal{R}$ which intersect with $(AD)$ can be majorized by 
\begin{equation}
2 \max(\frac{K_2 |\cos \theta|}{2^{r_2}}, \frac{K_2 |\sin\theta|}{2^{r_1}}).
\end{equation}
A similar reasoning hold for the two other rotated vertical segments.

The same reasoning generates the following estimates for the number of rectangles in $\mathcal{R}^h$:
\begin{align}
 2\max(\frac{K_1|\sin \theta|}{2^{r_2}},\frac{K_1|\cos\theta|}{2^{r_1}}).
\end{align}
\end{proof}
\begin{remark}
We shall point out that in the current setting $2^{k_1} + 2^{k_2} \lesssim 2^{r_2}$ and the dyadic grid is chosen so that $\phi(K)$ does not intersect with the vertical boundaries of $R$ for $(R,K) \in \Lambda(K_{\max,i}, r_1, r_2, k_1, k_2)$. For $R$ intersecting with $\phi(l_1^v(K))\cup \phi(l_2^v(K)) \cup \phi(l_3^v(K))$ nontrivially, if $2^{k_2} |\cos \theta| < 2^{r_2}$, then we have the exact equality
$$
\mathcal{N}^v_{r_1,r_2}(\phi(K))= \#\text{top}(\phi(K)),
$$
which is a more precise estimate than the trivial observation in Remark \ref{trivial_top_vert}. Similarly, for $R$ intersecting with $\phi(l_1^h(K))\cup \phi(l_2^h(K)) \cup \phi(l_3^h(K))$ nontrivially, if $2^{k_2} |\sin \theta| < 2^{r_1}$, then 
$$
\mathcal{N}^h_{r_1,r_2}(\phi(K))= \#\text{top}(\phi(K)).
$$

\end{remark}
This ends the discussion of the even case. 
\vskip .15in
\noindent
\textbf{Sparse Case.} We would now develop the estimate for $\Gamma(K_{\max,i}, r_1, r_2, k_1, k_2)$ in the sparse case. We recall that the sparse case happens when each $\phi(K)$ intersects with at most two $R$'s for $(R,K) \in \Lambda(K_{\max,i}, r_1, r_2, k_1, k_2)$. From Corollary \ref{cor:r1}, we know that
\begin{equation} \label{assumption_r_1}
\max(2^{k_1}, 2^{k_2}) \leq \frac{1}{2} 2^{r_1},
\end{equation}
so we can quantify the sparseness by
\begin{equation} \label{sparse_exp}
2^{r_2} \geq 2^{k_2} |\cos \theta| + 2^{k_1} |\sin \theta|.
\end{equation}

Let $\Lambda^{sparse}$ denote $\Lambda(K_{\max,i}, r_1, r_2, k_1, k_2)$ with $(r_1, r_2,k_1,k_2)$ satisfying (\ref{assumption_r_1}) and (\ref{sparse_exp}). Correspondingly, let $\Gamma^{sparse}$ denote $\Gamma(K_{\max,i}, r_1, r_2, k_1, k_2)$ restricted on $(R,K) \in \Lambda^{sparse}$. 
In this sparse case, according to the estimates in Proposition \ref{prop:number},
$$
\mathcal{N}^v_{r_1,r_2}(\phi(K)) + \mathcal{N}^h_{r_1,r_2}(\phi(K)) \leq 10,$$
which by plugging in Lemma \ref{count_Delta_K}, yields that
\begin{equation} \label{sparse_trivial}
\Gamma^{sparse}(K_{\max,i},r_1,r_2,k_1,k_2) \lesssim \left(\sup_{(R,K) \in \Lambda} \frac{\min(K_1^2 |\tan \theta|, K_2^2 |\cot\theta|)}{|K|^{1/2} |R|^{1/2}}\right)^2 \cdot \frac{|K_{\max,i}|}{2^{k_1}2^{k_2}}.
\end{equation}
However, what is not taken into consideration in the estimate (\ref{sparse_trivial}) is that there are considerably many $K \subseteq K_{\max,i}$ such that $(R,K) \in \Lambda^{sparse}$ and $\phi(K) \subseteq R$ and thus 
$$
\phi(K) \cap (l_1^h(R) \cup l_2^h(R) \cup l_3^h(R) \cup l_1^v(R) \cup l_2^v(R) \cup l_3^v(R))  = \emptyset
$$
 in which case
$$
\langle h_R \circ \phi, h_K \rangle = 0.
$$
Let $\Lambda^{sparse, nontrivial} \subseteq \Lambda^{sparse}$ denote the sub-collection with $(R,K) \in \Lambda^{sparse}$ satisfying
$$
\phi(K) \cap (l_1^h(R) \cup l_2^h(R) \cup l_3^h(R) \cup l_1^v(R) \cup l_2^v(R) \cup l_3^v(R))  \neq \emptyset.
$$
We can then rewrite 
$$
\Gamma^{sparse}(K_{\max,i},r_1,r_2,k_1,k_2) =  \sum_{\substack{K \subseteq K_{\max,i} \\ |K| = 2^{k_1} \times 2^{k_2}}} \left(\sum_{\substack{R: (R,K) \in \Lambda^{sparse, nontrivial}}} |\langle h_R \circ \phi, h_K \rangle|^{p'} \right)^{\frac{2}{p'}}. 
$$
We now consider the inner sum of $\Gamma^{sparse}$. In particular
\begin{align*}
&\left(\sum_{\substack{(R,K) \in  \Lambda^{sparse, nontrivial}}} |\langle h_R \circ \phi, h_K \rangle|^{p'} \right)^{\frac{2}{p'}}.
\end{align*} 
Fix $K \subseteq K_{\max,i}$, there are at most two $R's$ intersecting with $\phi(K)$ where $(R,K) \in \Lambda^{sparse, nontrivial}$. Due to its almost unique intersections with $R$, $K$ can be associated to a unique $R$ for $(R, K) \in \Lambda^{sparse, nontrivial}$ and thus can be denoted by $K(R)$. The inner sum can be rewritten as
\begin{align*}
& \left(2 \cdot  \sup_{(R,K) \in \Lambda^{sparse, nontrivial}} |\langle h_R \circ \phi, h_K \rangle|^{p'} \right)^{\frac{2}{p'}}
\end{align*}
and
\begin{align} \label{Gamma_sparse}
\Gamma^{sparse}(K_{\max,i}, r_1, r_2, k_1, k_2) \lesssim \sum_{\substack{R \in \mathbb{D}^{\beta(\phi(K_{\max,i}))} \\R \nsubseteq \widetilde{\phi(\Omega)} \\ R \cap \phi(K_{\max,i}) \neq \emptyset \\ |R| = 2^{r_1} \times 2^{r_2}}}\sum_{\substack{K: (R,K) \in \Lambda^{sparse, nontrivial} \\ K= K(R)}} \sup_{(R,K) \in \Lambda^{sparse, nontrivial}} |\langle h_R \circ \phi, h_K \rangle|^{2}
\end{align}
for $(r_1, r_2,k_1,k_2)$ satisfying (\ref{assumption_r_1}) and (\ref{sparse_exp}). In the sparse case, the necessary counting estimates are \begin{enumerate}
\item
for $R$ fixed, how many rotated $K$'s intersect with $R$ where $(R,K) \in \Lambda^{sparse, nontrivial}$;
\item
for $K_{\max,i}$ fixed, how many $R$'s intersect with $\phi(K_{\max,i})$.
\end{enumerate}
By the assumption (\ref{assumption_r_1}) and the delicate choice of the sparse grid $\mathbb{D}^{\beta(\phi(K_{\max,i}))}$, we deduce that the vertical boundary of $R$'s never intersect with $K_{\max,i}$ and thus any $K\subseteq K_{\max,i}$. The intersections of $\phi(K)$ and $\phi(K_{\max,i})$ with the boundaries of $R$ can be simplified to the intersections with only the horizontal segments of $R$. We define the counting number as follows. 
\begin{definition} \label{def_count_sparse}
Fix a rotation map $\phi:= \phi^{\theta}$ of angle $\theta \in [0,2\pi]$.
\begin{itemize}
\item
Define for each $R \in \mathbb{D}^{\beta(\phi(K_{\max,i}))}$, the number of dyadic rectangles with fixed dimension that are contained in $K_{\max,i}$ and have non-trivial intersections with the rotated horizontal (or vertical) segments of $R$ as follows:
$$
\mathcal{L}^h_{k_1,k_2}(\phi^{-1}(R)) := \# \{K: (R,K) \in  \Lambda(K_{\max,i}, r_1, r_2, k_1, k_2), K \cap \phi^{-1}\big(l^h_1(R)\cup l^h_2(R)\cup l^h_3(R)\big)\neq\emptyset \}.
$$
\item
Define for each $K_{\max,i} \in \mathbb{D}^{\alpha}$, the number of dyadic rectangles of fixed dimension whose horizontal (or vertical) segments intersect non-trivially with $\phi(K_{\max,i})$ as follows:
$$
\mathcal{M}^h_{r_1, r_2}(\phi (K_{\max,i})) := \# \{R: R \in \mathbb{D}^{\beta(\phi(K_{\max,i}))}, |R| = 2^{r_1} \times 2^{r_2}, \phi(K_{\max,i}) \cap \big(l^h_1(R)\cup l^h_2(R)\cup l^h_3(R)\big)\neq\emptyset\}.
$$
\end{itemize}
\end{definition}
The following lemma summarizes the estimate of $\Gamma^{sparse}$ in (\ref{Gamma_sparse}) with the application of the counting numbers defined above. 
\begin{lemma} \label{lemma_sparse}
\begin{equation}
\Gamma^{sparse} \lesssim \sup_{(R,K) \in \Lambda^{sparse, nontrivial}} |\langle h_R \circ \phi, h_K \rangle|^{2} \cdot \sup_{|R| = 2^{r_1} \times 2^{r_2}}\mathcal{L}^h_{k_1,k_2}(\phi^{-1}(R)) \cdot \mathcal{M}^h_{r_1, r_2}(\phi (K_{\max,i})).
\end{equation}
\end{lemma}
While $\displaystyle \sup_{(R,K) \in \Lambda^{sparse, nontrivial}} |\langle h_R \circ \phi, h_K \rangle|^{2}$ can be majorzied by the bounds in Proposition \ref{prop:1}, we would need precise estimates for $\mathcal{L}^h_{k_1,k_2}(\phi^{-1}(R))$  and $\mathcal{M}^h_{r_1, r_2}(\phi (K_{\max,i})) $.
\begin{proposition} \label{prop_sparse_count}
$\mathcal{L}^h_{k_1,k_2}(\phi^{-1}(R))$  and $\mathcal{M}^h_{r_1, r_2}(\phi (K_{\max,i})) $ (Definition \ref{def_count_sparse}) satisfy the following estimates:
\begin{align*}
\sup_{|R| = 2^{r_1} \times 2^{r_2}}\mathcal{L}^h_{k_1,k_2}(\phi^{-1}(R)) \leq & \mathbbm{1}_{|\tan \theta| \leq \frac{K_{\max,i}^2}{K_{\max,i}^1}}\max \left(\frac{K_{\max,i}^1}{2^{k_1}}, \frac{K_{\max,i}^1 |\tan \theta|}{2^{k_2}} \right) + \\
& \mathbbm{1}_{|\tan \theta| \geq \frac{K_{\max,i}^2}{K_{\max,i}^1}}\max \left(\frac{K_{\max,i}^2 |\cot \theta|}{2^{k_1}}, \frac{K_{\max,i}^2}{2^{k_2}} \right),
\end{align*}
and
\begin{equation*}
\mathcal{M}^h_{r_1, r_2}(\phi (K_{\max,i})) \leq \max \left(\frac{K_{\max,i}^2 |\cos \theta| + K_{\max,i}^1 |\sin \theta|}{2^{r_2}},1 \right) \leq 4 \cdot \max \left(\frac{K_{\max,i}^2 |\cos \theta| + K_{\max,i}^1 |\sin \theta|}{2^{r_2}},\frac{1}{4} \right).
\end{equation*}
\end{proposition}

\begin{remark} \label{rem:important}
We notice that if $$\frac{K_{\max,i}^2 |\cos \theta| + K_{\max,i}^1 |\sin \theta|}{2^{r_2}} \leq \frac{1}{4},$$ 
or equivalently 
$$K_{\max,i}^2 |\cos \theta| + K_{\max,i}^1 |\sin \theta| \leq 
\frac{1}{4} \cdot 2^{r_2},$$
then thanks to the choice of the dyadic grid $\mathbb{D}^{\beta(\phi(K_{\max,i}))}$, we have chosen a dyadic rectangle such that
$$
\phi(K_{\max,i}) \subseteq Q
$$
and (\ref{choice_grid}) is satisfied, which implies that
\begin{align*}
& Q_2 \leq \frac{1}{2} \cdot 2^{r_2}.
\end{align*}
Moreover, under the assumptions and Corollary \ref{cor:r1}, we also have
\begin{align*}
& Q_1 \leq \frac{1}{2} \cdot 2^{r_1}.
\end{align*}
Since $Q$ and $R$ are dyadic rectangles in the chosen dyadic grid $\mathbb{D}^{\beta(\phi(K_{\max,i}))}$, we can deduce that $Q$ is included in one of the $4$ childrens of $R$. As a consequence, for any $K \subseteq K_{\max,i}$, $h_K \circ \phi^{-1}$ is constant on $R$, which implies that
\begin{equation}
\langle h_R \circ \phi, h_K \rangle = \langle h_R, h_K \circ \phi^{-1} \rangle = 0.
\end{equation}
Thus we would separate the discussion by the cases when 
$$
K_{\max,i}^2 |\cos \theta| + K_{\max,i}^1 |\sin \theta| \leq 
\frac{1}{4} \cdot 2^{r_2},
$$
so that we have the perfect degeneracy for $\Gamma$ and the more non-trivial case
$$
K_{\max,i}^2 |\cos \theta| + K_{\max,i}^1 |\sin \theta| \geq 
\frac{1}{4} \cdot 2^{r_2}.$$
\end{remark}

\begin{remark}
Combining the estimates for $\mathcal{L}^h_{k_1,k_2}(\phi^{-1}(R))$ and $\mathcal{M}^h_{r_1, r_2}(\phi (K_{\max,i}))$ in Proposition \ref{prop_sparse_count}, we derive that thanks to the previous remark, in the interesting situations (i.e. when $K_{\max,i}^2 |\cos \theta| + K_{\max,i}^1 |\sin \theta| \geq 
\frac{1}{4} \cdot 2^{r_2}$) then
\begin{equation}\label{sparse_count_simplified}
\sup_{|R| = 2^{r_1} \times 2^{r_2}}\mathcal{L}^h_{k_1,k_2}(\phi^{-1}(R)) \cdot \mathcal{M}^h_{r_1, r_2}(\phi (K_{\max,i})) \lesssim |K_{\max,i}|\left(\mathbbm{1}_{|\tan \theta| \geq \frac{2^{k_2}}{2^{k_1}}} \cdot \frac{|\sin \theta|}{2^{r_2}2^{k_2}} + \mathbbm{1}_{|\tan \theta| \leq \frac{2^{k_2}}{2^{k_1}}} \cdot \frac{|\cos \theta|}{2^{r_2}2^{k_1}}  \right).
\end{equation}
We remark that given the sparse condition 
$$
2^{r_2} \geq 2^{k_2} |\cos \theta| + 2^{k_1} |\sin \theta|,
$$
(\ref{sparse_count_simplified}) can be trivially bounded by
$$
\sup_{|R| = 2^{r_1} \times 2^{r_2}}\mathcal{L}^h_{k_1,k_2}(\phi^{-1}(R)) \cdot \mathcal{M}^h_{r_1, r_2}(\phi (K_{\max,i})) \lesssim \frac{|K_{\max,i}|}{2^{k_1}2^{k_2}},
$$
which agrees with the more naive estimate (\ref{sparse_trivial}) and indicates that (\ref{sparse_count_simplified}) is a natural and more optimal bound. 
\end{remark}

\subsubsection{Statement of Cancellation Estimates}  \label{subsec:cancelation2}
We have established all the partial estimates and would now assemble them to state the cancellation estimates in the format of a table. We list the different cases and clarify the lemma used in each case. According to the corresponding lemma, we specify the useful quantities and their estimates.

 \begin{table}[!h]
        \centering
        
\begin{tabular}{|p{0.08\textwidth}|p{0.03\textwidth}|p{0.32\textwidth}|p{0.19\textwidth}|p{0.19\textwidth}|p{0.19\textwidth}|}
\hline 
\multicolumn{6}{|l|}{$\displaystyle \textcolor[rgb]{0.25,0.46,0.02}{ K_{\max,i}^2 |\cos \theta| + K_{\max,i}^1 |\sin \theta| \leq  \cdot 2^{r_2-2}}$ \textbf{[Remark \ref{rem:important}]} } \\
\hline 
 \multicolumn{2}{|l|}{$\displaystyle \Gamma = $} & \multicolumn{4}{l|}{0} \\
\hline 
 \multicolumn{6}{|l|}{$ \textcolor[rgb]{0.25,0.46,0.02}{\displaystyle K_{\max,i}^2 |\cos \theta| + K_{\max,i}^1 |\sin \theta| \geq \cdot 2^{r_2-2}}$} \\
\hline 
 \multirow{13}{*}{Even} & \multicolumn{5}{l|}{\textcolor[rgb]{0.95,0.39,0.07}{$\displaystyle r_{2} \ \leq \min( k_{1} ,k_{2})$ }} \\
\cline{2-6} 
   & \multicolumn{5}{l|}{{$\displaystyle r_{2} \leq \min(k_{1},k_2) \leq \max(k_1, k_{2}) \leq r_{1}$
\textbf{[Lemma \ref{count_Delta_R}]}}} \\
\cline{2-6} 
   & \multirow{5}{*}{$\displaystyle \Gamma ^{0}$} &  & $\displaystyle 2^{k_{1}} |\sin \theta|\geq 2^{r_{2}}$ and $\displaystyle 2^{k_{2}} |\cos \theta|\geq 2^{r_{2}}$ & $\displaystyle 2^{k_{1}} |\sin \theta|\leq 2^{r_{2}}$ and
$\displaystyle 2^{k_{2}} |\cos \theta|\geq 2^{r_{2}}$ & $\displaystyle 2^{k_{1}} |\sin \theta |\geq 2^{r_{2}}$ and
$\displaystyle 2^{k_{2}} |\cos \theta |\leq 2^{r_{2}}$ \\
\cline{3-6} 
   &   & $\displaystyle \min\left( 2^{2r_{2}} |\tan \theta |,2^{2r_{1}} |\cot \theta |\right) \leq $ & $\displaystyle 2^{2r_{2}} |\tan \theta |$  & $\displaystyle 2^{2r_{2}} |\tan \theta |$ & $\displaystyle 2^{2r_{2}} |\tan \theta |$ \\
\cline{3-6} 
   &   & $\displaystyle \min\left( 2^{2r_{1}} |\tan \theta |,2^{2r_{2}} |\cot \theta |\right) \leq $ & $\displaystyle 2^{2r_{2}} |\cot \theta |$ & $\displaystyle 2^{2r_{2}} |\cot \theta |$ & $\displaystyle 2^{2r_{2}} |\cot \theta |$ \\
\cline{3-6} 
   &   & $\displaystyle \mathcal{N}^{v}_{r_{1} ,r_{2}}( \phi ( K)) - \#\text{top}(\phi(K)) \lesssim $ & $ \displaystyle \frac{2^{k_2}|\cos \theta|}{2^{r_2}}$  &$ \displaystyle \frac{2^{k_2}|\cos \theta|}{2^{r_2}}$  & $0$ \\
\cline{3-6} 
   &   & $\displaystyle \mathcal{N}^{h}_{r_{1} ,r_{2}}( \phi ( K)) -  \#\text{top}(\phi(K)) \lesssim $ & $ \displaystyle \frac{2^{k_1}|\sin \theta|}{2^{r_2}}$ & $0$ & $ \displaystyle \frac{2^{k_1}|\sin \theta|}{2^{r_2}}$  \\
\cline{2-6} 
   & \multicolumn{2}{l|}{$\displaystyle \Gamma ^{1} \lesssim $} & $\Gamma^0$ & $|K_{\max,i}| \cdot 2^{-r_1}2^{-r_2}$ & $|K_{\max,i}| \cdot 2^{-r_1}2^{-r_2}$  \\
\cline{2-6} 
   & \multicolumn{5}{l|}{\textcolor[rgb]{0.95,0.39,0.07}{$\displaystyle r_{2} \geq \min( k_{1} ,k_{2})$}} \\
\cline{2-6} 
   & \multicolumn{5}{l|}{{\textcolor{blue}{$\displaystyle 2^{k_{1}} |\sin \theta |+2^{k_{2}} |\cos \theta | >2^{r_{2}}$}
\textbf{[Lemma \ref{count_Delta_K}]}}} \\
\cline{2-6} 
   & \multicolumn{2}{l|}{} & $\displaystyle |\tan \theta |\leq \frac{2^{k_{2}}}{2^{k_{1}}}$ & \multicolumn{2}{l|}{$\displaystyle |\tan \theta | >\frac{2^{k_{2}}}{2^{k_{1}}}$} \\
\cline{2-6} 
   & \multicolumn{2}{l|}{$\displaystyle \min\left( 2^{2k_{1}} |\tan \theta |,2^{2k_{2}} |\cot \theta |\right) \leq $} & $\displaystyle 2^{2k_{1}} |\tan \theta |$ & \multicolumn{2}{l|}{$\displaystyle 2^{2k_{2}} |\cot \theta |$} \\
\cline{2-6} 
   & \multicolumn{2}{l|}{$\displaystyle \mathcal{N}^{v}_{r_{1} ,r_{2}}( \phi ( K))$+$\displaystyle \mathcal{N}^{h}_{r_{1} ,r_{2}}( \phi ( K)) \lesssim $} & $\displaystyle \frac{2^{k_{2}} |\cos \theta |}{2^{r_{2}}}$ & \multicolumn{2}{l|}{$\displaystyle \frac{2^{k_{1}} |\sin \theta |}{2^{r_{2}}}$} \\
\hline 
 \multirow{4}{*}{Sparse} & \multicolumn{5}{l|}{{\textcolor{blue}{$\displaystyle 2^{k_{1}} |\sin \theta |+2^{k_{2}} |\cos \theta |\leq 2^{r_{2}}$}
\textbf{[Lemma \ref{lemma_sparse}]}}} \\
\cline{2-6} 
   & \multicolumn{2}{l|}{} & $\displaystyle |\tan \theta |\leq \frac{2^{k_{2}}}{2^{k_{1}}}$ & \multicolumn{2}{l|}{$\displaystyle |\tan \theta | >\frac{2^{k_{2}}}{2^{k_{1}}}$} \\
\cline{2-6} 
   & \multicolumn{2}{l|}{$\displaystyle \min\left( 2^{2k_{1}} |\tan \theta |,2^{2k_{2}} |\cot \theta |\right) \leq $} & $\displaystyle 2^{2k_{1}} |\tan \theta |$ & \multicolumn{2}{l|}{$\displaystyle 2^{2k_{2}} |\cot \theta |$} \\
\cline{2-6} 
   & \multicolumn{2}{l|}{$\displaystyle \mathcal{L}^{h}_{k_{1} ,k_{2}}\left( \phi ^{-1}( R)\right)\mathcal{M}^{h}_{r_{1} ,r_{2}}( \phi ( K_{\max,i})) \lesssim$} & $\displaystyle |K_{\max,i}|\frac{|\cos \theta |}{2^{r_{2}} 2^{k_{1}}}$ & \multicolumn{2}{l|}{$\displaystyle |K_{\max,i}|\frac{|\sin \theta |}{2^{r_{2}} 2^{k_{2}}}$} \\
 \hline
\end{tabular}

\medskip

\caption{Cancellation Estimates Table}\label{cancellation_est_table}
        
        \end{table}
\begin{remark}
One can verify that on the 7th row where 
\begin{equation*}
 \min\left( 2^{2r_{2}} |\tan \theta |,2^{2r_{1}} |\cot \theta |\right) \leq \displaystyle 2^{2r_{2}} |\tan \theta |
\end{equation*}
is indeed 
\begin{equation*}
 \min\left( 2^{2r_{2}} |\tan \theta |,2^{2r_{1}} |\cot \theta |\right) = \displaystyle 2^{2r_{2}} |\tan \theta |.
\end{equation*}
In the other case when $2^{2r_{2}} |\tan \theta | \geq 2^{2r_{1}} |\cot \theta |$, we would have the configuration in Figure \ref{Est_box_min}.
Under the assumption that there are no tops of $\phi(K)$ intersecting $R$ (since it is about $\Gamma^0$), we could only be in case (b), which contradicts the condition that $2^{k_2} \leq \frac{1}{4}2^{r_1}$ as highlighted in (\ref{eqcor:r1}).
\end{remark}

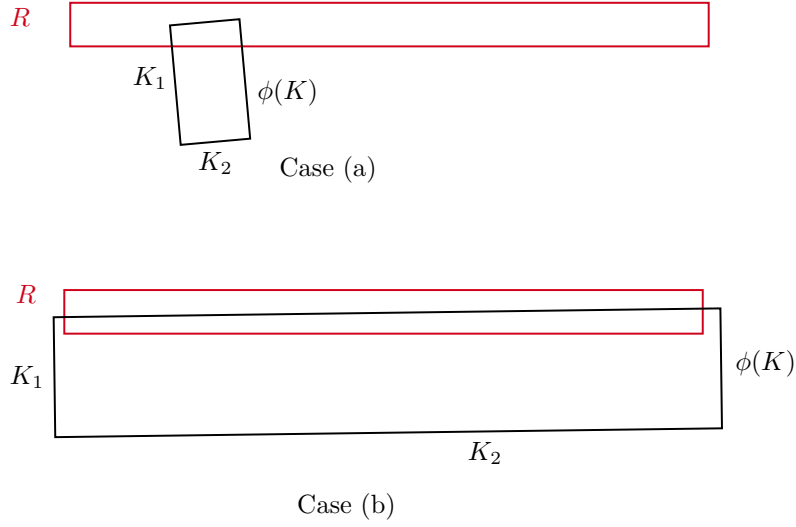
\begin{figure}
\centering

\tikzset{every picture/.style={line width=0.75pt}} 

\begin{tikzpicture}[x=0.75pt,y=0.75pt,yscale=-1,xscale=1]

\draw  [color={rgb, 255:red, 208; green, 2; blue, 27 }  ,draw opacity=1 ] (152,1524.5) -- (474,1524.5) -- (474,1546.5) -- (152,1546.5) -- cycle ;
\draw   (482.9,1533.75) -- (483.7,1594.28) -- (147.44,1598.74) -- (146.64,1538.21) -- cycle ;
\draw  [color={rgb, 255:red, 208; green, 2; blue, 27 }  ,draw opacity=1 ] (155,1379.5) -- (477,1379.5) -- (477,1401.5) -- (155,1401.5) -- cycle ;
\draw   (240.35,1387.79) -- (245.74,1448.08) -- (210.7,1451.21) -- (205.31,1390.92) -- cycle ;

\draw (123,1380.5) node [anchor=north west][inner sep=0.75pt]  [color={rgb, 255:red, 208; green, 2; blue, 27 }  ,opacity=1 ]  {$R$};
\draw (126,1520.5) node [anchor=north west][inner sep=0.75pt]  [color={rgb, 255:red, 208; green, 2; blue, 27 }  ,opacity=1 ]  {$R$};
\draw (248,1416.5) node [anchor=north west][inner sep=0.75pt]    {$\phi ( K)$};
\draw (489,1553.5) node [anchor=north west][inner sep=0.75pt]    {$\phi ( K)$};
\draw (259,1455.5) node [anchor=north west][inner sep=0.75pt]   [align=left] {Case (a)};
\draw (268,1625.5) node [anchor=north west][inner sep=0.75pt]   [align=left] {Case (b)};
\draw (185,1410.5) node [anchor=north west][inner sep=0.75pt]    {$K_{1}$};
\draw (123,1561.5) node [anchor=north west][inner sep=0.75pt]    {$K_{1}$};
\draw (218,1453.5) node [anchor=north west][inner sep=0.75pt]    {$K_{2}$};
\draw (354,1599.5) node [anchor=north west][inner sep=0.75pt]    {$K_{2}$};

\end{tikzpicture}
\caption{Geometric reasoning for minimum estimates in Cancellation Estimates Table} \label{Est_box_min}
\end{figure}

\subsubsection{Error Term Estimate: Application of Cancellation Lemma} \label{subsection:endproof} 
We will now demonstrate how to call the Cancellation Estimates Table to analyze the error term in different cases. We will exhaust all the cases in Table \ref{error_term_table} and categorize the types of arguments by different colors. There are 4 types of arguments, numbered as $\mathcal{I}, \mathcal{II}, \mathcal{III}$ and $\mathcal{IV}$ and indicated by green, blue, purple and red respectively. We will focus on one case for each type and all the other cases of the same type can be estimated using the exactly same argument. 

First, we recall that under symmetry arguments, we are working on the following conditions: 
\begin{itemize}
\item from \eqref{eq:r1}
\begin{equation*}
2^{-l}2^{r_1} \geq \frac{1}{2}\min(\frac{K_{\max}^1}{|\cos\theta|}, \frac{K_{\max}^2}{|\sin \theta|}) = \frac{1}{2}\frac{K_{\max}^1}{|\cos\theta|}. 
\end{equation*}
In particular, the considered $k_1$ parameter satisfies $2^{k_1} \lesssim 2^{-l} 2^{r_1}$.
\item and from \eqref{eqcor:r1}
\begin{equation*} 2^{k_1}+2^{k_2} \leq \frac{1}{4} 2^{r_1}. \end{equation*}
\end{itemize}

\begin{table}[!h]
\small
\begin{tabular}{|m{2cm}| m{1.5cm}| m{1.5cm} |m{1.5cm}| m{1.5cm}| m{1.5cm}| m{1.5cm}| m{1.5cm}| m{1.5cm}| m{1.5cm}| m{1.5cm}| m{1.5cm}| m{1.5cm}| m{1.5cm}| }
\hline
    & $r_2 \leq k_1 \leq k_2 \leq r_1$ & $r_2 \leq k_2 \leq k_1 \leq r_1$ & $k_1 \leq k_2 \leq r_2 \leq r_1$ & $k_2 \leq k_1 \leq r_2\leq r_1$ & $k_1 \leq k_2 \leq r_1 \leq r_2$ & $k_2 \leq k_1 \leq r_1 \leq r_2$ & $k_1 \leq r_2 \leq k_2 \leq r_1$ & $k_2 \leq r_2 \leq k_1 \leq r_1$ \\ \hline
$|\tan \theta| > \frac{2^{k_2}}{2^{k_1}}$ & {\cellcolor[HTML]{D2EED2}$\mathcal{I}$}                                              & \cellcolor[HTML]{D2EED2}                         & {\cellcolor[HTML]{D2ECEB}$\mathcal{II}$}                          & \cellcolor[HTML]{C5C6EE}                         &         \cellcolor[HTML]{D2ECEB}                                          &              \cellcolor[HTML]{C5C6EE}                                    & \cellcolor[HTML]{D2ECEB}                         &    \begin{tabular}{@{}l@{}}
                   {\cellcolor[HTML]{D2ECEB} $2^{k_1} |\sin \theta|$ } \\
                   {\cellcolor[HTML]{D2ECEB}  $ \leq 2^{r_2}\ \ \ \ $} \\ 
                  {\cellcolor[HTML]{FFCCC9}$2^{k_1} |\sin \theta| \ \ $}\\
                   {\cellcolor[HTML]{FFCCC9} $\geq 2^{r_2}\ \ \ $}\\
                 \end{tabular}                            \\ \cline{1-1} \cline{4-9} 
$|\tan \theta| \leq \frac{2^{k_2}}{2^{k_1}}$ & \multirow{-2}{*}{\cellcolor[HTML]{D2EED2}}                            & \multirow{-2}{*}{\cellcolor[HTML]{D2EED2}}       & {\cellcolor[HTML]{C5C6EE}  $\mathcal{III}$}                       & \cellcolor[HTML]{D2ECEB}                         &   \cellcolor[HTML]{C5C6EE}                                               &                 \cellcolor[HTML]{D2ECEB}                                  & \begin{tabular}{@{}l@{}}
                   {\cellcolor[HTML]{C5C6EE}$2^{k_2} |\cos \theta| \ \ \ $}\\
                    {\cellcolor[HTML]{C5C6EE}$\leq 2^{r_2}$}\\ 
                  {\cellcolor[HTML]{FFCCC9}$2^{k_2} |\cos \theta| \ \ $}\\
                   {\cellcolor[HTML]{FFCCC9}$\geq 2^{r_2}: \mathcal{IV}$}\\
                 \end{tabular}                         &       \cellcolor[HTML]{C5C6EE}                       \\ \hline
\end{tabular} 

\medskip

 \caption{Error Term Estimates} \label{error_term_table}
\end{table}
\begin{enumerate}
\item[Case $\mathcal{I}$: ]
$2^{r_2} \leq 2^{k_1} \leq 2^{k_2} \leq 2^{r_1}$
\begin{enumerate}
\item
$2^{k_2} |\cos \theta| \geq 2^{r_2}$ and $2^{k_1} |\sin \theta| \geq 2^{r_2}$
\\
According to the Cancellation Estimates Table, we apply Lemma \ref{count_Delta_R} where
\begin{align*}
\min\left( 2^{2r_{2}} |\tan \theta |,2^{2r_{1}} |\cot \theta |\right) &= 2^{2r_2}|\tan \theta| \\
\displaystyle \min\left( 2^{2r_{1}} |\tan \theta |,2^{2r_{2}} |\cot \theta |\right) &= \displaystyle 2^{2r_{2}} |\cot \theta | \\
\displaystyle \mathcal{N}^{v}_{r_{1} ,r_{2}}( \phi ( K)) - \#\text{top}(\phi(K)) & \lesssim \frac{2^{k_2}|\cos \theta|}{2^{r_2}} \\
 \displaystyle \mathcal{N}^{h}_{r_{1} ,r_{2}}( \phi ( K)) -  \#\text{top}(\phi(K)) & \lesssim  \frac{2^{k_1}|\sin \theta|}{2^{r_2}}.
\end{align*}
By plugging the above estimates into Lemma \ref{count_Delta_R}, we obtain
\begin{align}
& \sum_{r_2, k_1, k_2, r_1} 2^{2r_1(\frac{1}{2}+ \gamma_1)}2^{2r_2(\frac{1}{2}+ \gamma_2)} \Gamma(K_{\max,i}, r_1, r_2, k_1,k_2) \nonumber\\
\lesssim & \sum_{r_2, k_1, k_2, r_1}  \left(\frac{2^{4r_2}\tan^2 \theta}{2^{k_1+k_2}2^{r_1+r_2}} \cdot \left(\frac{2^{k_2}|\cos\theta|}{2^{r_2}}\right)^{\frac{2}{p'}} +  \frac{2^{4r_2}\cot^2 \theta}{2^{k_1+k_2}2^{r_1+r_2}} \cdot \left(\frac{2^{k_1} |\sin \theta|}{2^{r_2}}\right)^{\frac{2}{p'}}\right) \cdot \frac{|K_{\max,i}|}{2^{k_1+k_2}} \cdot 2^{2r_1(\frac{1}{2}+ \gamma_1)}2^{2r_2(\frac{1}{2}+ \gamma_2)} \nonumber\\
\leq &  |K_{\max,i}| \underbrace{\sum_{r_2, k_1, k_2, r_1} 2^{4r_2}\tan^2 \theta \cdot \left(\frac{2^{k_2}|\cos\theta|}{2^{r_2}}\right)^{\frac{2}{p'}} \cdot 2^{-2k_1-2k_2}2^{2\gamma_1r_1}2^{2\gamma_2r_2}}_{I_a^1} + \nonumber\\
& |K_{\max,i}| \underbrace{ \sum_{r_2, k_1, k_2, r_1} 2^{4r_2}\cot^2 \theta \cdot \left(\frac{2^{k_1} |\sin \theta|}{2^{r_2}}\right)^{\frac{2}{p'}} \cdot 2^{-2k_1-2k_2}2^{2\gamma_1r_1}2^{2\gamma_2r_2}}_{I_a^2}.
\end{align}
Now we will estimate both terms in details. We first consider
\begin{align}\label{I_a^1}
I_a^1 = & \sum_{r_2, k_1, k_2, r_1} 2^{4r_2}\tan^2 \theta \cdot \left(\frac{2^{k_2}|\cos\theta|}{2^{r_2}}\right)^{\frac{2}{p'}} \cdot 2^{-2k_1-2k_2}2^{2\gamma_1r_1}2^{2\gamma_2r_2} \nonumber\\
= & \sum_{r_2, k_1, k_2, r_1} \sin^2 \theta |\cos \theta|^{\frac{2}{p'} - 2} 2^{r_2(2\gamma_2 +4 - \frac{2}{p'})}2^{-2k_1}2^{k_2(\frac{2}{p'} - 2)} 2^{2\gamma_1r_1} \nonumber\\
\leq & \sum_{r_2, k_1, k_2, r_1} \left(\frac{2^{k_2}}{2^{r_2}}\right)^{2-\frac{2}{p'} } 2^{r_2(2\gamma_2 +4 - \frac{2}{p'})}2^{-2k_1}2^{k_2(\frac{2}{p'} - 2)} 2^{2\gamma_1r_1} 
\end{align}
where the last inequality follows from the condition $|\cos \theta| \geq \frac{2^{r_2}}{2^{k_2}}$. Therefore, (\ref{I_a^1}) can be majorized by 
\begin{equation} \label{geom_k_1}
\sum_{r_2, k_1, k_2, r_1} 2^{r_2(2\gamma_2+2)}2^{-2k_1}2^{2\gamma_1r_1} \lesssim   \sum_{k_1, k_2, r_1} 2^{2\gamma_2k_1}2^{2\gamma_1r_1}. 
\end{equation}
We recall that $2^{k_1} \lesssim 2^{-l}2^{r_1}$. Furthermore, $2^{k_1} \leq 2^{k_2}$ in this case. As a consequence, 
\begin{equation} \label{k_1_interpolate}
2^{k_1} \lesssim \min(2^{-l}2^{r_1}, 2^{k_2}) \leq (2^{-l}2^{r_1})^{\theta} 2^{k_2(1-\theta)}
\end{equation}
for any $0 \leq \theta \leq 1.$ By applying the condition (\ref{k_1_interpolate}) to the geometric sum (\ref{geom_k_1}) and choosing $\gamma_2 >0$, we obtain
\begin{equation*}
2^{-l \cdot 2\gamma_2\theta } \sum_{k_2, r_1} 2^{k_2 \cdot 2\gamma_2(1-\theta) }2^{r_1(2\gamma_1+\ 2 \gamma_2\theta) } \lesssim  2^{-l \cdot 2\gamma_2\theta}\sum_{ r_1} 2^{r_1(2\gamma_1+2\gamma_2)}.
\end{equation*}
Similarly, we can estimate $I_a^2$ by
\begin{equation}\label{I_a^2}
\sum_{r_2, k_1, k_2, r_1} |\sin \theta|^{\frac{2}{p'}-2} 2^{4r_2} \cdot \left(\frac{2^{k_1}}{2^{r_2}}\right)^{\frac{2}{p'}} \cdot 2^{-2k_1-2k_2}2^{2\gamma_1r_1}2^{2\gamma_2r_2} \leq  \sum_{r_2, k_1, k_2, r_1} \left(\frac{2^{k_1}}{2^{r_2}}\right)^{2-\frac{2}{p'} } 2^{r_2(2\gamma_2 +4 - \frac{2}{p'})}2^{k_1(\frac{2}{p'} -2)}2^ {- 2k_2} 2^{2\gamma_1r_1} 
\end{equation}
since $|\sin\theta| \geq \frac{2^{r_2}}{2^{k_1}}$. Simplification of (\ref{I_a^2}) yields
\begin{equation*}
\sum_{r_2, k_1, k_2, r_1} 2^{r_2(2\gamma_2+2)}2^{-2k_2}2^{2\gamma_1r_1} \lesssim   \sum_{k_1, k_2, r_1} 2^{k_1(2\gamma_2+2)}2^{-2k_2}2^{2\gamma_1r_1} \lesssim  \sum_{k_1, r_1} 2^{2\gamma_2k_1}2^{2\gamma_1r_1} 
\lesssim   2^{-2\gamma_2l}\sum_{ r_1} 2^{r_1(2\gamma_1+2\gamma_2)}
\end{equation*}
where the last inequality holds for any $\gamma_2 > 0$. One can then choose appropriate $\gamma_1$ such that both series (\ref{I_a^1}) and (\ref{I_a^2}) converge. In conclusion, for $\mu := \min(2\gamma_2 \theta, 2\gamma_2) = 2\gamma_2 \theta$ such that $0 < \gamma_2 < \min(s-\frac{2}{p}, \frac{1}{p}) $ and $0 \leq \theta \leq 1$,
\begin{align} \label{I_final}
\sum_{r_2, k_1, k_2, r_1} 2^{2r_1(\frac{1}{2}+ \gamma_1)}2^{2r_2(\frac{1}{2}+ \gamma_2)} \Gamma(K_{\max,i}, r_1, r_2, k_1,k_2) \lesssim 2^{-l\mu}|K_{\max,i}|
\end{align}
as desired.
\item
$2^{k_2} |\cos \theta| \geq 2^{r_2}$ and $2^{k_1} |\sin \theta| < 2^{r_2}$ \\
We refer to the Cancellation Estimates Table and apply Lemma \ref{count_Delta_R} where
\begin{align*}
&\min\left( 2^{2r_{2}} |\tan \theta |,2^{2r_{1}} |\cot \theta |\right) = 2^{2r_2}|\tan \theta| \\
&\displaystyle \min\left( 2^{2r_{1}} |\tan \theta |,2^{2r_{2}} |\cot \theta |\right) = \displaystyle 2^{2r_{2}} |\cot \theta | \\
&\displaystyle \mathcal{N}^{v}_{r_{1} ,r_{2}}( \phi ( K)) - \#\text{top}(\phi(K))  \lesssim \frac{2^{k_2}|\cos \theta|}{2^{r_2}} \\
& \displaystyle \mathcal{N}^{h}_{r_{1} ,r_{2}}( \phi ( K)) -  \#\text{top}(\phi(K))  = 0 \\
& \Gamma^1  \lesssim |K_{\max,i}|2^{-r_1}2^{-r_2}.
\end{align*}
By plugging the above estimates into Lemma \ref{count_Delta_R}, we have
\begin{align} \label{II}
&\sum_{r_2,k_1,k_2,r_1}  2^{2r_1(\frac{1}{2}+ \gamma_1)}2^{2r_2(\frac{1}{2}+ \gamma_2)}\Gamma(K_{\max,i}, r_1, r_2, k_1, k_2)\nonumber \\
\lesssim & |K_{\max,i}|\sum_{r_2,k_1,k_2,r_1} 2^{4r_2}\tan^2 \theta \cdot \left(\frac{2^{k_2}|\cos\theta|}{2^{r_2}}\right)^{\frac{2}{p'}} 2^{-2k_1-2k_2} 2^{2\gamma_1r_1}2^{2\gamma_2r_2}  +  |K_{\max,i}| \underbrace{\sum_{r_2,k_1,k_2,r_1}2^{2\gamma_1r_1}2^{2\gamma_2r_2}}_{I_b}
\end{align}
where the first summand is the same as $I_a^1$ and the second term can be estimated as follows:
\begin{equation*}
I_b \lesssim \sum_{r_2,k_1,k_2,r_1}2^{2\gamma_2 r_2} 2^{2\gamma_1r_1} \lesssim  \sum_{k_1, k_2,,r_1}2^{2\gamma_2k_1}2^{2\gamma_1r_1}.
\end{equation*}
We notice that we have deduced the same expression as (\ref{geom_k_1}) so that we can apply the same argument to conclude
$$
(\ref{II}) \lesssim 2^{-l\mu}|K_{\max,i}|.
$$
for $\mu:= 2\gamma_2 \theta$ where $\gamma_2$ and $\theta$ satisfy the conditions $0 < \gamma_2 < \min(s-\frac{2}{p}, \frac{1}{p}) $ and $0 \leq \theta \leq 1$ as in Case $\mathcal{I}$(a). 
\item
$2^{k_2}|\cos \theta| < 2^{r_2}$ and $2^{k_1} |\sin \theta| \geq 2^{r_2}$ \\
As indicated by the Cancellation Estimates Table, we apply Lemma \ref{count_Delta_R} where
\begin{align*}
&\min\left( 2^{2r_{2}} |\tan \theta |,2^{2r_{1}} |\cot \theta |\right) = 2^{2r_2}|\tan \theta| \\
&\displaystyle \min\left( 2^{2r_{1}} |\tan \theta |,2^{2r_{2}} |\cot \theta |\right) = \displaystyle 2^{2r_{2}} |\cot \theta | \\
&\displaystyle \mathcal{N}^{v}_{r_{1} ,r_{2}}( \phi ( K)) - \#\text{top}(\phi(K))  = 0  \\
& \displaystyle \mathcal{N}^{h}_{r_{1} ,r_{2}}( \phi ( K)) -  \#\text{top}(\phi(K)) \lesssim \frac{2^{k_1}|\sin \theta|}{2^{r_2}} \\
& \Gamma^1  \lesssim |K_{\max,i}|2^{-r_1}2^{-r_2},
\end{align*}
which yields the following estimate:
\begin{align}
&\sum_{r_2,k_1,k_2,r_1}  2^{2r_1(\frac{1}{2}+ \gamma_1)}2^{2r_2(\frac{1}{2}+ \gamma_2)}\Gamma(K_{\max,i}, r_1, r_2, k_1, k_2)\nonumber \\
\lesssim &  |K_{\max,i}|\sum_{r_2,k_1,k_2,r_1} 2^{2\gamma_1r_1}2^{2\gamma_2 r_2} +|K_{\max,i}| \sum_{r_2,k_1,k_2,r_1}2^{4r_2}\cot^2 \theta \cdot \left(\frac{2^{k_1} |\sin \theta|}{2^{r_2}}\right)^{\frac{2}{p'}} \cdot 2^{-2k_1-2k_2}2^{2\gamma_1r_1}2^{2\gamma_2 r_2}.
\end{align}
The second term is the same as $I_a^2$ and the first term can be estimated similarly as $I_b$ so that (\ref{I_final}) would be achieved.
\end{enumerate}

\item[Case $\mathcal{II}$:]
$2^{k_1} \leq 2^{k_2} \leq 2^{r_2} \leq 2^{r_1}$ and $|\tan \theta| > \frac{2^{k_2}}{2^{k_1}}$ \\
This is the sparse case and the Cancellation Estimates Table suggests that we apply Lemma \ref{lemma_sparse} with
\begin{align*}
& \min\left( 2^{2k_{1}} |\tan \theta |,2^{2k_{2}} |\cot \theta |\right) = 2^{2k_{2}} |\cot \theta| \\
&  \mathcal{L}^{h}_{k_{1} ,k_{2}}\left( \phi ^{-1}( R)\right)\mathcal{M}^{h}_{r_{1} ,r_{2}}( \phi ( K_{\max,i})) \lesssim |K_{\max,i}|\frac{|\sin \theta |}{2^{r_{2}} 2^{k_{2}}},
\end{align*}
which generates the following estimates:
\begin{align} \label{II_a}
&\sum_{k_1,k_2,r_2, r_1}  2^{2r_1(\frac{1}{2}+ \gamma_1)}2^{2r_2(\frac{1}{2}+ \gamma_2)}\Gamma(K_{\max,i},r_1, r_2, k_1, k_2) \nonumber \\
\lesssim & \sum_{k_1,k_2, r_2,r_1}  \frac{(2^{2k_2}\cot \theta )^2}{2^{k_1+k_2}2^{r_1+r_2}} \cdot \frac{|K_{\max,i}||\sin \theta|}{2^{k_2}2^{r_2}}\cdot 2^{2r_1(\frac{1}{2}+ \gamma_1)}2^{2r_2(\frac{1}{2}+ \gamma_2)}\nonumber \\
\leq &  \sum_{k_1,k_2, r_2,r_1}  \frac{(2^{k_1}2^{k_2})^2}{2^{k_1+k_2}2^{r_1+r_2}} \cdot \frac{|K_{\max,i}||\sin \theta|}{2^{k_2}2^{r_2}}\cdot 2^{2r_1(\frac{1}{2}+ \gamma_1)}2^{2r_2(\frac{1}{2}+ \gamma_2)}\nonumber \\
\lesssim & |K_{\max,i}| \sum_{k_1,k_2, r_2,r_1} 2^{k_1}2^{r_2(2\gamma_2-1)}2^{2\gamma_1r_1} 
\end{align}
where the second to last inequality follows from the condition $|\tan \theta| > \frac{2^{k_2}}{2^{k_1}}$. 
We recall that $\gamma_2 \in (-\delta, \delta)$ with $\delta:= \min(s-\frac{2}{p}, \frac{1}{p}) \leq \frac{1}{p} < \frac{1}{2}$ since $p \in (2,\infty)$. Thus $2\gamma_2 -1 < 0$, which yields
\begin{equation}\label{II_a^1}
(\ref{II_a})  \lesssim |K_{\max,i}| \sum_{k_1,k_2,r_1} 2^{k_1}2^{k_2(2\gamma_2-1)}2^{2\gamma_1r_1} \lesssim  |K_{\max,i}| \sum_{k_1,r_1}2^{2\gamma_2k_1}2^{2\gamma_1r_1}.
\end{equation}
By choosing $\gamma_2 > 0$, we deduce that
\begin{equation} \label{II_a^1_final}
(\ref{II_a^1}) \lesssim 2^{-l \cdot 2\gamma_2}\sum_{r_1}2^{r_1(2\gamma_1 + \gamma_2)}
 \end{equation}
 One can finally choose $\gamma_1$ such that the series (\ref{II_a^1_final}) converges. As a result, one obtains
 \begin{equation} \label{II_desired}
 \sum_{k_1,k_2,r_2, r_1}  2^{2r_1(\frac{1}{2}+ \gamma_1)}2^{2r_2(\frac{1}{2}+ \gamma_2)}\Gamma(K_{\max,i},r_1, r_2, k_1, k_2) \lesssim 2^{-\mu l}|K_{\max,i}|
  \end{equation}
 for $\mu := 2\gamma_2$ with $0 < \gamma_2 < \min(s-\frac{2}{p}, \frac{1}{p})$.

 \item[Case $\mathcal{III}$:]
$2^{k_1} \leq 2^{k_2} \leq 2^{r_2} \leq 2^{r_1}$ and $|\tan \theta| \leq \frac{2^{k_2}}{2^{k_1}}$ \\
We are in the sparse case and need to apply Lemma \ref{lemma_sparse} with the following estimates indicated in the Cancellation Estimates Table:
\begin{align*}
& \min\left( 2^{2k_{1}} |\tan \theta |,2^{2k_{2}} |\cot \theta |\right) = 2^{2k_{1}} |\tan \theta| \\
&  \mathcal{L}^{h}_{k_{1} ,k_{2}}\left( \phi ^{-1}( R)\right)\mathcal{M}^{h}_{r_{1} ,r_{2}}( \phi ( K_{\max,i})) \lesssim \frac{|\cos \theta |}{2^{r_{2}} 2^{k_{1}}}.
\end{align*}
As a consequence, we have
 \begin{align}\label{III}
 &\sum_{k_1,k_2,r_2, r_1}  2^{2r_1(\frac{1}{2}+ \gamma_1)}2^{2r_2(\frac{1}{2}+ \gamma_2)}\sum_{\substack{K \subseteq K_{\max,i} \\ |K|= 2^{k_1} \times 2^{k_2}}}\left(\sum_{\substack{R \nsubseteq \widetilde{\phi(\Omega)} \\ |R| = 2^{r_1} \times 2^{r_2}}} |\langle h_{\tilde{R}_1}, h_K \rangle|^{p'} \right)^{\frac{2}{p'}} \nonumber \\
& \sum_{k_1,k_2,r_2, r_1} \frac{(2^{2k_1}\tan \theta )^2}{2^{k_1+k_2}2^{r_1+r_2}} \cdot \frac{|K_{\max,i}| |\cos \theta|}{2^{k_1}2^{r_2}} \cdot 2^{2r_1(\frac{1}{2}+ \gamma_1)}2^{2r_2(\frac{1}{2}+ \gamma_2)} \nonumber\\
= &  \sum_{k_1,k_2, r_2,r_1} |\tan \theta| |\sin \theta|  \frac{2^{4k_1}}{2^{k_1+k_2}2^{r_1+r_2}} \cdot \frac{|K_{\max,i}|}{2^{k_1}2^{r_2}}\cdot 2^{2r_1(\frac{1}{2}+ \gamma_1)}2^{2r_2(\frac{1}{2}+ \gamma_2)}\nonumber \\
\leq & \sum_{k_1,k_2, r_2,r_1} \frac{2^{k_2}}{2^{k_1}} \cdot \frac{2^{4k_1}}{2^{k_1+k_2}2^{r_1+r_2}} \cdot \frac{|K_{\max,i}|}{2^{k_1}2^{r_2}}\cdot 2^{2r_1(\frac{1}{2}+ \gamma_1)}2^{2r_2(\frac{1}{2}+ \gamma_2)}
\end{align}
where the last inequality follows from the condition $|\tan \theta| \leq \frac{2^{k_2}}{2^{k_1}}$. As a consequence, 
\begin{equation*}
(\ref{III}) =  |K_{\max,i}| \sum_{k_1, k_2, r_2,r_1} 2^{k_1}2^{r_2(2\gamma_2-1)}2^{2\gamma_1r_1} 
\end{equation*}
which agrees with the expression (\ref{II_a}) and can be estimated by the exactly same argument. 
Therefore, (\ref{II_desired}) can be attained with $\mu:= 2\gamma_2$ with $0 < \gamma_2 < \min(s-\frac{2}{p}, \frac{1}{p})$.
\item[Case $\mathcal{IV}$:]
$2^{k_1} \leq 2^{r_2} \leq 2^{k_2} \leq 2^{r_1}$ and $2^{k_2}|\cos \theta| \geq 2^{r_2}$ \\
In this case, we observe that a necessary condition imposed by $2^{k_1} \leq 2^{r_2}$ and $2^{k_2}|\cos \theta| \geq 2^{r_2}$ is 
$$
|\tan \theta|  \leq \frac{2^{k_2}}{2^{k_1}}
$$
because $2^{k_1}|\sin \theta| \leq 2^{r_2} \leq 2^{k_2}|\cos \theta|$. Hence we apply Lemma \ref{count_Delta_K} with 
\begin{align*}
& \min\left( 2^{2k_{1}} |\tan \theta |,2^{2k_{2}} |\cot \theta |\right) = 2^{2k_{1}} |\tan \theta |\\
&\mathcal{N}^{v}_{r_{1} ,r_{2}}( \phi ( K))+ \mathcal{N}^{h}_{r_{1} ,r_{2}}( \phi ( K)) \lesssim \frac{2^{k_{2}} |\cos \theta |}{2^{r_{2}}},
\end{align*}
which gives
 \begin{align} \label{IV}
&\sum_{k_1,r_2, k_2, r_1}  2^{2r_1(\frac{1}{2}+ \gamma_1)}2^{2r_2(\frac{1}{2}+ \gamma_2)}\Gamma(K_{\max,i}, r_1. r_2, k_1, k_2)\nonumber \\
\lesssim &\sum_{k_1,r_2, k_2, r_1}  2^{2r_1(\frac{1}{2}+ \gamma_1)}2^{2r_2(\frac{1}{2}+ \gamma_2)}  \cdot \frac{(2^{2k_1}\tan \theta )^2}{2^{k_1+k_2}2^{r_1+r_2}} \cdot \left(\frac{2^{k_2} |\cos \theta|}{2^{r_2}} \right)^{\frac{2}{p'}}\frac{|K_{\max,i}|}{2^{k_1+k_2}} \nonumber\\
= & |K_{\max,i}|\sum_{k_1,r_2, k_2, r_1} |\tan \theta|^{2-\frac{2}{p'}}|\sin\theta|^{\frac{2}{p'}}2^{2k_1}2^{r_2(2\gamma_2-\frac{2}{p'})}2^{k_2(\frac{2}{p'}-2)}2^{2\gamma_1r_1} \nonumber\\
\leq & |K_{\max,i}|\sum_{k_1,r_2, k_2, r_1} \left(\frac{2^{k_2}}{2^{k_1}}\right)^{2-\frac{2}{p'}}2^{2k_1}2^{r_2(2\gamma_2-\frac{2}{p'})}2^{k_2(\frac{2}{p'}-2)}2^{2\gamma_1r_1}
\end{align}
due to the condition $|\tan \theta|  \leq \frac{2^{k_2}}{2^{k_1}}$. After simplifications, we obtain
\begin{equation*}
(\ref{IV}) = |K_{\max,i}|\sum_{k_1,r_2, k_2, r_1} 2^{\frac{2}{p'}k_1}2^{r_2(2\gamma_2-\frac{2}{p'})}2^{2\gamma_1r_1}.
\end{equation*}
We notice that $\gamma_2 \in (-\delta, \delta)$ with $\delta:= \min(s-\frac{2}{p}, \frac{1}{p}) < \frac{1}{2} \leq \frac{1}{p'}$ since $\frac{1}{p} + \frac{1}{p'} =1$ and $p >2$. Therefore, $2\gamma_2 - \frac{2}{p'} < 0$, which implies
\begin{equation*}
(\ref{IV}) \lesssim |K_{\max,i}|\sum_{k_1, r_1} 2^{2\gamma_2 k_1}2^{2\gamma_1r_1}
\lesssim  |K_{\max,i}|2^{-2\gamma_2l}\sum_{ r_1} 2^{r_1(2\gamma_1+2\gamma_2)} 
 \end{equation*}
 where the last two inequalities follow from $\gamma_2 >0$. We finally select $\gamma_1$ so that the series converges and thus conclude that
\begin{equation}
 \sum_{k_1,k_2,r_2, r_1}  2^{2r_1(\frac{1}{2}+ \gamma_1)}2^{2r_2(\frac{1}{2}+ \gamma_2)}\Gamma(K_{\max,i},r_1, r_2, k_1, k_2) \lesssim 2^{-\mu l}|K_{\max,i}|
\end{equation}
with $\mu := 2 \gamma_2$ with $0 < \gamma_2 < \min(s- \frac{2}{p}, \frac{1}{p})$.
\end{enumerate}

To sum up, we have showed that for any $l \in \mathbb{N}$ with $2^{-l} \lesssim \epsilon^{\frac{1}{2}}$, \eqref{subpart} holds
with $\mu := 2\theta\gamma_2 $ where $0 < \gamma_2 < \min(s-\frac{2}{p}, \frac{1}{p}) $ and $0 \leq \theta \leq 1$. Then we have seen how to deduce (\ref{E_2_to_fill}), which implies
\begin{equation} \label{E_2_result}
\mathcal{E}_2 \lesssim \epsilon^{\frac{\mu}{4}}\|F\|_{W^{s,p}}|\Omega|^{\frac{1}{2}} = \epsilon^{\frac{\theta\gamma_2}{2}}\|F\|_{W^{s,p}}|\Omega|^{\frac{1}{2}}.
\end{equation}

We refer the reader to Figure \ref{outline_proof} and combine the estimates for $\mathcal{E}_1$, namely (\ref{error1}) and (\ref{error2}), $\mathcal{M}$ (\ref{main_term_result}) and $\mathcal{E}_2$ \eqref{E_2_result}, from which we can conclude that
\begin{equation*}
\sup_{\Omega \subseteq \mathbb{R}^2 \text{open}} |\Omega|^{-\frac{1}{2}} osc_{\Omega} (F \circ \phi) \lesssim \sup_{\Omega \subseteq \mathbb{R}^2 \text{open}} |\Omega|^{-\frac{1}{2}}(\mathcal{E}_1 + \mathcal{M} + \mathcal{E}_2) \lesssim \epsilon^{-1}\log\epsilon^{-1}\|F\|_{\BMO} + \epsilon^{\frac{\theta\gamma_2}{2}}\|F\|_{W^{s,p}}, 
\end{equation*}
which completes the proof of Theorem \ref{main_thm_pos}.

\end{document}